\documentclass[english]{article}
\usepackage[T1]{fontenc}
\usepackage[latin9]{inputenc}
\usepackage{babel}
\usepackage{units}
\usepackage{textcomp}
\usepackage{mathrsfs}
\usepackage{mathtools}
\usepackage{amsmath}
\usepackage{amsthm}
\usepackage{amssymb}
\usepackage{stackrel}
\usepackage{graphicx}
\usepackage[authoryear]{natbib}
\usepackage[unicode=true]
 {hyperref}

\makeatletter
\theoremstyle{plain}
\newtheorem{thm}{\protect\theoremname}[section]
\theoremstyle{plain}
\newtheorem{prop}[thm]{\protect\propositionname}
\theoremstyle{plain}
\newtheorem{lem}[thm]{\protect\lemmaname}

\usepackage{lineno}
\usepackage{amsthm}
\usepackage{thmtools}
\usepackage{thm-restate}

\allowdisplaybreaks

\newcommand{\medrightarrow}{\mathrel{\joinrel\to}}

\@ifundefined{showcaptionsetup}{}{%
 \PassOptionsToPackage{caption=false}{subfig}}
\usepackage{subfig}
\makeatother

\providecommand{\lemmaname}{Lemma}
\providecommand{\propositionname}{Proposition}
\providecommand{\theoremname}{Theorem}

\begin{document}
\title{Rank Distributions for Independent\\
Normals with a Single Outlier}
\author{Philip T.\ Labo\thanks{Author email: plabo@alumni.stanford.edu.}}
\maketitle
\begin{abstract}
Thurstone\textquoteright s latent-normal model, introduced a century
ago to describe human preferences in psychometrics (\citet{T27a,T27b,T27c}),
remains a cornerstone for modeling random rankings. Yet when the underlying
normals differ in distribution, the joint law of ranks $R_{i}\coloneqq\sum_{j=1}^{n}\mathbf{1}_{X_{j}\leq X_{i}}$
is virtually unexplored. We study the simplest non-identically-distributed
case: $n+1$ independent normals with $X_{0}\sim\mathcal{N}\left(\mu_{0},\,\sigma_{0}^{2}\right)$
and $X_{i}\sim\mathcal{N}\left(\mu,\,\sigma^{2}\right)$ for $1\leq i\leq n$.
Here, $\left(\left.R_{0}\right|X_{0}\right)\sim1+\mathrm{Binomial}\left(n,\Phi\left(\nicefrac{\left(X_{0}-\mu\right)}{\sigma}\right)\right)$,
and the success probability $\Phi\left(\nicefrac{\left(X_{0}-\mu\right)}{\sigma}\right)$
is accurately modeled by a beta distribution. Exploiting beta-binomial
conjugacy, we observe that $R_{0}-1$ follows a beta-binomial law,
which then yields a precise approximation for the joint distribution
of $\left(R_{0},R_{i_{1}},\ldots,R_{i_{m}}\right)$. We derive closed-form
expressions for $\mathbb{E}R_{i}$, $\mathrm{Cov}\left(R_{i},R_{j}\right)$,
and the limiting distributions of $\left(R_{0},R_{i_{1}},\ldots,R_{i_{m}}\right)$
as key parameters grow large or small.
\end{abstract}

\section{Introduction \label{sec:Introduction}}

Ranked lists permeate everyday life---from Google search results
and Facebook newsfeeds to supermarket checkout lines and university
rankings. A century ago, \citet{T27a,T27b,T27c} proposed modeling
individual preferences by treating the components of a multivariate
normal vector $\mathbf{X}\sim\mathcal{N}_{n}\left(\boldsymbol{\mu},\boldsymbol{\Sigma}\right)$
as latent utilities. Since then, numerous researchers (\emph{e.g.},
\citet{D50,M51,H81a,H81b,D86,LB94,YB99,Y00}) have explored many facets
of this normal approach.\footnote{Section \ref{subsec:Rank-Distributions} casts an even wider net,
reviewing ranking probabilities when latent utilities come from diverse
distributions---not just the normal.} Nevertheless, when these normal utilities are independent but not
identically distributed, the resulting rank-vector $\mathbf{R}\coloneqq\left(R_{1},R_{2},\ldots,R_{n}\right)$
remains poorly understood. In this paper, we investigate the simplest
non-i.i.d.\ scenario.

Assume that $X_{0}\sim\mathcal{N}\left(\mu_{0},\sigma_{0}^{2}\right)$
and, independently, $X_{1},\ldots,X_{n}\stackrel{\mathrm{iid}}{\sim}\mathcal{N}\left(\mu,\sigma^{2}\right)$.
Our goal is to determine the distribution of the rank variable 
\begin{equation}
R_{0}\coloneqq\sum_{i=0}^{n}\mathbf{1}_{\left\{ X_{i}\leq X_{0}\right\} }=1+\sum_{i=1}^{n}\mathbf{1}_{\left\{ X_{i}\leq X_{0}\right\} }.\label{eq:defR0}
\end{equation}
In other words, we are interested in the distribution of the rank
of the normal random variable $X_{0}$, whose distribution typically
differs from that of the others.\footnote{Here, the indicator function $\mathbf{1}_{\mathcal{S}}$ equals 1
if the statement $\mathcal{S}$ is true and 0 otherwise.} We ask, how does $\mathscr{L}\left(R_{0}\right)$ depend on the parameters
$\mu_{0},\mu\in\mathbb{R}$, $\sigma_{0},\sigma\in\left(0,\infty\right)$,
and the size of the in-group $n\geq1$?\footnote{The notation $\mathscr{L}\left(X\right)$ indicates the law (or distribution)
of the random variable $X$.} Intuitively, we suspect the following: 
\begin{itemize}
\item Mean Effects: If $\mu_{0}\ll\mu$, then $X_{0}$ will tend to be smaller
than the other observations, so $R_{0}$ will likely be near 1. If
$\mu_{0}\gg\mu$, then $X_{0}$ will tend to be larger than the other
observations, so $R_{0}$ will likely be near $n+1$.
\item Variance Effects: If $\sigma_{0}\ll\sigma$, then the variation in
$X_{0}$ is small relative to the other variations, suggesting that
$R_{0}$ will cluster around $\nicefrac{n}{2}+1$. If $\sigma_{0}\gg\sigma$,
then the variation in $X_{0}$ is large relative to the other variations,
suggesting that the absolute deviation $\left|R_{0}-\left(\nicefrac{n}{2}+1\right)\right|$
will be approximately $\nicefrac{n}{2}$.
\end{itemize}
We additionally seek to determine the distributions of the remaining
rank variables, $R_{1},R_{2},\ldots,R_{n}$, which correspond to the
positions of $X_{1},X_{2},\ldots,X_{n}$ when all $n+1$ values are
ordered. Conditional on $R_{0}=k\in\left[n+1\right]$, the remaining
ranks are uniformly distributed over the set $\left[n+1\right]\setminus\left\{ k\right\} $.\footnote{Here, for any positive integer $m$, the notation $\left[m\right]$
stands in for $\left\{ 1,2,\ldots,m\right\} $.}

Equation (\ref{eq:defR0}) leads to a key observation underlying our
results. In particular, note that 
\begin{equation}
\left(R_{0}\left|X_{0}\right.\right)\sim1+\mathrm{Binomial}\left(n,\,\Phi\left(\frac{X_{0}-\mu}{\sigma}\right)\right),\label{eq:prior}
\end{equation}
where $\Phi$ denotes the cumulative distribution function (CDF) for
the standard normal distribution. Consequently, we obtain
\begin{equation}
\Pr\left(R_{0}=k\right)=\binom{n}{k-1}\mathbb{E}\left[\Phi\left(\frac{X_{0}-\mu}{\sigma}\right)^{k-1}\Phi\left(\frac{\mu-X_{0}}{\sigma}\right)^{n+1-k}\right],\label{eq:PrR0eqk}
\end{equation}
for $1\leq k\leq n+1$. While (\ref{eq:PrR0eqk}) looks intractable,
computing $\mathscr{L}\left(R_{0}\right)$ would be straightforward
from (\ref{eq:prior}) if $\Phi\left(\nicefrac{\left(X_{0}-\mu\right)}{\sigma}\right)$
were beta-distributed. This leads to our second key finding: $\Phi\left(\nicefrac{\left(X_{0}-\mu\right)}{\sigma}\right)$
is approximately beta-distributed. Section \ref{sec:Prior-Dist} proves
this claim and derives the corresponding beta distribution.

The distribution of $\Phi\left(\nicefrac{\left(X_{0}-\mu\right)}{\sigma}\right)$
does not vary independently with each of the four parameters $\mu_{0}$,
$\sigma_{0}$, $\mu$, and $\sigma$. To see this, fix $y\in\left(0,1\right)$
and note that
\begin{equation}
\Pr\left(\Phi\left(\frac{X_{0}-\mu}{\sigma}\right)\leq y\right)=\Pr\left(\frac{X_{0}-\mu_{0}}{\sigma_{0}}\leq\frac{\mu-\mu_{0}}{\sigma_{0}}+\frac{\sigma}{\sigma_{0}}\Phi^{-1}\left(y\right)\right),\label{eq:startDistZ}
\end{equation}
for $\Phi^{-1}$ the inverse standard normal CDF. By defining $\delta\coloneqq\nicefrac{\left(\mu-\mu_{0}\right)}{\sigma_{0}}$
and $\rho\coloneqq\nicefrac{\sigma}{\sigma_{0}}$, the expression
above becomes
\begin{equation}
F_{\rho,\delta}\left(y\right)\coloneqq\Phi\left(\delta+\rho\Phi^{-1}\left(y\right)\right)=\Pr\left(\frac{X_{0}-\mu_{0}}{\sigma_{0}}\leq\delta+\rho\Phi^{-1}\left(y\right)\right).\label{eq:endDistZ}
\end{equation}
The parameter $\delta\in\mathbb{R}$ standardizes the mean of $X_{1},X_{2},\ldots,X_{n}$
using the mean and standard deviation of $X_{0}$. Meanwhile, $\rho\in\left(0,\infty\right)$
gives the corresponding ratio of the two standard deviations. $\delta$
and $\rho$ fully characterize the distribution of $\Phi\left(\nicefrac{\left(X_{0}-\mu\right)}{\sigma}\right)$
and so assume a central role in our analysis.\footnote{WLOG one can focus on $X_{0}\sim\mathcal{N}\left(0,\,1\right)$ and
independent $X_{1},X_{2},\ldots,X_{n}\stackrel{\mathrm{iid}}{\sim}\mathcal{N}\left(\delta,\,\rho^{2}\right)$.\label{fn:one-case}} Note that if $\delta=0$ and $\rho=1$ (\emph{i.e.}, if $\mu=\mu_{0}$
and $\sigma=\sigma_{0}$), then $\Phi\left(\nicefrac{\left(X_{0}-\mu\right)}{\sigma}\right)$
is uniformly distributed on $\left(0,1\right)$. Moreover, let $\phi\coloneqq\Phi'$
denote the standard normal density. Then, the density of $\Phi\left(\nicefrac{\left(X_{0}-\mu\right)}{\sigma}\right)$
is given by
\begin{align}
f_{\rho,\delta}\left(y\right) & \coloneqq\rho\left.\phi\left(\delta+\rho\Phi^{-1}\left(y\right)\right)\right/\phi\left(\Phi^{-1}\left(y\right)\right)\label{eq:density}\\
 & =\rho\exp\left\{ \nicefrac{-1}{2}\left[\left(\rho^{2}-1\right)\Phi^{-1}\left(y\right){}^{2}+2\rho\delta\Phi^{-1}\left(y\right)+\delta^{2}\right]\right\} .\label{eq:densityExplicit}
\end{align}
Section \ref{subsec:Rank-Distributions} reviews rank distributions
based on order statistics, while Section \ref{subsec:Argument-Outline}
outlines the paper's structure.

\subsection{Order Statistics Models for Rank Distributions \label{subsec:Rank-Distributions}}

In the i.i.d.\ setting, where $X_{1},X_{2},\ldots,X_{n}\stackrel{\mathrm{iid}}{\sim}\mathcal{N}\left(\mu,\sigma^{2}\right)$
and $R_{i}$ denotes the rank of $X_{i}$ (as defined in (\ref{eq:defR0})),
the rank vector $\mathbf{R}$ is uniformly distributed over $\boldsymbol{\Pi}_{n}$,
the set of all permutations of $\left[n\right]$. That is, for every
$\mathbf{r}\in\boldsymbol{\Pi}_{n}$, $\Pr\left(\mathbf{R}=\mathbf{r}\right)=\nicefrac{1}{n!}$.
This result holds for any continuous distribution $F$.

The literature identifies four principal approaches for defining non-uniform
probability distributions over $\boldsymbol{\Pi}_{n}$ (\citet{CFV91,AY14}):
order statistics models, paired comparison models, distance-based
models, and multi-stage models. Given our focus, we define order statistics
models as follows. First, fix an arbitrary ranking $\mathbf{r}\in\boldsymbol{\Pi}_{n}$
and define indices $o_{j}$ so that $r_{o_{j}}=j$; that is, $o_{j}$
denotes the index of the $j$th smallest observation among the continuous
random variables $X_{1},X_{2},\ldots,X_{n}$ (the $X_{i}$ need not
be independent). An order statistics model then specifies that $\Pr\left(\mathbf{R}=\mathbf{r}\right)=\Pr\left(X_{o_{1}}<X_{o_{2}}<\cdots<X_{o_{n}}\right)$.
In other words, the probability assigned to $\mathbf{r}$ is the probability
that the latent $X_{i}$\textquoteright s occur in the order defined
by $\mathbf{r}$.

Order statistics models sometimes yield closed-form expressions for
ranking probabilities. For instance, \citet{MO67} showed that if
the $X_{i}$ are independent and $X_{i}\sim\mathrm{Exp}\left(\lambda_{i}\right)$,
then the ranking probability is

\begin{equation}
\Pr\left(\mathbf{R}=\mathbf{r}\right)=\prod_{j=1}^{n}\frac{\lambda_{o_{j}}}{\sum_{k=j}^{n}\lambda_{o_{k}}}.\label{eq:ExpRankProb}
\end{equation}
This result follows from the memoryless property of the exponential
distribution, which ensures that for any nonempty subset $\mathcal{S}\subset\left[n\right]$,
the minimum of $\left\{ X_{i},i\in\mathcal{S}\right\} $ has an $\mathrm{Exp}\left(\sum_{i\in\mathcal{S}}\lambda_{i}\right)$
distribution. Moreover, the exponential setting serves as an example
of both an order statistics model and a multi-stage model: the factor
corresponding to $j=1$ gives the probability that the $X_{i}$ with
$r_{i}=1$ is smallest, and conditional on this, the factor for $j=2$
gives the probability that the $X_{i}$ with $r_{i}=2$ is the next
smallest, and so on.

We now consider independent $X_{i}$ with $X_{i}\sim\mathrm{Gumbel}\left(\mu_{i},\sigma_{i}\right)$.
Since $e^{\nicefrac{-X_{i}}{\sigma_{i}}}\sim\mathrm{Exp}\left(e^{\nicefrac{\mu_{i}}{\sigma_{i}}}\right)$,
Equation (\ref{eq:ExpRankProb}) implies that
\begin{equation}
\Pr\left(\mathbf{R}=\mathbf{r}\right)=\prod_{j=1}^{n}\frac{\exp\left(\frac{\mu_{o_{n-j+1}}}{\sigma_{o_{n-j+1}}}\right)}{\sum_{k=1}^{n-j+1}\exp\left(\frac{\mu_{o_{k}}}{\sigma_{o_{k}}}\right)},\label{eq:GumbelRankProb}
\end{equation}
as shown by \citet{L59} and \citet{Y77}. The negative sign in the
exponent of $e^{\nicefrac{-X_{i}}{\sigma_{i}}}$ reverses the order
of traversal relative to the exponential case, so that the factor
corresponding to $j=1$ gives the probability that the $X_{i}$ with
$r_{i}=n$ is largest, and conditional on this, the factor for $j=2$
gives the probability that the $X_{i}$ with $r_{i}=n-1$ is the second
largest, and so on. Moreover, for any monotonically increasing function
$f$, vectors $\mathbf{Z}$ and $\left(f\left(Z_{1}\right),f\left(Z_{2}\right),\ldots,f\left(Z_{n}\right)\right)$
have identical rank distributions.

While the $\mathcal{O}\left(n^{2}\right)$ computation in (\ref{eq:ExpRankProb})
uses independent $X_{i}\sim\mathrm{Gamma}\left(1,\lambda_{i}\right)$,
we now consider a more general setting with independent $X_{i}\sim\mathrm{Gamma}\left(s,\lambda_{i}\right)$
and $s\in\left\{ 1,2,\ldots\right\} $. Since the sum of independent
$\xi_{i,1},\ldots,\xi_{i,s}\sim\textrm{Exp}\left(\lambda_{i}\right)$
follows a $\mathrm{Gamma}\left(s,\lambda_{i}\right)$ distribution,
\citet{H83} and \citet{S90} recast the problem as a race among $n$
independent Poisson processes $N_{i,t}\sim\mathrm{Poisson}\left(t\lambda_{i}\right)$,
with each racing to reach $s$ events before exiting. Setting $i_{n}\equiv0$
and defining $\Lambda_{j}\coloneqq\sum_{k=j}^{n}\lambda_{o_{k}}$,
\citet{H83} gives 
\begin{equation}
\Pr\left(\mathbf{R}=\mathbf{r}\right)=\sum_{i_{n-1}=0}^{s-1+i_{n}}\cdots\sum_{i_{1}=0}^{s-1+i_{2}}\prod_{j=1}^{n-1}\binom{s-1+i_{j}}{i_{j}}\left(\frac{\Lambda_{j+1}}{\Lambda_{j}}\right)^{i_{j}}\left(\frac{\lambda_{o_{j}}}{\Lambda_{j}}\right)^{s}.\label{eq:GammaRankProb}
\end{equation}
This formulation involves $\mathcal{O}\left(s^{n-1}\right)$ products
of negative binomial probabilities, each corresponding to the event
that process $o_{j}$ registers $i_{j}$ failures before achieving
$s$ successes, for $1\leq j<n$. The overall probability $\Pr\left(\mathbf{R}=\mathbf{r}\right)$
is obtained by summing these products over all numbers of failures
that conform with $\mathbf{r}$. While \citet{S90}'s expression similarly
involves $\mathcal{O}\left(s^{n-1}\right)$ summands---making direct
computation intractable---\citet{S90} advocates for more tractable
approximations. Noting that $\mathrm{Gamma}\left(s,\lambda_{i}\right)$
approaches $\mathcal{N}\left(\nicefrac{s}{\lambda_{i}},\nicefrac{s}{\lambda_{i}^{2}}\right)$
as $s\to\infty$, Section \ref{subsec:Benchmarking-Approximations}
approximates (\ref{eq:GammaRankProb}) in the setting with $s$ large
and $\lambda_{0}\ne\lambda_{1}=\lambda_{2}=\cdots=\lambda_{n}$.

Due to the widespread occurrence of the normal distribution in nature
and science, the originator of order statistics models for rank distributions
assumed latent variables $\mathbf{X}\sim\mathcal{N}_{n}\left(\boldsymbol{\mu},\boldsymbol{\Sigma}\right)$
(\citet{T27a,T27b,T27c}). Although closed-form expressions for $\Pr\left(\mathbf{R}=\mathbf{r}\right)$
exist in the exponential, Gumbel, and gamma cases (see (\ref{eq:ExpRankProb})--(\ref{eq:GammaRankProb})),
no such expressions have been derived for the normal setting---even
when the $X_{i}\sim\mathcal{N}\left(\mu_{i},\sigma_{i}^{2}\right)$
are independent. In this case, the probability is expressed as a sequence
of nested integrals:
\begin{equation}
\Pr\left(\mathbf{R}=\mathbf{r}\right)=\int_{-\infty}^{\infty}f_{o_{1},x_{1}}\int_{x_{1}}^{\infty}f_{o_{2},x_{2}}\cdots\int_{x_{n-1}}^{\infty}f_{o_{n},x_{n}}\,dx_{n}\cdots\,dx_{2}\:dx_{1},\label{eq:NormalRankProb}
\end{equation}
where $f_{i,x}\coloneqq\left.\phi\left(\nicefrac{\left(x-\mu_{i}\right)}{\sigma_{i}}\right)\right/\sigma_{i}$
(\citet{D50}). Although this formulation retains the nested structure
of the gamma case (\ref{eq:GammaRankProb}), replacing sums with integrals
makes it both computationally demanding and challenging to evaluate
accurately. As a result, published applications of (\ref{eq:NormalRankProb})
are generally limited to cases with $n\leq4$ (\emph{e.g.}, \citet{H81a,D86,LB94}).

\citet{H81a} approximates (\ref{eq:NormalRankProb}) under the conditions
$\mu_{i}\approx0$ and $\sigma_{i}=1$, for $1\leq i\leq n$. Define
$\mu_{\left(i\right)}\coloneqq\mathbb{E}Z_{\left(i\right)}\approx\Phi^{-1}\left(\nicefrac{\left(i-\nicefrac{3}{8}\right)}{\left(n-\nicefrac{3}{4}\right)}\right)$
as the mean of the $i$th order statistics of $n$ independent standard
normals (see \citet{B58}), and let $\psi_{k}\coloneqq\Phi^{-1}\left(\nicefrac{1}{k}\right)$.
Using a Taylor series expansion, \citet{H81a} shows that 
\begin{equation}
\Pr\left(\mathbf{R}=\mathbf{r}\right)\approx\Phi\left(\psi_{n!}+\frac{\sum_{i=1}^{n}\mu_{o_{i}}\mu_{\left(i\right)}}{n!\phi\left(\psi_{n!}\right)}\right).\label{eq:H81Rvec}
\end{equation}
\citet{H81a} and \citet{LB94} then put $j\ne i$ and sum arguments
to $\Phi$ on the right-hand side of (\ref{eq:H81Rvec}) to obtain
\begin{align}
\Pr\left(R_{i}=1\right) & \approx\Phi\left(\psi_{n}+\frac{\mu_{i}\mu{}_{\left(1\right)}}{\left(n-1\right)\phi\left(\psi_{n}\right)}\right),\label{eq:H81Ri1}\\
\Pr\left(R_{i}=1,R_{j}=2\right) & \approx\Phi\left(\psi_{n\left(n-1\right)}+\frac{\mu_{i}\mu{}_{\left(1\right)}+\mu_{j}\mu{}_{\left(2\right)}}{n\left(n-1\right)\phi\left(\psi_{n\left(n-1\right)}\right)}\right.\\
 & \hphantom{aaaaaaaaaaaaa}\left.+\frac{\left(\mu_{i}+\mu_{j}\right)\left(\mu{}_{\left(1\right)}+\mu{}_{\left(2\right)}\right)}{n\left(n-1\right)\left(n-2\right)\phi\left(\psi_{n\left(n-1\right)}\right)}\right).\label{eq:H81Ri1Rj2}
\end{align}
While our approximations assume at most one outlier, they remain accurate
for any choice of $\mu_{0}$, $\sigma_{0}$, $\mu$, and $\sigma$.
In contrast, the formulas in (\ref{eq:H81Rvec})--(\ref{eq:H81Ri1Rj2})
can handle up to $n$ distinct distributions but require $\mu_{i}\approx0$
and $\sigma_{i}=1$ for all $1\leq i\leq n$. Section \ref{subsec:Benchmarking-Approximations}
compares these methods in the single-outlier case with $\sigma_{0}=\sigma=1$.
Our approach delivers superior performance as $\left|\mu_{0}-\mu\right|$
grows (see Figure \ref{fig:NormalRankProb}).

\subsection{Outline \label{subsec:Argument-Outline}}

Section \ref{sec:Prior-Dist} begins by deriving the parameters $a_{\rho,\delta}$
and $b_{\rho,\delta}$. We then show that the transformed variable
$\Phi\left(\nicefrac{\left(X_{0}-\mu\right)}{\sigma}\right)$ is roughly
distributed as $\mathrm{Beta}\left(a_{\rho,\delta},b_{\rho,\delta}\right)$.
Building on this foundation, Section \ref{sec:Distribution-of-Rvec}
approximates $\mathscr{L}\left(R_{0},R_{i_{1}},\ldots,R_{i_{m}}\right)$
for indices $1\leq i_{1}<i_{2}<\cdots<i_{m}\leq n$. In Section (\ref{sec:Applications}),
we consider two applications of these main results, demonstrating
their practical impact. Finally, Section \ref{sec:Discussion-and-Conclusions}
wraps up with a discussion of our findings and potential future directions.
Rigorous proofs of key results are provided in the Appendices.

\section{The Distribution of $\Phi\left(\nicefrac{\left(X_{0}-\mu\right)}{\sigma}\right)$
\label{sec:Prior-Dist}}

Equation (\ref{eq:prior}) shows that the prior $\mathscr{L}\left(\Phi\left(\nicefrac{\left(X_{0}-\mu\right)}{\sigma}\right)\right)$
underpins the derivation of $\mathscr{L}\left(R_{0}\right)$. We approximate
it by $\mathrm{Beta}\left(a_{\rho,\delta},b_{\rho,\delta}\right)$
via three steps. First, Section \ref{subsec:Mean-and-Variance} derives
expressions for the mean and variance of $\Phi\left(\nicefrac{\left(X_{0}-\mu\right)}{\sigma}\right)$.
Next, Section \ref{subsec:A-Beta-Approximation} chooses $a_{\rho,\delta}$
and $b_{\rho,\delta}$ so that $\mathrm{Beta}\left(a_{\rho,\delta},b_{\rho,\delta}\right)$
matches those moments. Finally, Section \ref{subsec:Shared-Limiting-Distributions}
shows that, as $\rho$ or $\delta$ or both approach large or small
values, the beta approximation converges in distribution to $\mathscr{L}\left(\Phi\left(\nicefrac{\left(X_{0}-\mu\right)}{\sigma}\right)\right)$.

\subsection{Mean and Variance \label{subsec:Mean-and-Variance}}

Before computing the mean and variance of $Z_{\rho,\delta}\coloneqq\Phi\left(\nicefrac{\left(X_{0}-\mu\right)}{\sigma}\right)$,
we observe from (\ref{eq:densityExplicit}) that $f_{\rho,-\delta}\left(y\right)=f_{\rho,\delta}\left(1-y\right)$.
Consequently, for any $k\geq1$,
\begin{equation}
\mathbb{E}Z_{\rho,-\delta}^{k}=\sum_{j=0}^{k}\binom{k}{j}\left(-1\right)^{j}\mathbb{E}Z_{\rho,\delta}^{j}.\label{eq:expEquation}
\end{equation}
In particular, setting $k=1$ gives $\mathbb{E}Z_{\rho,-\delta}=1-\mathbb{E}Z_{\rho,\delta}$
(hence $\mathbb{E}Z_{\rho,0}=\nicefrac{1}{2}$), and setting $k=2$
yields $\mathrm{Var}\left(Z_{\rho,-\delta}\right)=\mathrm{Var}\left(Z_{\rho,\delta}\right)$.
These symmetries mirror those between $\mathrm{Beta}\left(\alpha,\beta\right)$
and $\mathrm{Beta}\left(\beta,\alpha\right)$. 

The next theorem---whose proof appears in Appendix \ref{sec:Derivations-for-=0000A72.1}---gives
an explicit formula for $\mathbb{E}Z_{\rho,\delta}$. 

\begin{restatable}{theorem}{genMean}

\label{thm:theGenMean}If $Z_{\rho,\delta}\sim F_{\rho,\delta}$ as
in (\ref{eq:endDistZ}), $\mathbb{E}Z_{\rho,\delta}=\Pr\left(X_{1}\leq X_{0}\right)=\Phi\left(\nicefrac{-\delta}{\sqrt{\rho^{2}+1}}\right)$.

\end{restatable}

We confirm that $\mathbb{E}Z_{\rho,-\delta}=1-\mathbb{E}Z_{\rho,\delta}$
and $\mathbb{E}Z_{\rho,0}=\nicefrac{1}{2}$. Before computing $\mathrm{Var}\left(Z_{\rho,\delta}\right)$,
we note from (\ref{eq:densityExplicit}) that the density $f_{\rho,\delta}\left(y\right)$
attains its maximum at: $y=\Phi\left(\nicefrac{-\rho\delta}{\left(\rho^{2}-1\right)}\right)$
if $\rho>1$; $y=0$ and $y=1$ if $\rho<1$; $y=0$ if $\rho=1$
and $\delta>0$; $y=1$ if $\rho=1$ and $\delta<0$; and any $0\leq y\leq1$
if $\rho=1$ and $\delta=0$. These mode locations align closely with
those of a beta distribution under similar parameter configurations
(see \citet{L24}).

The following theorem gives an intricate integral representation of
$\mathrm{Var}\left(Z_{\rho,\delta}\right)$. Its proof appears in
Appendix \ref{sec:Derivations-for-=0000A72.1}.

\begin{restatable}{theorem}{genVar}

\label{thm:theGenVar}Define, for any $\theta\in\mathbb{R}$,
\begin{align}
B_{\rho,\delta}\left(\theta\right) & \coloneqq\frac{\sqrt{6}\delta\sin\left(\theta+\nicefrac{\pi}{4}\right)}{\rho^{2}+2},\label{eq:B}\\
A_{\rho}\left(\theta\right) & \coloneqq\frac{\rho^{2}\left(\sin\left(2\theta\right)+2\right)+2\cos^{2}\left(\theta+\nicefrac{\pi}{4}\right)}{2\rho^{2}\left(\rho^{2}+2\right)},\textrm{ and}\label{eq:A}\\
G_{\rho,\delta}\left(\theta\right) & \coloneqq\frac{\sqrt{3}e^{-\frac{\delta^{2}}{\rho^{2}+2}}}{2\pi\rho\sqrt{\rho^{2}+2}}\frac{B_{\rho,\delta}\left(\theta\right)}{\left[2A_{\rho}\left(\theta\right)\right]^{\nicefrac{3}{2}}}\frac{\Phi\left(\frac{B_{\rho,\delta}\left(\theta\right)}{\sqrt{2A_{\rho}\left(\theta\right)}}\right)}{\phi\left(\frac{B_{\rho,\delta}\left(\theta\right)}{\sqrt{2A_{\rho}\left(\theta\right)}}\right)}.\label{eq:G}
\end{align}
If $Z_{\rho,\delta}\sim F_{\rho,\delta}$ as in (\ref{eq:endDistZ}),
then 
\begin{align}
\mathrm{Var}\left(Z_{\rho,\delta}\right) & =\Pr\left(X_{1}\leq X_{0},X_{2}\leq X_{0}\right)-\Pr\left(X_{1}\leq X_{0}\right)\Pr\left(X_{2}\leq X_{0}\right)\label{eq:varProbDiff}\\
 & =\int_{\nicefrac{11\pi}{12}}^{\nicefrac{19\pi}{12}}G_{\rho,\delta}\left(\theta\right)d\theta+\frac{\cos^{-1}\left(\nicefrac{-1}{\left(\rho^{2}+1\right)}\right)}{2\pi\exp\left(\nicefrac{\delta^{2}}{\left(\rho^{2}+2\right)}\right)}-\Phi\left(\nicefrac{-\delta}{\sqrt{\rho^{2}+1}}\right)^{2}.\label{eq:varZrhodelta}
\end{align}

\end{restatable}

\begin{figure}
\hfill{}\includegraphics[scale=0.4]{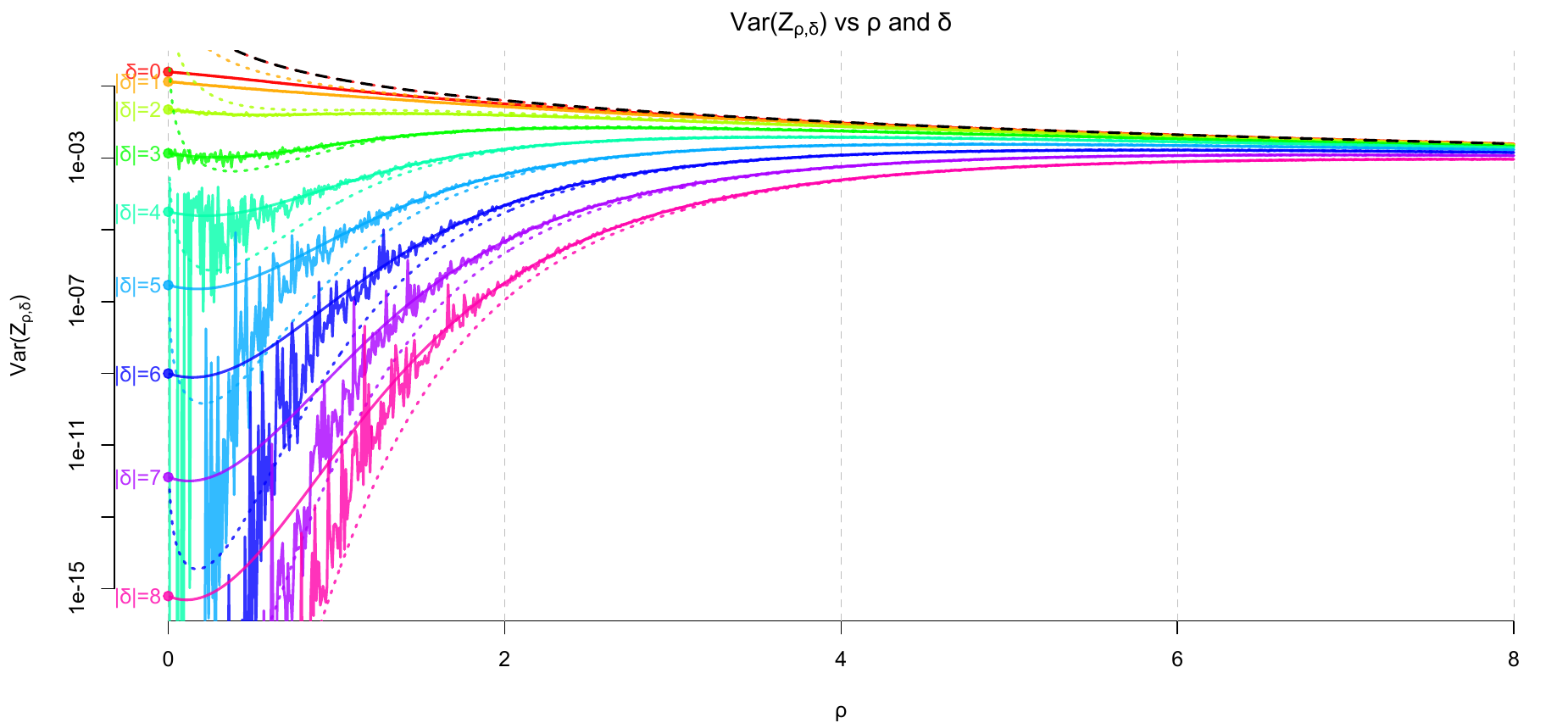}\hfill{}

\caption{Five approximations of $\mathrm{Var}\left(Z_{\rho,\delta}\right)$
are compared. Smooth, solid curves depict the numerical approximation
in (\ref{eq:approxVarIntegral}). Jagged lines show sample variances
computed from simulated $Z_{\rho,\delta}$ values. The dashed curve
corresponds to the approximation $\nicefrac{1}{2\pi\rho^{2}}$ (see
(\ref{eq:stanZtoN1})) while dotted curves are based on $\nicefrac{\phi\left(\nicefrac{\delta}{\sqrt{\rho^{2}+1}}\right)^{2}}{\rho^{2}}$
(see (\ref{eq:stanZtoN2})). Finally, solid circles above $\rho=0$
are based on $\Phi\left(-\delta\right)\Phi\left(\delta\right)$ (see
(\ref{eq:ZtoBern})).}

\label{fig:varZapprox}
\end{figure}

Since $G_{\rho,\delta}$ is symmetric in $\theta$ about $\nicefrac{5\pi}{4}$,
the integral in (\ref{eq:varZrhodelta}) can be rewritten as twice
the integral taken over either the interval $\left[\nicefrac{11\pi}{12},\nicefrac{5\pi}{4}\right]$
or $\left[\nicefrac{5\pi}{4},\nicefrac{19\pi}{12}\right]$. Moreover,
the integral vanishes when $\delta=0$, which reproduces a result
from an earlier version of this paper (\citet{L24}). The proof in
Appendix \ref{sec:Derivations-for-=0000A72.1} further establishes
that, for $k\geq1$, $\mathbb{E}Z_{\rho,\delta}^{k}=\Pr\left(X_{1}\leq X_{0},\,X_{2}\leq X_{0},\,\ldots,\,X_{k}\leq X_{0}\right)$. 

Figure \ref{fig:varZapprox} computes $\mathrm{Var}\left(Z_{\rho,\delta}\right)$
in five different ways: 
\begin{enumerate}
\item Smooth, solid curves: These approximate the true value of $\mathrm{Var}\left(Z_{\rho,\delta}\right)$
from (\ref{eq:varZrhodelta}) using the expression
\begin{equation}
\epsilon\sum_{k=0}^{\left\lfloor \nicefrac{2\pi}{3\epsilon}\right\rfloor }G_{\rho,\delta}\left(\frac{11\pi}{12}+k\epsilon\right)+\frac{\cos^{-1}\left(\nicefrac{-1}{\left(\rho^{2}+1\right)}\right)}{2\pi\exp\left(\nicefrac{\delta^{2}}{\left(\rho^{2}+2\right)}\right)}-\Phi\left(\nicefrac{-\delta}{\sqrt{\rho^{2}+1}}\right)^{2},\label{eq:approxVarIntegral}
\end{equation}
with $\epsilon\coloneqq10^{-4}$. The interval $\left[\nicefrac{11\pi}{12},\nicefrac{19\pi}{12}\right)$
is divided into roughly 20,000 equal-width bins. 
\item Jagged lines: These show the sample variances of sets $\left\{ \Phi\left(\nicefrac{\left(X_{0,j}-\delta\right)}{\rho}\right)\right\} _{j=1}^{m}$,
where $X_{0,j}$ are independent samples from $\mathcal{N}\left(0,1\right)$
and $m=10^{4}$. 
\item Dashed curve: This employs the approximation $\nicefrac{1}{2\pi\rho^{2}}$
from (\ref{eq:stanZtoN1}). 
\item Dotted curves: These use the approximation $\nicefrac{\phi\left(\nicefrac{\delta}{\sqrt{\rho^{2}+1}}\right)^{2}}{\rho^{2}}$
from (\ref{eq:stanZtoN2}). 
\item Solid circles above $\rho=0$: These use the approximation $\Phi\left(-\delta\right)\Phi\left(\delta\right)$
from (\ref{eq:ZtoBern}). 
\end{enumerate}
Overall, the simulated variances shown in Figure \ref{fig:varZapprox}
roughly match our approximations especially for $\rho$ large or $\left|\delta\right|$
small.

Finally, by combining Theorem \ref{thm:theGenMean} with $\mathbb{E}\left[Z_{\rho,\delta}\left(1-Z_{\rho,\delta}\right)\right]>0$
we obtain 
\begin{equation}
0<\mathrm{Var}\left(Z_{\rho,\delta}\right)<\Phi\left(\nicefrac{-\delta}{\sqrt{\rho^{2}+1}}\right)\Phi\left(\nicefrac{\delta}{\sqrt{\rho^{2}+1}}\right)\leq\nicefrac{1}{4},\label{eq:varBounds}
\end{equation}
which implies that $\lim_{\left|\delta\right|\rightarrow\infty}\mathrm{Var}\left(Z_{\rho,\delta}\right)=0$,
as expected. Furthermore, if $\left|\delta_{1}\right|<\left|\delta_{2}\right|$,
then $\mathrm{Var}\left(Z_{\rho,\delta_{1}}\right)>\mathrm{Var}\left(Z_{\rho,\delta_{2}}\right)$.
Since $\mathrm{Var}\left(Z_{\rho,\delta}\right)=\mathrm{Var}\left(Z_{\rho,-\delta}\right)$,
it is sufficient to consider the case $0\leq\delta_{1}<\delta_{2}$.
In this setting, since $\Phi\left(\nicefrac{-\left(\delta_{1}+\delta_{2}\right)}{2\rho}\right)<\nicefrac{1}{2}$,
the desired result follows if $\mathrm{sign}\left(f_{\rho,\delta_{1}}\left(z\right)-f_{\rho,\delta_{2}}\left(z\right)\right)=\mathrm{sign}\left(z-\Phi\left(\nicefrac{-\left(\delta_{1}+\delta_{2}\right)}{2\rho}\right)\right)$,
for all $0<z<1$.\footnote{Let $\mathrm{sign}\left(x\right)\coloneqq\nicefrac{x}{\left|x\right|}$
if $x\ne0$ and zero otherwise.} This relationship is confirmed by (\ref{eq:densityExplicit}). Thus,
$\mathrm{Var}\left(Z_{\rho,\delta}\right)$ monotonically approaches
zero as $\delta\rightarrow-\infty$ or as $\delta\rightarrow\infty$
(see Figure \ref{fig:varZapprox}).

\subsection{A Beta Approximation for $F_{\rho,\delta}$ \label{subsec:A-Beta-Approximation}}

In this section and the next, we argue that $\mathscr{L}\left(\Phi\left(\nicefrac{\left(X_{0}-\mu\right)}{\sigma}\right)\right)\approx\mathrm{Beta}\left(a_{\rho,\delta},b_{\rho,\delta}\right)$,
for specific parameters $a_{\rho,\delta}$ and $b_{\rho,\delta}$.
In this section we derive $a_{\rho,\delta}$ and $b_{\rho,\delta}$
and provide empirical evidence supporting our claim. In the next section,
we adopt a more theoretical approach, showing that when either $\rho$
or $\delta$ becomes extreme (\emph{i.e.}, very small or very large),
the distribution $\mathscr{L}\left(\Phi\left(\nicefrac{\left(X_{0}-\mu\right)}{\sigma}\right)\right)$
converges to $\mathrm{Beta}\left(a_{\rho,\delta},b_{\rho,\delta}\right)$.
In both sections, we quantify the difference between $F_{X}$ and
$F_{Y}$ using the 2-Wasserstein distance defined as
\begin{equation}
W_{2}\left(X,Y\right)\coloneqq\sqrt{\int_{0}^{1}\left\{ F_{X}^{-1}\left(z\right)-F_{Y}^{-1}\left(z\right)\right\} ^{2}dz}.\label{eq:Wasserstein}
\end{equation}
We approximate the integral in (\ref{eq:Wasserstein}) using the simple
binning method described in (\ref{eq:approxVarIntegral}).

Our approach relies on mapping the parameter space of $\mathscr{L}\left(\Phi\left(\nicefrac{\left(X_{0}-\mu\right)}{\sigma}\right)\right)$
(which is $\left(0,\infty\right)\times\mathbb{R}$) to that of $\mathrm{Beta}\left(\alpha,\beta\right)$
(which is $\left(0,\infty\right)^{2}$). We achieve this by introducing
functions $a,b:\left(0,\infty\right)\times\mathbb{R}\rightarrow\left(0,\infty\right)$.
Define $Z_{\rho,\delta}\coloneqq\Phi\left(\nicefrac{\left(X_{0}-\mu\right)}{\sigma}\right)$
and let $X_{\alpha,\beta}\sim\mathrm{Beta}\left(\alpha,\beta\right)$.
Although uncountably many mappings exist, we focus on the one that
matches the mean and variance of $X_{\alpha,\beta}$ with those of
$Z_{\rho,\delta}$. This choice is natural and, as we shall see, yields
useful results. Specifically, we require that $\mathbb{E}X_{a\left(\rho,\delta\right),b\left(\rho,\delta\right)}=\mathbb{E}Z_{\rho,\delta}$
and $\mathrm{Var}\left(X_{a\left(\rho,\delta\right),b\left(\rho,\delta\right)}\right)=\mathrm{Var}\left(Z_{\rho,\delta}\right)$.
Solving these equations yields the following positive parameters (see
(\ref{eq:varBounds})):
\begin{align}
a_{\rho,\delta} & \coloneqq\frac{\Phi\left(\nicefrac{-\delta}{\sqrt{\rho^{2}+1}}\right)}{\mathrm{Var}\left(Z_{\rho,\delta}\right)}\left[\Phi\left(\nicefrac{-\delta}{\sqrt{\rho^{2}+1}}\right)\Phi\left(\nicefrac{\delta}{\sqrt{\rho^{2}+1}}\right)-\mathrm{Var}\left(Z_{\rho,\delta}\right)\right]\label{eq:alpha3}\\
b_{\rho,\delta} & \coloneqq\frac{\Phi\left(\nicefrac{\delta}{\sqrt{\rho^{2}+1}}\right)}{\mathrm{Var}\left(Z_{\rho,\delta}\right)}\left[\Phi\left(\nicefrac{-\delta}{\sqrt{\rho^{2}+1}}\right)\Phi\left(\nicefrac{\delta}{\sqrt{\rho^{2}+1}}\right)-\mathrm{Var}\left(Z_{\rho,\delta}\right)\right].\label{eq:beta3}
\end{align}
The examples that follow use the binning approach described in (\ref{eq:approxVarIntegral})
to approximate the integral in $\mathrm{Var}\left(Z_{\rho,\delta}\right)$.

\begin{figure}[t]
\hfill{}\includegraphics[scale=0.5]{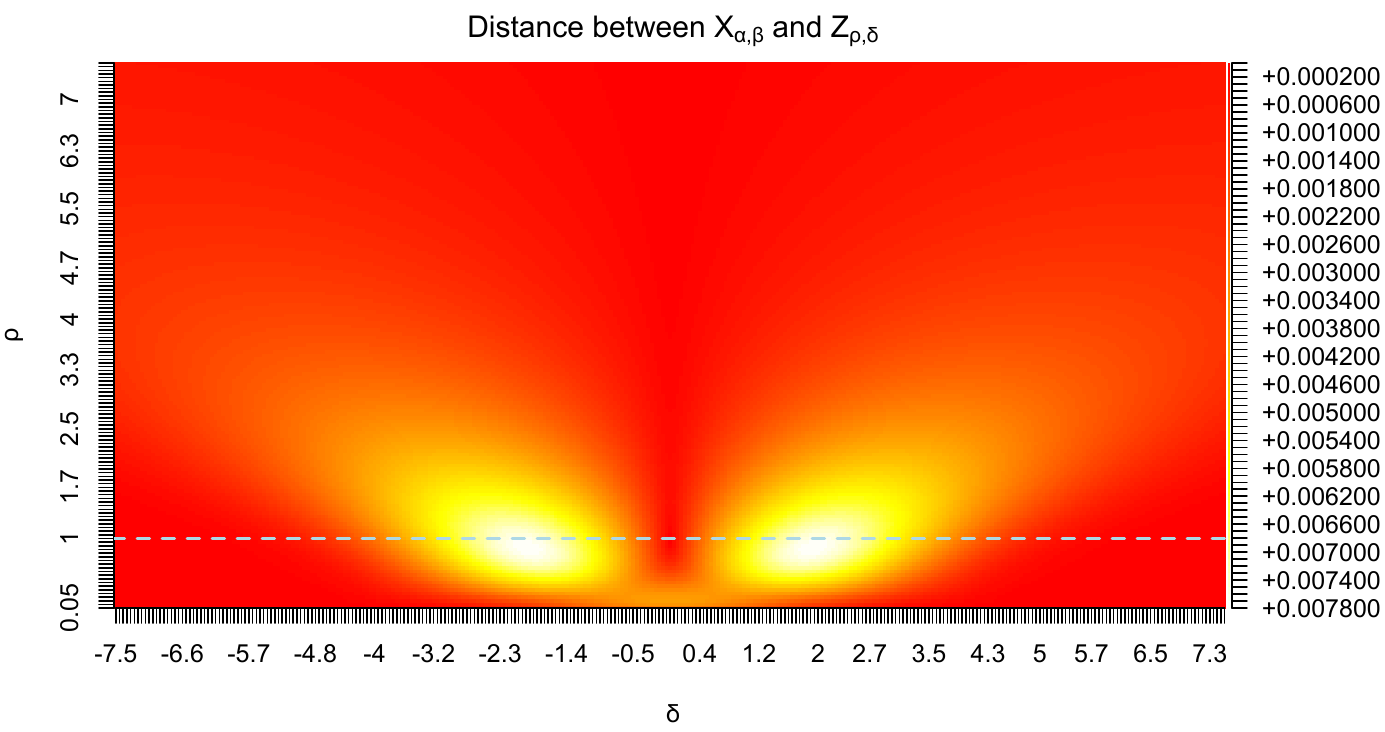}\hfill{}

\caption{2-Wasserstein distance $W_{2}\left(Z_{\rho,\delta},X_{a_{\rho,\delta},b_{\rho,\delta}}\right)$
for $\rho,\left|\delta\right|\protect\leq7.5$. The distance peaks
near $\left(\rho,\left|\delta\right|\right)=\left(0.85,1.90\right)$.
In this range, the distributions differ most when $\sigma\approx0.85\sigma_{0}$
and $\mu\approx\mu_{0}+1.90\sigma_{0}$. See Figure \ref{fig:egs}.}

\label{fig:distZX}
\end{figure}

\begin{figure}[b]
\hfill{}\includegraphics[scale=0.5]{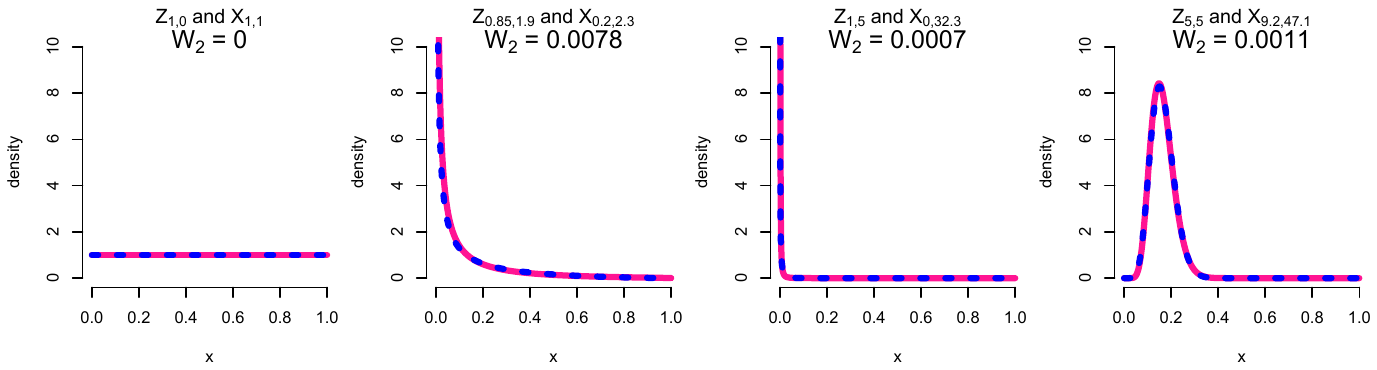}\hfill{}

\caption{Density functions for $\left(\rho,\delta\right)\in\left\{ \left(1,0\right),\left(0.85,1.90\right),\left(1,5\right),\left(5,5\right)\right\} $.
The second panel displays the maximally different case (see Figure
\ref{fig:distZX}), where the density of $X_{a_{\rho,\delta},b_{\rho,\delta}}$
(blue) exceeds that of $Z_{\rho,\delta}$ (pink) on $\left[0.16,0.70\right]$
and vice versa on $\left(0,0.16\right)\cup\left(0.70,1\right)$. For
$\rho,\left|\delta\right|\protect\leq7.5$, the distributions are
nearly identical.\vspace{-1em}}

\label{fig:egs}
\end{figure}

Figure \ref{fig:distZX} displays the values of $W_{2}\left(Z_{\rho,\delta},X_{a_{\rho,\delta},b_{\rho,\delta}}\right)$
for $\rho$ and $\left|\delta\right|$ up to 7.5. We observe that
this distance reaches its maximum near $\left(\rho,\left|\delta\right|\right)=\left(0.85,1.90\right)$.
In this range, the distributions $\mathscr{L}\left(\Phi\left(\nicefrac{\left(X_{0}-\mu\right)}{\sigma}\right)\right)$
and $\mathrm{Beta}\left(a_{\rho,\delta},b_{\rho,\delta}\right)$ differ
most when $\sigma\approx0.85\sigma_{0}$ and $\mu\approx\mu_{0}+1.90\sigma_{0}$.
Figure \ref{fig:egs} compares the density functions of $Z_{\rho,\delta}$
and $X_{a_{\rho,\delta},b_{\rho,\delta}}$ for several parameter pairs.
In the second panel---the maximally different case---the beta density
(in blue) exceeds the transformed normal density (in pink) on the
interval $\left[0.16,0.70\right]$ and vice versa on $\left(0,0.16\right)\cup\left(0.70,1\right)$.
Within the range $\rho,\left|\delta\right|\leq7.5$, the distributions
are nearly indistinguishable (see Figures \ref{fig:distZX} and \ref{fig:egs}),
coinciding exactly at $\left(\rho,\delta\right)=\left(1,0\right)$
where $\mathscr{L}\left(Z_{1,0}\right)=\mathscr{L}\left(X_{1,1}\right)=\mathrm{Uniform}\left(0,1\right)$.
Section \ref{subsec:Shared-Limiting-Distributions} further shows
that this near equivalence extends beyond $\rho,\left|\delta\right|\leq7.5$.

\subsection{Shared Limiting Distributions \label{subsec:Shared-Limiting-Distributions}}

Next, we consider cases where one or both of $\rho$ and $\delta$
become extremely large or small. In these regimes, we show that the
distributions $\mathscr{L}\left(\Phi\left(\nicefrac{\left(X_{0}-\mu\right)}{\sigma}\right)\right)$
and $\mathrm{Beta}\left(a_{\rho,\delta},b_{\rho,\delta}\right)$ grow
increasingly similar. We formalize this observation with two theorems---one
describing the limiting behavior of $Z_{\rho,\delta}\coloneqq\Phi\left(\nicefrac{\left(X_{0}-\mu\right)}{\sigma}\right)$
and the other describing the limiting behavior of $X_{a_{\rho,\delta},b_{\rho,\delta}}\sim\mathrm{Beta}\left(a_{\rho,\delta},b_{\rho,\delta}\right)$---along
with a corollary addressing the limiting 2-Wasserstein distances.
For clarity, we define the standardization function $\mathfrak{s}:\mathbb{R}\rightarrow\mathbb{R}$
by $\mathfrak{s}\left(x\right)\coloneqq\rho\left(x-\Phi\left(\nicefrac{-\delta}{\sqrt{\rho^{2}+1}}\right)\right)$,
and use $\longrightarrow$ and $\implies$ to denote convergence in
probability and convergence in distribution, respectively.

\begin{restatable}{theorem}{limitZ}

\label{thm:limitZ}With $Z_{\rho,\delta}$ and $\mathfrak{s}$ defined
as above, the following limits hold: 
\begin{align}
Z_{\rho,\delta} & \longrightarrow1\textrm{ as }\delta\rightarrow-\infty,\label{eq:Zto1}\\
Z_{\rho,\delta} & \longrightarrow\nicefrac{1}{2}\textrm{ as }\rho\rightarrow\infty,\label{eq:ZtoHalf}\\
Z_{\rho,\delta} & \longrightarrow\Phi\left(-r\right)\textrm{ as }\rho,\left|\delta\right|\rightarrow\infty,\,\nicefrac{\delta}{\rho}=r\textrm{ fixed},\label{eq:ZtoPhi}\\
Z_{\rho,\delta} & \longrightarrow0\textrm{ as }\delta\rightarrow\infty,\label{eq:Zto0}\\
Z_{\rho,\delta} & \implies\mathrm{Bernoulli}\left(\Phi\left(-\delta\right)\right)\textrm{ as }\rho\rightarrow0^{+},\label{eq:ZtoBern}\\
\mathfrak{s}\left(Z_{\rho,\delta}\right) & \implies\mathcal{N}\left(0,\nicefrac{1}{2\pi}\right)\textrm{ as }\rho\rightarrow\infty,\label{eq:stanZtoN1}\\
\mathfrak{s}\left(Z_{\rho,\delta}\right) & \implies\mathcal{N}\left(0,\phi\left(r\right)^{2}\right)\textrm{ as }\rho,\left|\delta\right|\rightarrow\infty,\,\nicefrac{\delta}{\rho}=r\textrm{ fixed.}\label{eq:stanZtoN2}
\end{align}

\end{restatable}

Note too that $\mathfrak{s}\left(Z_{\rho,\delta}\right)\longrightarrow0$
as $\rho\rightarrow0^{+}$. See Appendix \ref{sec:Derivations-for-=0000A72.3}
for the proof of Theorem \ref{thm:limitZ}. For additional insight
into Theorem \ref{thm:limitZ}, note that $Z_{\rho,\delta}$ serves
as a binomial prior for the (shifted) rank of $X_{0}\sim\mathcal{N}\left(0,1\right)$
among independent $X_{1},X_{2},\ldots,X_{n}\sim\mathcal{N}\left(\delta,\rho^{2}\right)$
(see (\ref{eq:prior})). Theorem \ref{thm:RankAsymp} directly addresses
these ranks. For now, our goal is simply to show that $Z_{\rho,\delta}$
converges to a beta distribution as its parameters become large or
small. We now turn to the corresponding limiting distributions of
$X_{a_{\rho,\delta},b_{\rho,\delta}}$.

\begin{restatable}{theorem}{limitX}

\label{thm:limitX}With $X_{a_{\rho,\delta},b_{\rho,\delta}}$ and
$\mathfrak{s}$ defined as above, the following limits hold:
\begin{align}
X_{a_{\rho,\delta},b_{\rho,\delta}} & \longrightarrow1\textrm{ as }\delta\rightarrow-\infty,\tag{\ref{eq:Zto1}\ensuremath{'}}\label{eq:Xto1}\\
X_{a_{\rho,\delta},b_{\rho,\delta}} & \longrightarrow\nicefrac{1}{2}\textrm{ as }\rho\rightarrow\infty,\tag{\ref{eq:ZtoHalf}\ensuremath{'}}\label{eq:XtoHalf}\\
X_{a_{\rho,\delta},b_{\rho,\delta}} & \longrightarrow\Phi\left(-r\right)\textrm{ as }\rho,\left|\delta\right|\rightarrow\infty,\,\nicefrac{\delta}{\rho}=r\textrm{ fixed},\tag{\ref{eq:ZtoPhi}\ensuremath{'}}\label{eq:XtoPhi}\\
X_{a_{\rho,\delta},b_{\rho,\delta}} & \longrightarrow0\textrm{ as }\delta\rightarrow\infty,\tag{\ref{eq:Zto0}\ensuremath{'}}\label{eq:Xto0}\\
X_{a_{\rho,\delta},b_{\rho,\delta}} & \implies\mathrm{Bernoulli}\left(\Phi\left(-\delta\right)\right)\textrm{ as }\rho\rightarrow0^{+},\tag{\ref{eq:ZtoBern}\ensuremath{'}}\label{eq:XtoBern}\\
\mathfrak{s}\left(X_{a_{\rho,\delta},b_{\rho,\delta}}\right) & \implies\mathcal{N}\left(0,\nicefrac{1}{2\pi}\right)\textrm{ as }\rho\rightarrow\infty,\tag{\ref{eq:stanZtoN1}\ensuremath{'}}\label{eq:stanXtoN1}\\
\mathfrak{s}\left(X_{a_{\rho,\delta},b_{\rho,\delta}}\right) & \implies\mathcal{N}\left(0,\phi\left(r\right)^{2}\right)\textrm{ as }\rho,\left|\delta\right|\rightarrow\infty,\,\nicefrac{\delta}{\rho}=r\textrm{ fixed.}\tag{\ref{eq:stanZtoN2}\ensuremath{'}}\label{eq:stanXtoN2}
\end{align}

\end{restatable}

Note that, as above, $\mathfrak{s}\left(X_{a_{\rho,\delta},b_{\rho,\delta}}\right)\longrightarrow0$
as $\rho\rightarrow0^{+}$. Appendix \ref{sec:Derivations-for-=0000A72.3}
proves Theorem \ref{thm:limitX} by dividing the proof into two parts.
The first---and more challenging---part demonstrates that

\begin{equation}
\lim_{\delta\rightarrow\infty}a_{\rho,\delta}=\lim_{\delta\rightarrow-\infty}b_{\rho,\delta}=0\quad\textrm{ and }\quad\lim_{\delta\rightarrow-\infty}a_{\rho,\delta}=\lim_{\delta\rightarrow\infty}b_{\rho,\delta}=\infty,
\end{equation}
so that $\Phi\left(\nicefrac{-\delta}{\sqrt{\rho^{2}+1}}\right)^{2}\lesssim\mathrm{Var}\left(Z_{\rho,\delta}\right)\lesssim\Phi\left(\nicefrac{-\delta}{\sqrt{\rho^{2}+1}}\right)$,
as $\delta\rightarrow\infty$\emph{ }(\emph{cf.}\ (\ref{eq:varBounds})).
The second, much simpler, part shows that a beta random variable under
these conditions exhibits the stated limiting behaviors.

The following corollary reinforces our main point by summarizing the
results of Theorems \ref{thm:limitZ} and \ref{thm:limitX}.

\begin{restatable}{mycorollary}{WasZero}

\label{cor:Was0}Under the settings above, the following convergence
results hold:
\begin{align}
W_{2}\left(Z_{\rho,\delta},X_{a_{\rho,\delta},b_{\rho,\delta}}\right) & \longrightarrow0\textrm{ as }\delta\rightarrow-\infty,\tag{\ref{eq:Zto1}\ensuremath{''}}\label{eq:ZXto1}\\
W_{2}\left(Z_{\rho,\delta},X_{a_{\rho,\delta},b_{\rho,\delta}}\right) & \longrightarrow0\textrm{ as }\rho\rightarrow\infty,\tag{\ref{eq:ZtoHalf}\ensuremath{''}}\label{eq:ZXtoHalf}\\
W_{2}\left(Z_{\rho,\delta},X_{a_{\rho,\delta},b_{\rho,\delta}}\right) & \longrightarrow0\textrm{ as }\rho,\left|\delta\right|\rightarrow\infty,\,\nicefrac{\delta}{\rho}=r\textrm{ fixed},\tag{\ref{eq:ZtoPhi}\ensuremath{''}}\label{eq:ZXtoPhi}\\
W_{2}\left(Z_{\rho,\delta},X_{a_{\rho,\delta},b_{\rho,\delta}}\right) & \longrightarrow0\textrm{ as }\delta\rightarrow\infty,\tag{\ref{eq:Zto0}\ensuremath{''}}\label{eq:ZXto0}\\
W_{2}\left(Z_{\rho,\delta},X_{a_{\rho,\delta},b_{\rho,\delta}}\right) & \longrightarrow0\textrm{ as }\rho\rightarrow0^{+},\tag{\ref{eq:ZtoBern}\ensuremath{''}}\label{eq:ZXtoBern}\\
W_{2}\left(\mathfrak{s}\left(Z_{\rho,\delta}\right),\mathfrak{s}\left(X_{a_{\rho,\delta},b_{\rho,\delta}}\right)\right) & \longrightarrow0\textrm{ as }\rho\rightarrow\infty,\tag{\ref{eq:stanZtoN1}\ensuremath{''}}\label{eq:stanZXtoN1}\\
W_{2}\left(\mathfrak{s}\left(Z_{\rho,\delta}\right),\mathfrak{s}\left(X_{a_{\rho,\delta},b_{\rho,\delta}}\right)\right) & \longrightarrow0\textrm{ as }\rho,\left|\delta\right|\rightarrow\infty,\,\nicefrac{\delta}{\rho}=r\textrm{ fixed.}\tag{\ref{eq:stanZtoN2}\ensuremath{''}}\label{eq:stanZXtoN2}
\end{align}

\end{restatable}
\begin{proof}
By comparing Theorems \ref{thm:limitZ} and \ref{thm:limitX}, we
observe that $Z_{\rho,\delta}$ and $X_{a_{\rho,\delta},b_{\rho,\delta}}$
share the same limiting distributions and second moments in all the
specified settings. Consequently, the 2-Wasserstein distances between
them converge to zero (\citet{PZ19}).
\end{proof}
In summary, when $\rho,\left|\delta\right|\leq7.5$, $\mathscr{L}\left(\Phi\left(\nicefrac{\left(X_{0}-\mu\right)}{\sigma}\right)\right)$
and $\mathrm{Beta}\left(a_{\rho,\delta},b_{\rho,\delta}\right)$ differ
most when $\left(\rho,\left|\delta\right|\right)\approx\left(0.85,1.90\right)$\emph{
}(Figure \ref{fig:distZX}). However, even in this \textquotedblleft worst-case\textquotedblright{}
scenario the difference is relatively small (Figure \ref{fig:egs}).
Moreover, the results presented here confirm that, in every specified
setting, the 2-Wasserstein distance between these distributions converges
to zero (Corollary \ref{cor:Was0}). Although we do not provide explicit
rates of convergence, these findings reinforce the robustness of the
beta approximation across all parameter combinations.

\section{Approximating $\mathscr{L}\left(R_{0},R_{i_{1}},\ldots,R_{i_{m}}\right)$
\label{sec:Distribution-of-Rvec}}

Section \ref{sec:Prior-Dist} showed that $\mathscr{L}\left(\Phi\left(\nicefrac{\left(X_{0}-\mu\right)}{\sigma}\right)\right)\approx\mathrm{Beta}\left(a_{\rho,\delta},b_{\rho,\delta}\right)$,
where $a_{\rho,\delta}$ and $b_{\rho,\delta}$ are defined in (\ref{eq:alpha3})--(\ref{eq:beta3}).
Meanwhile, (\ref{eq:prior}) implies that $\left(\left.R_{0}\right|X_{0}\right)-1$
has a binomial distribution with parameters $\left(n,\Phi\left(\nicefrac{\left(X_{0}-\mu\right)}{\sigma}\right)\right)$.
By marrying these two results through the beta-binomial framework,
we see that $R_{0}-1$ is accurately approximated by a beta-binomial
distribution with parameters $\left(n,a_{\rho,\delta},b_{\rho,\delta}\right)$.
Below, we explore this approximation and its implications for the
other ranks.

The beta-binomial law arises by mixing a binomial with a beta-distributed
success probability. Concretely, fix $\alpha,\beta>0$ and let $X_{\alpha,\beta}\sim\mathrm{Beta}\left(\alpha,\beta\right)$
and $\left.Y_{n}\right|X_{\alpha,\beta}\sim\mathrm{Binomial}\left(n,X_{\alpha,\beta}\right)$.
Marginally, $Y_{n}\sim\mathrm{BetaBinomial}\left(n,\alpha,\beta\right)$
with probability mass function

\begin{align}
\Pr\left(Y_{n}=j\right) & =\int_{0}^{1}\Pr\left(\left.Y_{n}=j\right|X_{\alpha,\beta}=z\right)g_{\alpha,\beta}\left(z\right)dz\label{eq:betaBinom}\\
 & =\binom{n}{j}\frac{B\left(\alpha+j,\beta+n-j\right)}{B\left(\alpha,\beta\right)},\textrm{ for }j=0,1,\ldots,n.
\end{align}
Here, $B:\left(0,\infty\right)^{2}\,\medrightarrow\left(0,\infty\right)$
and $g:\left(0,\infty\right)^{2}\times\left(0,1\right)\,\medrightarrow\left(0,\infty\right)$
give the
\begin{align}
\textrm{beta function:}\enskip & B\left(\alpha,\beta\right)\coloneqq\int_{0}^{1}y^{\alpha-1}\left(1-y\right)^{\beta-1}dy,\label{eq:betaFunc}\\
\textrm{beta density function:}\enskip & g_{\alpha,\beta}\left(y\right)\coloneqq\left.y^{\alpha-1}\left(1-y\right)^{\beta-1}\right/B\left(\alpha,\beta\right).\label{eq:betaDens}
\end{align}
Equivalently, if $\left.\xi_{1},\xi_{2},\ldots,\xi_{n}\right|X_{\alpha,\beta}$
are independent $\mathrm{Bernoulli}\left(X_{\alpha,\beta}\right)$
trials, then $Y_{n}=\sum_{i=1}^{n}\xi_{i}$ with $\mathbb{E}Y_{n}=\nicefrac{n\alpha}{\left(\alpha+\beta\right)}$
and $\mathrm{Var}\left(Y_{n}\right)=\nicefrac{n\alpha\beta}{\left(\alpha+\beta\right)^{2}}\left[1+\left(n-1\right)\iota\right]$,
where $\iota\coloneqq\mathrm{Cor}\left(\xi_{1},\xi_{2}\right)=\nicefrac{1}{\left(\alpha+\beta+1\right)}\in\left(0,1\right)$
is the intra-class correlation. Note that, while the mean of $Y_{n}$
coincides with that of a $\mathrm{Binomial}\left(n,\,\mathbb{E}X_{\alpha,\beta}\right)$
distribution, its variance is inflated by a factor of $1+\left(n-1\right)\iota$.
$Y_{n}$ counts positively correlated successes that rise and fall
with the latent beta variable.

Our presentation unfolds in three stages: In Section \ref{subsec:RONO},
we approximate the distribution of $R_{0}$, the rank of the odd normal
out. Building on that, Section \ref{subsec:All-Normals} derives an
approximation for the joint distribution of the rank vector $\left(R_{0},R_{i_{1}},\ldots,R_{i_{m}}\right)$,
where $1\leq i_{1}<i_{2}<\cdots<i_{m}\leq n$ and $1\leq m\leq n$.
Finally, Section \ref{subsec:RvecAsymptotics} investigates the asymptotic
behavior of these joint rank distributions.

\subsection{The Rank of the Odd Normal Out \label{subsec:RONO}}

\begin{figure}
\hfill{}\includegraphics[scale=0.5]{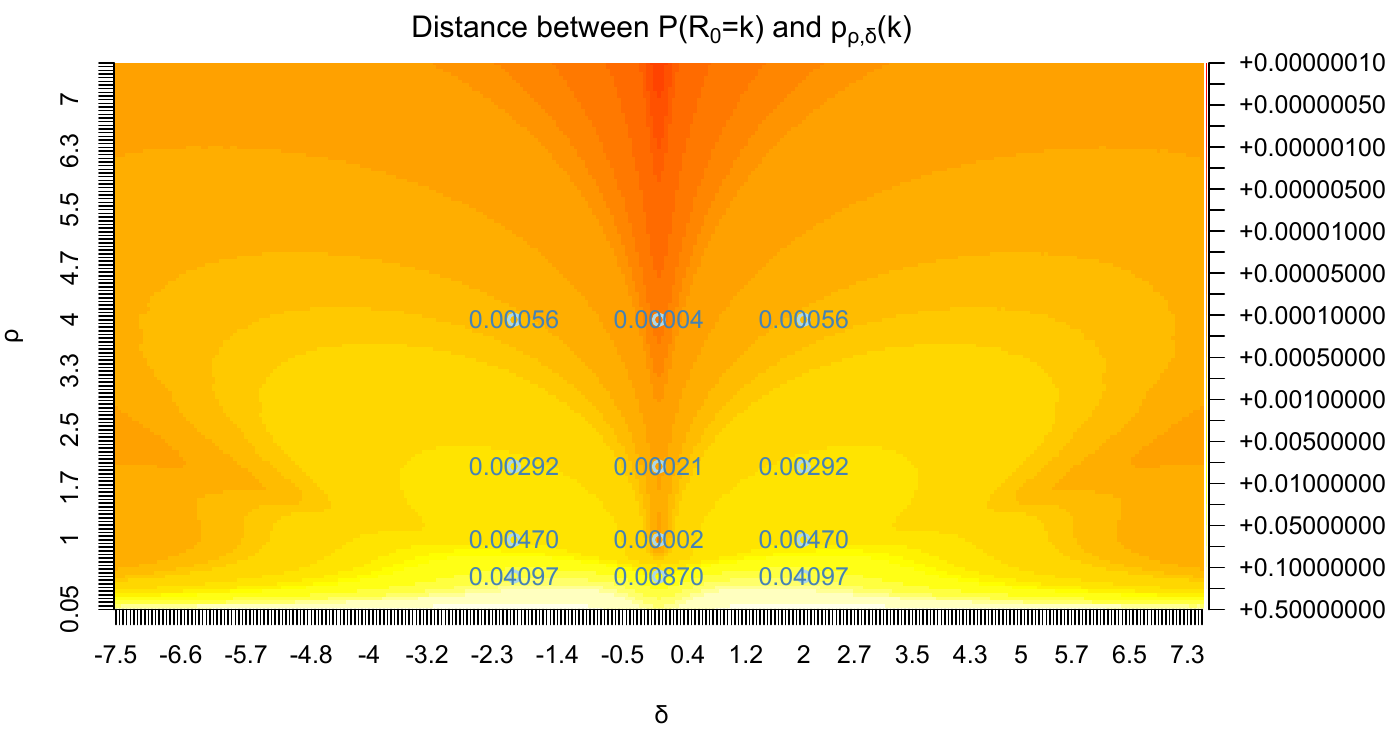}\hfill{}

\caption{1-Wasserstein distance $W_{1}\left(R_{0},R'_{0}\right)$ for $n=25$
and $\rho,\left|\delta\right|\protect\leq7.5$ (see (\ref{eq:W1})).
Light-blue circles---annotated with gray distance values---mark
the $\left(\rho,\delta\right)$ pairs displayed in Figure \ref{fig:R0hist}.
Distances are shown on a log-linear color scale with cutoffs at $\left\{ 1.0,2.5,5.0,7.5\right\} \times10^{-j}$,
for $j\in\left[7\right]$.}

\label{fig:W1R0}
\end{figure}
\begin{figure}
\hfill{}\includegraphics[scale=0.65]{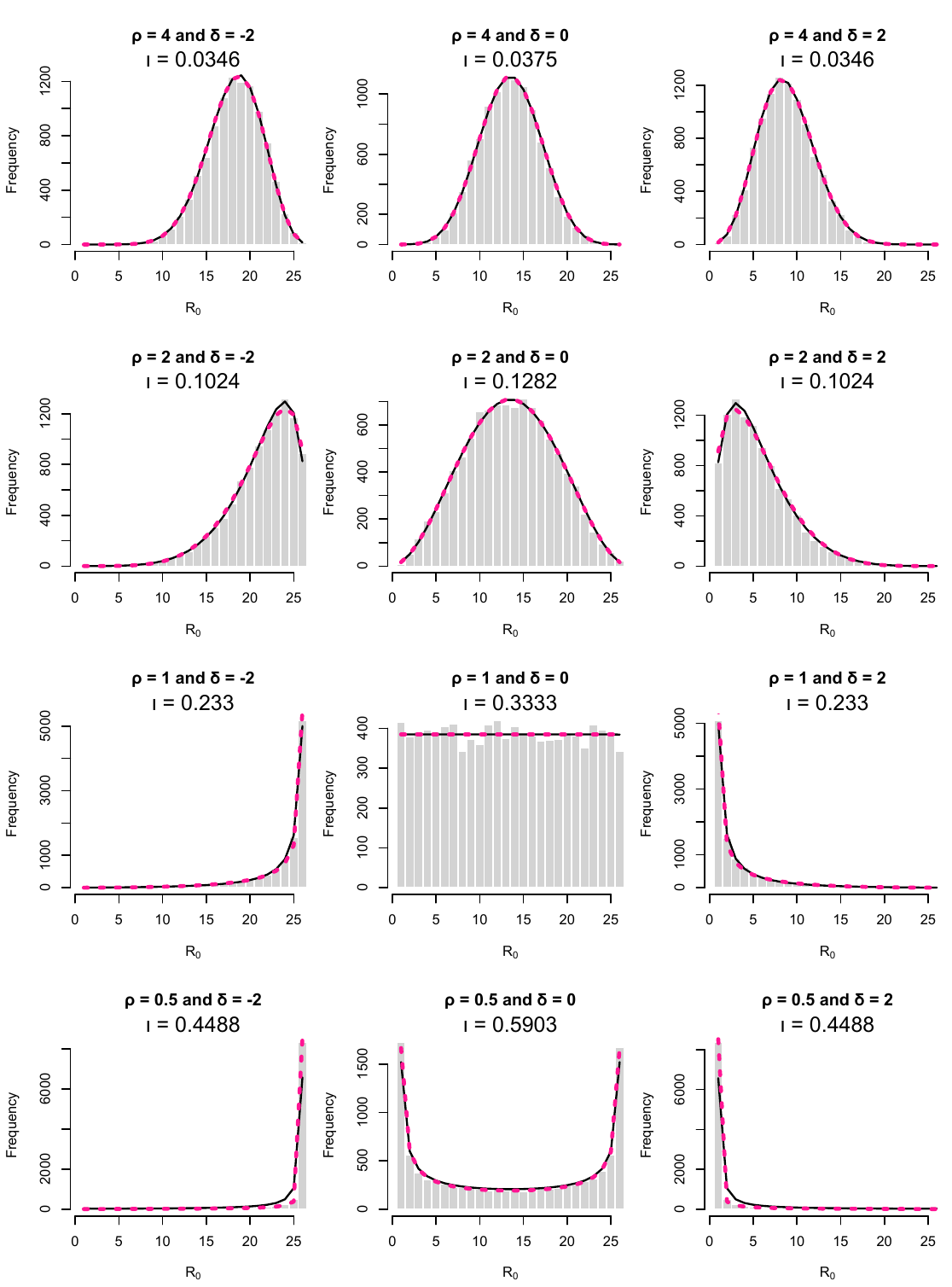}\hfill{}

\caption{Distributions of $R_{0}$ for $n=25$ and $\left(\rho,\delta\right)\in\left\{ \nicefrac{1}{2},1,2,4\right\} \times\left\{ -2,0,2\right\} $.
Gray bars show simulated frequencies of $R_{0}$; solid black curves
are the numerical-integral approximation of $\mathscr{L}\left(R_{0}\right)$
(Equations (\ref{eq:integralPR0}) and (\ref{eq:approxVarIntegral}));
dotted pink curves are the beta-binomial surrogate $\mathscr{L}\left(R'_{0}\right)$
(Equation (\ref{eq:approxPR0})). See Figure \ref{fig:W1R0} for the
corresponding 1-Wasserstein distances $W_{1}\left(R_{0},R'_{0}\right)$.
Each panel\textquoteright s title reports the intra-class correlation
$\iota_{\rho,\delta}\protect\coloneqq\mathrm{Cor}\left(\mathbf{1}_{\left\{ X_{0}\protect\leq X_{1}\right\} },\mathbf{1}_{\left\{ X_{0}\protect\leq X_{2}\right\} }\right)=\nicefrac{1}{\left(1+a_{\rho,\delta}+b_{\rho,\delta}\right)}$.}

\label{fig:R0hist}
\end{figure}

Let $1\leq k\leq n+1$. Combining Equation (\ref{eq:prior}) with
the beta approximation from Section \ref{sec:Prior-Dist} gives
\begin{align}
\Pr\left(R_{0}=k\right) & =\binom{n}{k-1}\int_{0}^{1}z^{k-1}\left(1-z\right)^{n+1-k}f_{\rho,\delta}\left(z\right)dz\label{eq:integralPR0}\\
 & \approx\binom{n}{k-1}\int_{0}^{1}z^{k-1}\left(1-z\right)^{n+1-k}g_{a_{\rho,\delta},b_{\rho,\delta}}\left(z\right)dz,\label{eq:integralApprox}
\end{align}
where $f_{\rho,\delta}$ and $g_{\alpha,\beta}$ denote the densities
of $\mathscr{L}\left(\Phi\left(\nicefrac{\left(X_{0}-\mu\right)}{\sigma}\right)\right)$
and $\mathrm{Beta}\left(\alpha,\beta\right)$, and $a_{\rho,\delta},b_{\rho,\delta}$
are defined in (\ref{eq:alpha3})--(\ref{eq:beta3}). Evaluating
the beta-kernel integral yields the closed-form
\begin{equation}
p_{\rho,\delta}\left(k\right)\coloneqq\binom{n}{k-1}\frac{B\left(a_{\rho,\delta}+k-1,\,b_{\rho,\delta}+n+1-k\right)}{B\left(a_{\rho,\delta},\,b_{\rho,\delta}\right)}\approx\Pr\left(R_{0}=k\right).\label{eq:approxPR0}
\end{equation}
The next paragraphs investigate how accurately this beta-binomial
formula (\ref{eq:approxPR0}) approximates $\Pr\left(R_{0}=k\right)$. 

Figure \ref{fig:W1R0} shows the 1-Wasserstein distance
\begin{equation}
W_{1}\left(R_{0},R'_{0}\right)\coloneqq\sum_{k=1}^{n+1}\left|\Pr\left(R_{0}\leq k\right)-\sum_{j=1}^{k}p_{\rho,\delta}\left(j\right)\right|,\label{eq:W1}
\end{equation}
which quantifies the difference between the true rank law $\mathscr{L}\left(R_{0}\right)$
and its beta-binomial surrogate $\mathscr{L}\left(R'_{0}\right)$,
when $n=25$ and $\rho,\left|\delta\right|\leq7.5$. The integral
in (\ref{eq:integralPR0}) is evaluated via the binning scheme of
(\ref{eq:approxVarIntegral}). In line with Theorem \ref{thm:RONO},
$W_{1}\left(R_{0},R'_{0}\right)$ decreases as either $\rho$ or $\left|\delta\right|$
increase. Figure \ref{fig:R0hist} compares three approximations of
the distribution of $R_{0}$ when $n=25$ and $\left(\rho,\delta\right)\in\left\{ \nicefrac{1}{2},1,2,4\right\} \times\left\{ -2,0,2\right\} $:
\begin{itemize}
\item Gray histograms summarize simulated $R_{0}$ values;
\item Black curves approximate the integral in $\Pr\left(R_{0}=k\right)$;
\item Pink curves use the beta-binomial surrogate $\Pr\left(R'_{0}=k\right)=p_{\rho,\delta}\left(k\right)$.
\end{itemize}
The largest gap between the pink and black curves occurs at $\left(\rho,\left|\delta\right|\right)=\left(\nicefrac{1}{2},2\right)$,
yet there the beta-binomial curves (pink) more faithfully track the
simulations than do the numerical-integral approximations (black)---an
artifact we attribute to floating-point precision (Section \ref{subsec:Benchmarking-Approximations}).
In all cases, $\mathscr{L}\left(R'_{0}\right)$ offers an excellent
approximation of $\mathscr{L}\left(R_{0}\right)$.

We next analyze the limiting behavior of the beta-binomial approximation
in (\ref{eq:approxPR0}). Appendix \ref{sec:Derivations-for-=0000A73.1}
derives the following:

\begin{restatable}{theorem}{limitRONO}

\label{thm:RONO}Let $1\leq k\leq n+1$ and assume that $\nicefrac{\delta}{\rho}$
is held fixed when taking the limit in (\ref{eq:rhoAbsDeltaInf}).
Then, we have:
\begin{align}
\lim_{\rho\rightarrow0^{+}} & \left|\Pr\left(R_{0}=k\right)-p_{\rho,\delta}\left(k\right)\right|=0,\label{eq:rho0}\\
\lim_{\rho\rightarrow\infty} & \left|\Pr\left(R_{0}=k\right)-p_{\rho,\delta}\left(k\right)\right|=0,\label{eq:rhoInf}\\
\lim_{\left|\delta\right|\rightarrow\infty} & \left|\Pr\left(R_{0}=k\right)-p_{\rho,\delta}\left(k\right)\right|=0,\label{eq:absDeltaInf}\\
\lim_{\rho,\left|\delta\right|\rightarrow\infty} & \left|\Pr\left(R_{0}=k\right)-p_{\rho,\delta}\left(k\right)\right|=0.\label{eq:rhoAbsDeltaInf}
\end{align}

\end{restatable}In summary, the approximation $p_{\rho,\delta}\left(k\right)$
in (\ref{eq:approxPR0}) closely matches $\Pr\left(R_{0}=k\right)$.
Equivalently, $R_{0}-1$ is approximately $\mathrm{BetaBinomial}\left(n,a_{\rho,\delta},b_{\rho,\delta}\right)$-distributed.

We now turn to the first two moments of $R_{0}$. First, define the
intra-class correlation 
\begin{align}
\iota_{\rho,\delta} & \coloneqq\mathrm{Cor}\left(\mathbf{1}_{\left\{ X_{1}\leq X_{0}\right\} },\mathbf{1}_{\left\{ X_{2}\leq X_{0}\right\} }\right)=\nicefrac{1}{\left(a_{\rho,\delta}+b_{\rho,\delta}+1\right)}\label{eq:intraClassCor}\\
 & =\mathrm{Var}\left(Z_{\rho,\delta}\right)\left/\left[\Phi\left(\nicefrac{-\delta}{\sqrt{\rho^{2}+1}}\right)\Phi\left(\nicefrac{\delta}{\sqrt{\rho^{2}+1}}\right)\right],\right.\label{eq:useVarBound}
\end{align}
which lies in $\left(0,1\right)$ by (\ref{eq:varBounds}). Applying
(\ref{eq:defR0}) together with Theorems \ref{thm:theGenMean}--\ref{thm:theGenVar}
gives
\begin{align}
\mathbb{E}R_{0} & =1+n\Phi\left(\nicefrac{-\delta}{\sqrt{\rho^{2}+1}}\right)\textrm{ and }\label{eq:meanR0}\\
\mathrm{Var}\left(R_{0}\right) & =n\Phi\left(\nicefrac{-\delta}{\sqrt{\rho^{2}+1}}\right)\Phi\left(\nicefrac{\delta}{\sqrt{\rho^{2}+1}}\right)\left[1+\left(n-1\right)\iota_{\rho,\delta}\right].\label{eq:varR0}
\end{align}
Although these formulae follow by straightforward calculation, they
match exactly the mean and variance of a $\mathrm{BetaBinomial}\left(n,a_{\rho,\delta},b_{\rho,\delta}\right)$
law for $R_{0}-1$. This agreement stems from our choice of $a_{\rho,\delta}$
and $b_{\rho,\delta}$ that put $\mathbb{E}X_{a_{\rho,\delta},b_{\rho,\delta}}=\mathbb{E}Z_{\rho,\delta}$
and $\mathrm{Var}\left(X_{a_{\rho,\delta},b_{\rho,\delta}}\right)=\mathrm{Var}\left(Z_{\rho,\delta}\right)$
(see Section \ref{subsec:A-Beta-Approximation}).

\subsection{The Approximate Distribution of $\left(R_{0},R_{i_{1}},\ldots,R_{i_{m}}\right)$
\label{subsec:All-Normals}}

In this section, we extend the approximation $\Pr\left(R_{0}=k\right)\approx p_{\rho,\delta}\left(k\right)$
to the joint distribution of $\left(R_{0},R_{i_{1}},\ldots,R_{i_{m}}\right)$.
The following proposition, proved in Appendix \ref{sec:Derivations-for-=0000A73.2},
underpins our approximations:

\begin{restatable}{proposition}{otherRanks}

\label{prop:otherRanks}Let $1\leq m\leq n$, choose indices $1\leq i_{1}<\cdots<i_{m}\leq n$,
and let $j_{0},j_{1},\ldots,j_{m}$ be $m+1$ distinct elements of
$\left\{ 1,2,\ldots,n+1\right\} $. Then, we have the following joint
rank distributions:
\begin{align}
\Pr\left(R_{0}=j_{0},R_{i_{1}}=j_{1},\ldots,R_{i_{m}}=j_{m}\right) & =\frac{\Pr\left(R_{0}=j_{0}\right)}{n\left(n-1\right)\cdots\left(n-m+1\right)},\label{eq:withR0}\\
\Pr\left(R_{i_{1}}=j_{1},R_{i_{2}}=j_{2},\ldots,R_{i_{m}}=j_{m}\right) & =\frac{1-\sum_{k=1}^{m}\Pr\left(R_{0}=j_{k}\right)}{n\left(n-1\right)\cdots\left(n-m+1\right)}.\label{eq:noR0}
\end{align}
With $U\sim\mathrm{Uniform}\left[n\right]$ and $(V,W)\sim\mathrm{Uniform}\left\{ \left(i,j\right)\in\left[n\right]^{2}:i\ne j\right\} $,
define:
\[
\begin{array}{lll}
\mu_{Z}\coloneqq\mathbb{E}Z_{\rho,\delta}\eqqcolon1-\bar{\mu}_{Z}, & v_{Z}\coloneqq\mathrm{Var}\left(Z_{\rho,\delta}\right), & \iota_{\rho,\delta}\coloneqq\frac{v_{Z}}{\mu_{Z}\bar{\mu}_{Z}}\textrm{ from \eqref{eq:intraClassCor}},\\
\mu_{U}\coloneqq\frac{n+1}{2}=\mathbb{E}U, & v_{U}\coloneqq\frac{n^{2}-1}{12}=\mathrm{Var}\left(U\right), & c_{1,2}\coloneqq-\frac{n+1}{12}=\mathrm{Cov}\left(V,W\right).
\end{array}
\]
(Theorems \ref{thm:theGenMean}--\ref{thm:theGenVar} supply $\mu_{Z},v_{Z}$.)
With these definitions, the first two moments and covariances of the
ranks satisfy
\begin{align}
\mathbb{E}R_{1} & =\mu_{U}+\bar{\mu}_{Z},\label{eq:meanR1}\\
\mathrm{Var}\left(R_{1}\right) & =v_{U}+n\mu_{Z}\bar{\mu}_{Z}\left[1-\left.\left(n-1\right)\iota_{\rho,\delta}\right/n\right],\label{eq:varR1}\\
\mathrm{Cov}\left(R_{0},R_{1}\right) & =-\mu_{Z}\bar{\mu}_{Z}\left[1+\left(n-1\right)\iota_{\rho,\delta}\right],\label{eq:covR0R1}\\
\mathrm{Cov}\left(R_{1},R_{2}\right) & =c_{1,2}+2\mu_{Z}\bar{\mu}_{Z}\left[\iota_{\rho,\delta}-\nicefrac{1}{2}\right].\label{eq:covR1R2}
\end{align}
Finally, (\ref{eq:meanR0}) and (\ref{eq:varR0}) give $\mathbb{E}R_{0}=1+n\mu_{Z}$
and $\mathrm{Var}\left(R_{0}\right)=-n\mathrm{Cov}\left(R_{0},R_{1}\right)$.

\end{restatable}

\begin{figure}
\hfill{}\includegraphics[scale=0.65]{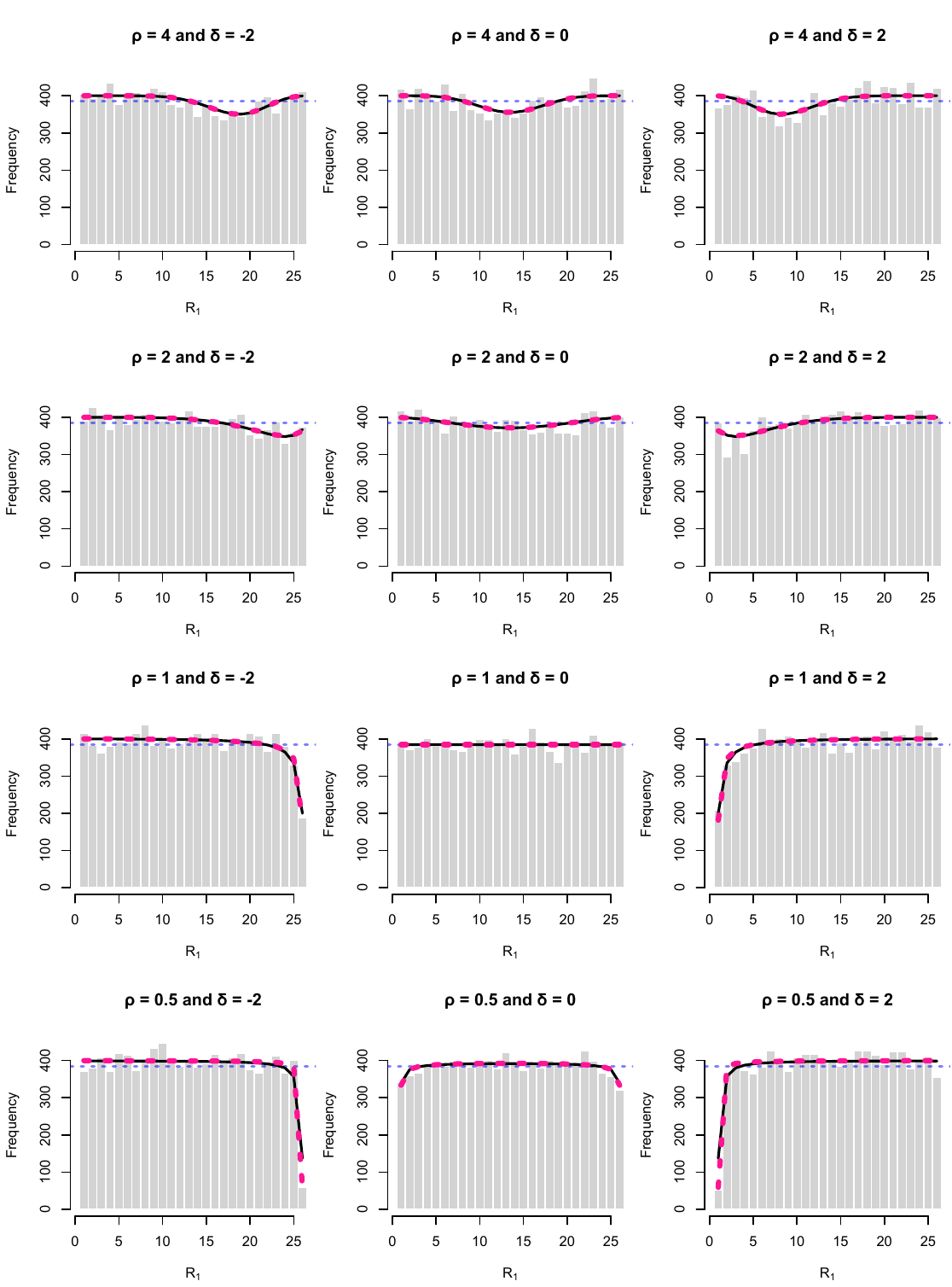}\hfill{}

\caption{Distributions of $R_{1}$ for $n=25$ and $\left(\rho,\delta\right)\in\left\{ \nicefrac{1}{2},1,2,4\right\} \times\left\{ -2,0,2\right\} $.
Gray bars show simulated frequencies of $R_{1}$; solid black curves
are the numerical-integral approximation of $\mathscr{L}\left(R_{1}\right)$
(Equations (\ref{eq:noR0}) and (\ref{eq:approxVarIntegral})); dotted
pink curves are the beta-binomial surrogate $\mathscr{L}\left(R'_{1}\right)$
(Equation (\ref{eq:approxNoR0})). The dotted blue line at $y=\frac{10000}{n+1}\approx384.6$
marks the expected number of ranks under $\mathrm{Uniform}\left[n+1\right]$.}

\label{fig:R1hist}
\end{figure}

Replacing $\Pr\left(R_{0}=k\right)$ in (\ref{eq:withR0})--(\ref{eq:noR0})
with the beta-binomial surrogate $p_{\rho,\delta}\left(k\right)$
from (\ref{eq:approxPR0}), we derive closed-form joint-rank approximations:
\begin{align}
\Pr\left(R'_{0}=j_{0},R'_{i_{1}}=j_{1},\ldots,R'_{i_{m}}=j_{m}\right) & \coloneqq\frac{p_{\rho,\delta}\left(j_{0}\right)}{n\left(n-1\right)\cdots\left(n-m+1\right)},\tag{\ref{eq:withR0}\ensuremath{'}}\label{eq:approxWithR0}\\
\Pr\left(R'_{i_{1}}=j_{1},R'_{i_{2}}=j_{2},\ldots,R'_{i_{m}}=j_{m}\right) & \coloneqq\frac{1-\sum_{k=1}^{m}p_{\rho,\delta}\left(j_{k}\right)}{n\left(n-1\right)\cdots\left(n-m+1\right)}.\tag{\ref{eq:noR0}\ensuremath{'}}\label{eq:approxNoR0}
\end{align}
Figure \ref{fig:R0hist} then compares three ways to approximate the
marginal distribution of $R_{1}$ when $n=25$ and $\left(\rho,\delta\right)\in\left\{ \nicefrac{1}{2},1,2,4\right\} \times\left\{ -2,0,2\right\} $: 
\begin{itemize}
\item Gray histograms summarize simulated $R_{1}$ values;
\item Black curves approximate the integral in $\Pr\left(R_{1}=k\right)=\frac{1-\Pr\left(R_{0}=k\right)}{n}$;
\item Pink curves use the beta-binomial surrogate $\Pr\left(R'_{1}=k\right)=\frac{1-p_{\rho,\delta}\left(k\right)}{n}$.
\end{itemize}
The surrogate law $\mathscr{L}\left(R'_{1}\right)$ closely tracks
the true law $\mathscr{L}\left(R_{1}\right)$. By construction, $\mathscr{L}\left(R_{1}\right)$
is the \textquotedblleft complement\textquotedblright{} of $\mathscr{L}\left(R_{0}\right)$,
placing mass where $\mathscr{L}\left(R_{0}\right)$ recedes. Symmetry
gives $\mathscr{L}\left(R_{1}\right)=\mathscr{L}\left(R_{2}\right)=\cdots=\mathscr{L}\left(R_{n}\right)$,
so that $\mathbb{E}\sum_{i=0}^{n}\mathbf{1}_{\left\{ R_{i}=k\right\} }=\sum_{i=0}^{n}\Pr\left(R_{i}=k\right)=1$
for $1\leq k\leq n+1$ (and likewise for the $R'_{i}$). Equivalently,
on average exactly one rank equals $k$. Finally, Theorem \ref{thm:RONO},
together with the triangle inequality, implies that both $\left|\textrm{\eqref{eq:withR0}}-\textrm{\eqref{eq:approxWithR0}}\right|$
and $\left|\textrm{\eqref{eq:noR0}}-\textrm{\eqref{eq:approxNoR0}}\right|$
tend to zero in the asymptotic regimes covered by the theorem.

Turning our attention to moment analysis, we outline five observations
that capture the essence of normal rank behavior when we have one
outlier.
\begin{enumerate}
\item Because $\mathbb{E}R_{1}$, $\mathrm{Var}\left(R_{1}\right)$, $\mathrm{Cov}\left(R_{0},R_{1}\right)$,
and $\mathrm{Cov}\left(R_{1},R_{2}\right)$ depend only on $\mathbb{E}R{}_{0}=\mathbb{E}R'_{0}$
and $\mathrm{Var}\left(R_{0}\right)=\mathrm{Var}\left(R'_{0}\right)$,
the surrogate ranks $R'_{i}$ share exactly the same means, variances
and covariances as the true ranks $R_{i}$.
\item Equations (\ref{eq:meanR0}) and (\ref{eq:meanR1}) then yield $\sum_{i=0}^{n}\mathbb{E}R_{i}=\sum_{i=0}^{n}\mathbb{E}R'_{i}=\sum_{k=1}^{n+1}k$.
\item Values not involving $R_{0}$ depend only weakly on $\left(\rho,\delta\right)\in\left(0,\infty\right)\times\mathbb{R}$.
This follows from the exchangeability of $R_{1},R_{2},\ldots,R_{n}$.
\begin{enumerate}
\item $\mathbb{E}R_{0}=1+n\Pr\left(X_{1}\leq X_{0}\right)$ ranges over
$\left(1,\,n+1\right)$ as $\left(\rho,\delta\right)$ vary while
$\mathbb{E}R_{1}=\mathbb{E}U+\Pr\left(X_{0}\leq X_{1}\right)$ lies
strictly between $\nicefrac{\left(n+1\right)}{2}$ and $\nicefrac{\left(n+3\right)}{2}$.
\item Defining $v_{Z}$ as in Proposition \ref{prop:otherRanks} and letting
$n\rightarrow\infty$ gives
\[
\begin{array}{ll}
\mathrm{Var}\left(\nicefrac{R_{0}}{n}\right)=v_{Z}+\mathcal{O}\left(\nicefrac{1}{n}\right), & \mathrm{Cov}\left(\nicefrac{R_{0}}{n},\nicefrac{R_{1}}{n}\right)=\nicefrac{-v_{Z}}{n}+\mathcal{O}\left(\nicefrac{1}{n^{2}}\right),\\
\mathrm{Var}\left(\nicefrac{R_{1}}{n}\right)=\nicefrac{1}{12}+\mathcal{O}\left(\nicefrac{1}{n}\right), & \mathrm{Cov}\left(\nicefrac{R_{1}}{n},\nicefrac{R_{2}}{n}\right)=\nicefrac{-1}{12n}+\mathcal{O}\left(\nicefrac{1}{n^{2}}\right).
\end{array}
\]
\end{enumerate}
\item For $n\geq2$, $\mathrm{sign}\left\{ \mathrm{Var}\left(R_{0}\right)-\mathrm{Var}\left(R_{1}\right)\right\} =\mathrm{sign}\left\{ \mathrm{Var}\left(Z_{\rho,\delta}\right)-\nicefrac{1}{12}\right\} $. 
\item Finally, $\mathrm{Var}\left(R_{0}\right)=\mathrm{Var}\left(R_{1}\right)=-\mathrm{Cov}\left(R_{0},R_{1}\right)=\Phi\left(\nicefrac{-\delta}{\sqrt{\rho^{2}+1}}\right)\Phi\left(\nicefrac{\delta}{\sqrt{\rho^{2}+1}}\right)$
in the special case of a single intra-class normal.
\end{enumerate}

\subsection{The Asymptotic Distributions of $\left(R_{0},R_{i_{1}},\ldots,R_{i_{m}}\right)$
\label{subsec:RvecAsymptotics}}

As we move into asymptotic regimes---letting $\rho,$ $\delta$,
or $n$ grow large or small---the joint rank law admits elegant simplifications.
To state our main result, we introduce three pieces of notation:
\begin{enumerate}
\item Let $\mathbf{0}_{k}$ and $\mathbf{1}_{k}$ denote the $k$-dimensional
vectors of zeros and ones.
\item For $1\leq m\leq n$, define $\mathcal{S}_{n,m}\coloneqq\left\{ \mathbf{j}\in\left[n\right]^{m}:j_{1},j_{2},\ldots,j_{m}\textrm{ all distinct}\right\} $.
Its cardinality is $\left|\mathcal{S}_{n,m}\right|\coloneqq\prod_{i=1}^{m}\left(n-i+1\right)$.
\item For $0<z<1$ and $\mathbf{v}\coloneqq\left(v_{0},v_{1},\ldots,v_{m}\right)\in\mathcal{S}_{n+1,m+1}$,
let $\mathbf{V}_{z}$ be the random vector in $\mathcal{S}_{n+1,m+1}$
with mass function
\begin{equation}
\Pr\left(\mathbf{V}_{z}=\mathbf{v}\right)=\frac{\binom{n}{v_{0}-1}z^{v_{0}-1}\left(1-z\right)^{n+1-v_{0}}}{\prod_{i=1}^{m}\left(n-i+1\right)}.
\end{equation}
\end{enumerate}
Appendix \ref{sec:Derivations-for-=0000A73.3} then establishes the
following theorem:

\begin{restatable}{theorem}{RankAsymp}

\label{thm:RankAsymp}Define $\mathbf{R}_{0,\mathbf{i}}$, $\mathbf{U}_{n,m}$,
$\xi$, $\boldsymbol{\xi}_{m}$, $\boldsymbol{\varUpsilon}_{m}$,
and $Z_{\rho,\delta}$ as follows:
\begin{enumerate}
\item Fix $1\leq i_{1}<i_{2}<\cdots<i_{m}\leq n$ and write $\mathbf{R}_{0,\mathbf{i}}\coloneqq\left(R_{0},R_{i_{1}},\ldots,R_{i_{m}}\right)$.
\item Let $\mathbf{U}_{n,m}\sim\mathrm{Uniform}\left(\mathcal{S}{}_{n,m}\right)$
and $\xi\sim\mathrm{Bernoulli}\left(\Phi\left(-\delta\right)\right)$
be independent. Also write $\boldsymbol{\xi}_{m}\coloneqq\left(\xi,\xi,\ldots,\xi\right)\in\left\{ \mathbf{0}_{m},\mathbf{1}_{m}\right\} $.
\item Let $\boldsymbol{\varUpsilon}_{m}\sim\mathrm{Uniform}\left(0,1\right)^{m}$
and $Z_{\rho,\delta}\coloneqq\Phi\left(\nicefrac{\left(X_{0}-\mu\right)}{\sigma}\right)$
be independent. 
\end{enumerate}
Then, as $\rho$, $\delta$, or $n$ diverge, $\mathbf{R}_{0,\mathbf{i}}$
converges in distribution as follows:
\begin{align}
\mathbf{R}_{0,\mathbf{i}} & \implies\left(n+1,\,\mathbf{U}_{n,m}\right)\textrm{ as }\delta\rightarrow-\infty,\label{eq:deltaNegInfRvec}\\
\mathbf{R}_{0,\mathbf{i}} & \implies\left(1+n\xi,\,\mathbf{1}_{m}-\boldsymbol{\xi}_{m}+\mathbf{U}_{n,m}\right)\textrm{ as }\rho\rightarrow0^{+},\label{eq:rhoZeroRvec}\\
\mathbf{R}_{0,\mathbf{i}} & \implies\left(1,\,1+\mathbf{U}_{n,m}\right)\textrm{ as }\delta\rightarrow\infty,\label{eq:deltaInfRvec}\\
\mathbf{R}_{0,\mathbf{i}} & \implies\mathbf{V}_{\nicefrac{1}{2}}\textrm{ as }\rho\rightarrow\infty,\label{eq:rhoInfRvec}\\
\mathbf{R}_{0,\mathbf{i}} & \implies\mathbf{V}_{\Phi\left(-r\right)}\textrm{ as }\rho,\left|\delta\right|\rightarrow\infty,\,\nicefrac{\delta}{\rho}=r\textrm{ fixed},\label{eq:rhoAbsDeltaInfRvec}\\
\nicefrac{1}{n}\left(\mathbf{R}_{0,\mathbf{i}}-\mathbf{1}_{m+1}\right) & \implies\left(Z_{\rho,\delta},\,\boldsymbol{\varUpsilon}_{m}\right)\textrm{ as }n\rightarrow\infty,\,m\textrm{ fixed}.\label{eq:nInfRvec}
\end{align}

\end{restatable}

We conclude with a few supplementary observations:
\begin{itemize}
\item Appendix \ref{sec:Derivations-for-=0000A73.3} refines Equations (\ref{eq:deltaNegInfRvec})
and (\ref{eq:deltaInfRvec}), showing that $R_{0}$ converges to the
stated limits in probability.
\item The same appendix demonstrates that, when $\sigma_{0}\ll\sigma$,
the indicators $\mathbf{1}_{\left\{ X_{i}\leq X_{0}\right\} }$ become
i.i.d., even though each depends on $X_{0}$. They follow $\mathrm{Bernoulli}\left(\nicefrac{1}{2}\right)$
and $\mathrm{Bernoulli}\left(\Phi\left(-r\right)\right)$ in settings
(\ref{eq:rhoInfRvec}) and (\ref{eq:rhoAbsDeltaInfRvec}).
\item Applying the normalization from Equation (\ref{eq:nInfRvec}) to both
sides of Equations (\ref{eq:deltaNegInfRvec})--(\ref{eq:rhoAbsDeltaInfRvec})
and then letting $n\rightarrow\infty$ recovers the limits (\ref{eq:Zto1})--(\ref{eq:ZtoBern})
in Theorem \ref{thm:limitZ}. Here, the limits in $\rho$ and $\delta$
commute with the limit in $n$.
\end{itemize}

\section{Applications \label{sec:Applications}}

Building on our earlier results, we explore two applications. In Section
\ref{subsec:Benchmarking-Approximations}, we benchmark our $\Pr(\mathbf{R}=\mathbf{r})$
approximation against existing single-outlier results. In Section
\ref{subsec:The-Median}, we then reexamine the minimum, median, and
maximum under the one-outlier assumption.

\subsection{Benchmarking Approximations \label{subsec:Benchmarking-Approximations}}

\begin{figure}[!t]
\hfill{}\includegraphics[scale=0.49]{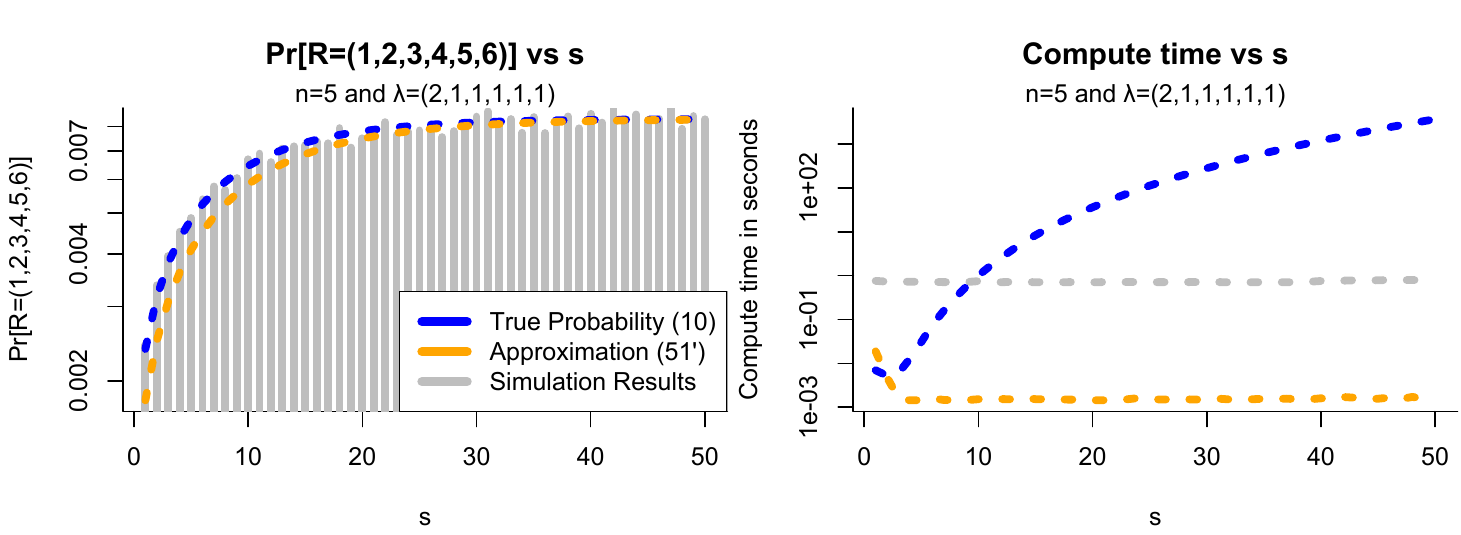}\hfill{}

\caption{Probability $\Pr\left(\mathbf{R}=\left(1,2,3,4,5,6\right)\right)$
(left) and computation time in seconds (right) vs $1\protect\leq s\protect\leq50$
for independent $X_{i}\sim\mathrm{Gamma}\left(s,\lambda_{i}\right)$
($0\protect\leq i\protect\leq5$) with $\boldsymbol{\lambda}=\left(2,1,1,1,1,1\right)$.
Probabilities are computed using the exact formula (\ref{eq:GammaRankProb}),
the beta-binomial approximation (\ref{eq:approxWithR0}), and Monte
Carlo simulation with $10^{5}$ samples per $s$. Both vertical axes
use a logarithmic scale.}

\label{fig:GammaRankProb}
\end{figure}

In settings with a single outlier, we benchmark beta-binomial approximations
of $\Pr\left(\mathbf{R}=\mathbf{r}\right)$ against the exact gamma-based
formula. Let $X_{i}\sim\mathrm{Gamma}\left(s,\lambda_{i}\right)$
be independent for $0\leq i\leq n$, with $s=1,2,\ldots$ and parameters
satisfying $\lambda_{0}\neq\lambda_{1}=\lambda_{2}=\cdots=\lambda_{n}$.
Although Equation (\ref{eq:GammaRankProb}) provides an exact expression
for $\Pr\left(\mathbf{R}=\mathbf{r}\right)$, its $\mathcal{O}\left(s^{n}\right)$
summands make it hard to compute for moderate $s$ and $n$ (\citet{S90}).
As $s$ increases, each $\mathrm{Gamma}\left(s,\lambda_{i}\right)$
converges to $\mathcal{N}\left(\nicefrac{s}{\lambda_{i}},\nicefrac{s}{\lambda_{i}^{2}}\right)$,
so we expect the beta-binomial approximation to improve. Figure \ref{fig:GammaRankProb}
(with $n=5$, $\boldsymbol{\lambda}=\left(2,1,1,1,1,1\right)$, $\mathbf{r}=\left(1,2,3,4,5,6\right)$,
and $1\leq s\leq50$) shows our approximation approaching the true
probability while the exact method\textquoteright s compute times
diverge. Hence, the value proposition of our approach grows with $s$.

At $s=50$, we also observe that the average runtime of our approximation
actually decreases slightly as $n$ increases. Define $\bar{t}_{n}$
as the mean compute time over 1000 runs with $n$ in-group normals
($1\leq n\leq1000$). A linear fit yields 
\[
\bar{t}_{n}\approx5.615\times10^{-3}-3.473\times10^{-8}n\textrm{ seconds,}
\]
implying about a 0.6\% runtime reduction across $n=1,2,\ldots,1000$
($p\approx10^{-9}$). We attribute this negative slope to R\textquoteright s
growing efficiency in evaluating the beta function at larger argument
values (\citet{R23}).

\begin{figure}[!t]
\hfill{}\includegraphics[scale=0.49]{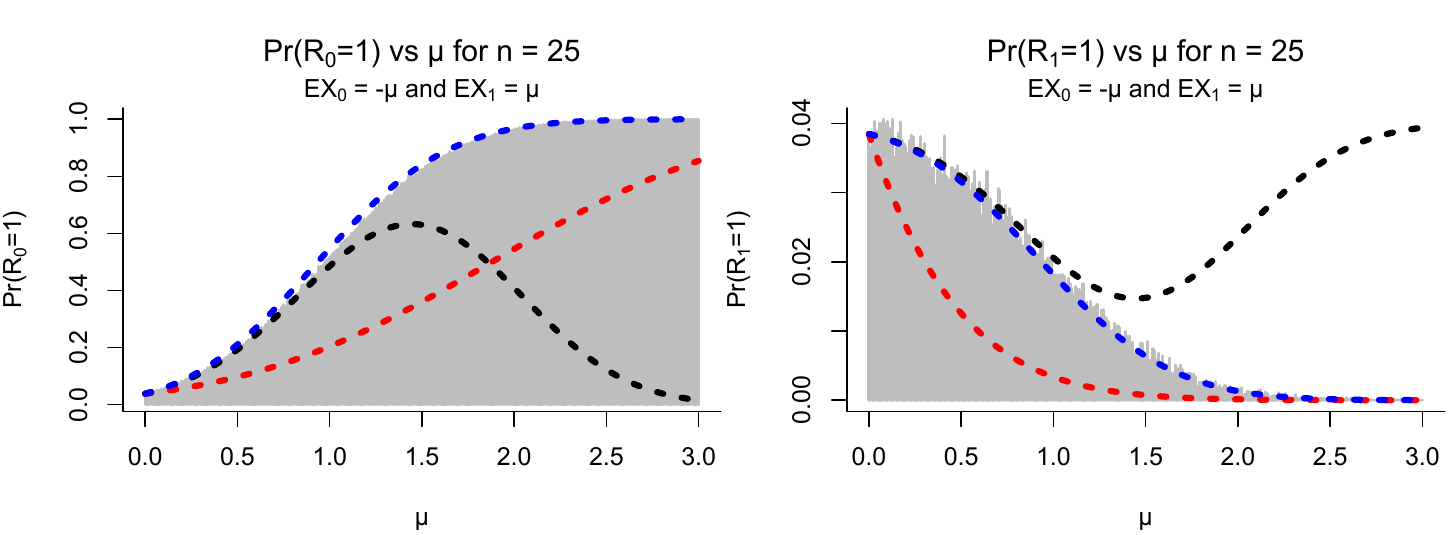}\hfill{}

\caption{Probabilities $\Pr\left(R_{0}=1\right)$ (left) and $\Pr\left(R_{1}=1\right)$
(right) as functions of $\mu$ under the model with independent $X_{0}\sim\mathcal{N}\left(-\mu,1\right)$
and $X_{i}\sim\mathcal{N}\left(\mu,1\right)$ for $1\protect\leq i\protect\leq25$.
Gray bars show simulation-based proportions using $10^{5}$ draws
at each $\mu\in\left\{ 0,0.01,\ldots,3\right\} $. Blue curves trace
the beta-binomial approximations from (\ref{eq:approxPR0}) and (\ref{eq:approxNoR0}).
Red curves depict the Taylor-series approximation in (\ref{eq:H81Ri1}).
Black curves represent the numerical-integration estimates of (\ref{eq:integralPR0})
and (\ref{eq:noR0}) using the binning approach of (\ref{eq:approxVarIntegral})
with $\epsilon=10^{-4}$.}

\label{fig:NormalRankProb}
\end{figure}

Returning to the normal setting with a single outlier, we compare
our beta-binomial approximations for $\Pr\left(R_{i}=1\right)$ against
the Taylor series approximation in (\ref{eq:H81Ri1}). We draw independent
$X_{0}\sim\mathcal{N}\left(-\mu,1\right)$ and $X_{i}\sim\mathcal{N}\left(\mu,1\right)$
for $1\leq i\leq25$ with $\mu$ running from 0 to 3 in steps of 0.01.
For each $\mu$, we estimate $\Pr\left(R_{0}=1\right)$ and $\Pr\left(R_{1}=1\right)$
by four methods:
\begin{enumerate}
\item Simulation (gray bars): We simulate $10^{5}$ vectors $\mathbf{X}\in\mathbb{R}^{26}$
for each $\mu$.
\item Beta-binomial (blue curves): (\ref{eq:approxPR0}) and (\ref{eq:approxNoR0})
track the simulated proportions closely but slightly overestimate
$\Pr\left(R_{0}=1\right)$ when $\nicefrac{3}{4}\leq\mu\leq\nicefrac{3}{2}$.
\item Numerical integration (black curves): The accuracy of (\ref{eq:approxVarIntegral})
with $\epsilon=10^{-4}$ applied to (\ref{eq:integralPR0}) and (\ref{eq:noR0})
deteriorates once $\mu\geq1$ ($\left|\delta\right|\geq2$; see Figures
\ref{fig:R0hist}--\ref{fig:R1hist}).
\item Taylor series (red curves): As expected, (\ref{eq:H81Ri1}) is serviceable
around $\mu\approx0$; unexpectedly, it outperforms numerical integration
at larger $\mu$.
\end{enumerate}
Across the entire $\mu$ range, the beta-binomial method matches simulation
results. Furthermore, as $\mu$ increases, it decisively outperforms
both the Taylor-series and numerical-integration approaches (Figure
\ref{fig:NormalRankProb}).

\subsection{The Minimum, Median, and Maximum \label{subsec:The-Median}}

\begin{figure}[!t]
\subfloat[$\Pr\left(R_{0}=m+1\right)$ versus $n+1=2m+1$.]{\hfill{}\includegraphics[scale=0.48]{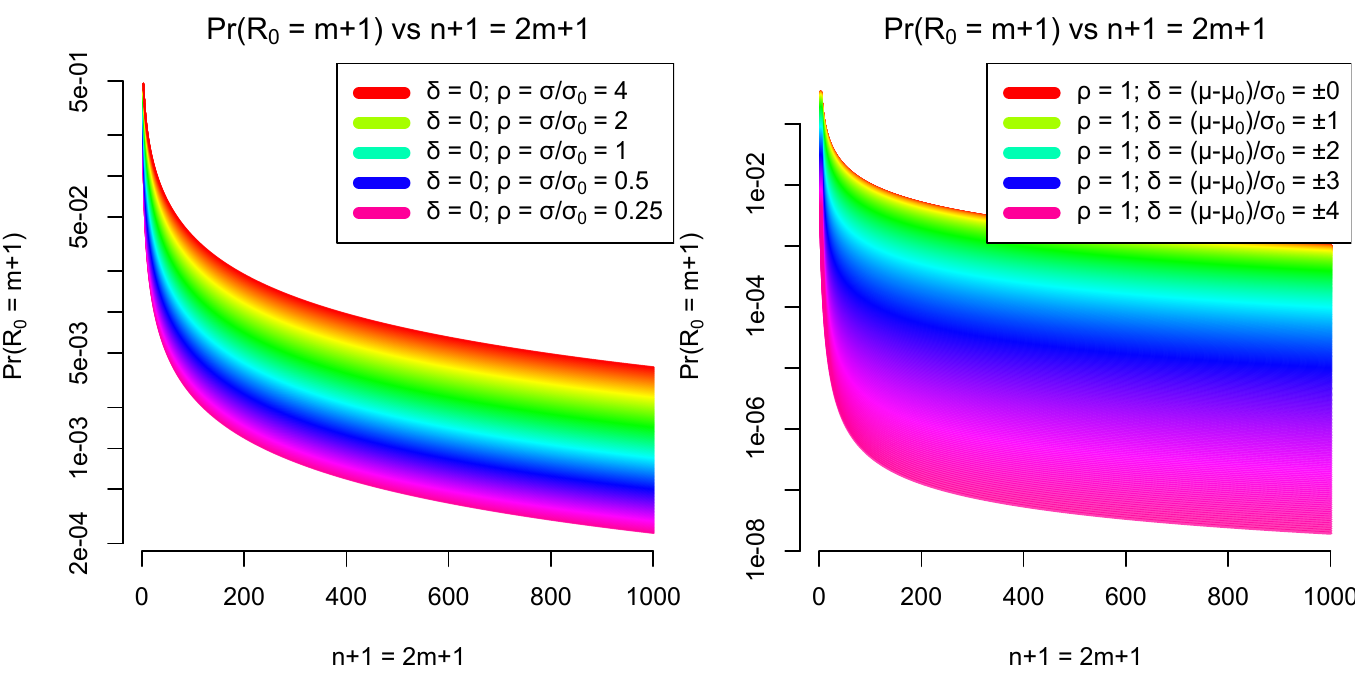}\hfill{}

\label{sub-fig:ProbR0med}}

\subfloat[$\Pr\left(R_{0}=n+1\right)$ versus $n+1$.]{\hfill{}\includegraphics[scale=0.48]{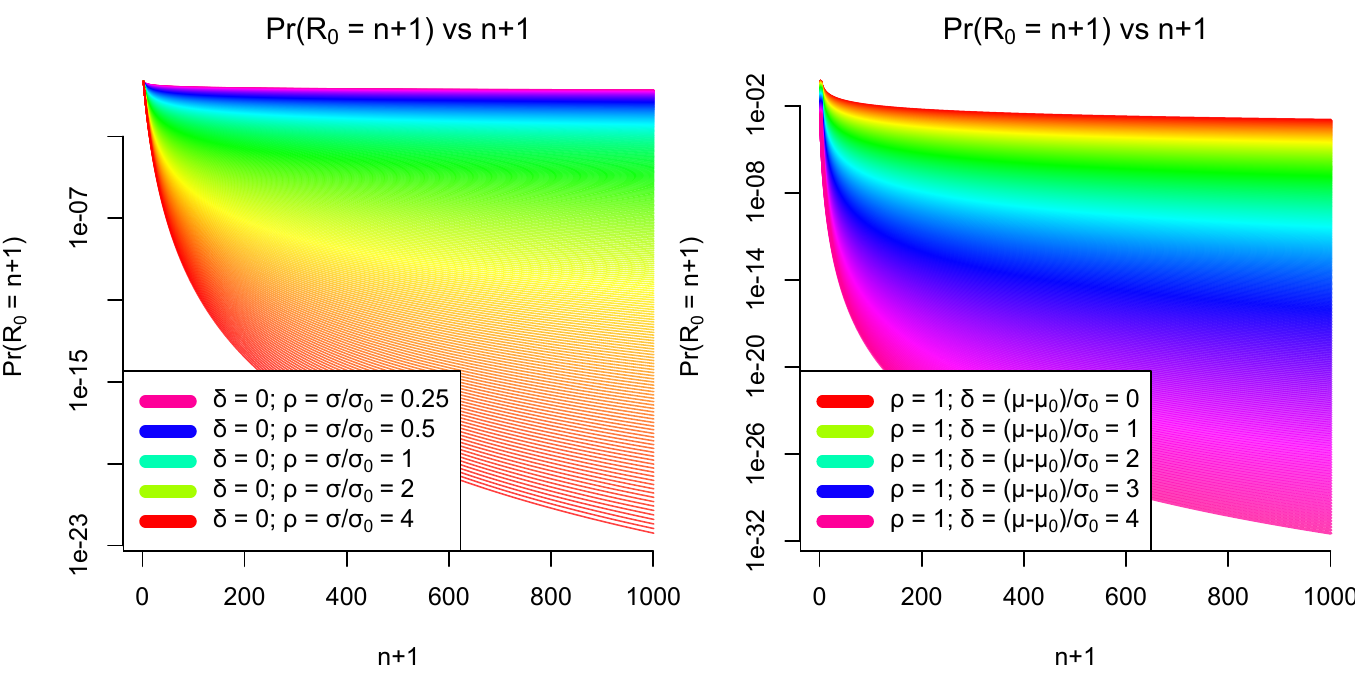}\hfill{}

\label{sub-fig:ProbR0max}}

\caption{Approximated probabilities $\Pr\left(R_{0}=m+1\right)$ and $\Pr\left(R_{0}=n+1\right)$
as functions of the parameters $\rho$, $\delta$, and the sample
size $n+1$. Panel \ref{sub-fig:ProbR0med} employs the central expression
from (\ref{eq:probR0med}) and an odd sample size while panel \ref{sub-fig:ProbR0max}
uses the central expression from (\ref{eq:probR0max}). The left panels
hold $\delta=0$ constant and vary $\rho$ over $\left\{ 2^{-2},2^{-1.99},\ldots,2^{2}\right\} $.
In contrast, the right panels fix $\rho=1$ and vary $\left|\delta\right|$
or $\delta$ over $\left\{ 0,0.01,\ldots,4\right\} $. Colors denote
distinct parameter combinations while the $y$-axes are displayed
on a logarithmic scale. Note that $\lim_{\delta\to-\infty}\Pr\left(R_{0}=n+1\right)=1$.}

\label{fig:probR0}
\end{figure}

Given their significance as measures of centrality and extremity,
we focus on the minimum, median, and maximum values. From (\ref{eq:approxPR0})
the probability that the differently-distributed normal occupies the
minimum, median, or maximum position is approximated by
\begin{align}
\Pr\left(R_{0}=1\right) & \approx\frac{B\left(a_{\rho,\delta},b_{\rho,\delta}+n\right)}{B\left(a_{\rho,\delta},b_{\rho,\delta}\right)}\stackrel[\infty]{n}{\sim}\frac{\Gamma\left(a_{\rho,\delta}+b_{\rho,\delta}\right)}{n^{b_{\rho,\delta}}\Gamma\left(a_{\rho,\delta}\right)},\label{eq:probR0min}\\
\Pr\left(R_{0}=m+1\right) & \approx\binom{2m}{m}\frac{B\left(a_{\rho,\delta}+m,b_{\rho,\delta}+m\right)}{B\left(a_{\rho,\delta},b_{\rho,\delta}\right)}\stackrel[\infty]{m}{\sim}\frac{2^{1-a_{\rho,\delta}-b_{\rho,\delta}}}{m\,B\left(a_{\rho,\delta},b_{\rho,\delta}\right)},\label{eq:probR0med}\\
\Pr\left(R_{0}=n+1\right) & \approx\frac{B\left(a_{\rho,\delta}+n,b_{\rho,\delta}\right)}{B\left(a_{\rho,\delta},b_{\rho,\delta}\right)}\stackrel[\infty]{n}{\sim}\frac{\Gamma\left(a_{\rho,\delta}+b_{\rho,\delta}\right)}{n^{a_{\rho,\delta}}\Gamma\left(b_{\rho,\delta}\right)}.\label{eq:probR0max}
\end{align}
Here, $a_{\rho,\delta}$ and $b_{\rho,\delta}$ are defined in (\ref{eq:alpha3})
and (\ref{eq:beta3}) while (\ref{eq:probR0med}) presumes $n=2m$
for $m\geq1$. Consequently, as $n\to\infty$, we observe that $\Pr\left(R_{0}=1\right)=\mathcal{O}\left(\nicefrac{1}{n^{b_{\rho,\delta}}}\right)$,
$\Pr\left(R_{0}=m+1\right)=\mathcal{O}\left(\nicefrac{1}{n}\right)$,
and $\Pr\left(R_{0}=n+1\right)=\mathcal{O}\left(\nicefrac{1}{n^{a_{\rho,\delta}}}\right)$.
For instance, in the i.i.d.\ setting, with $\rho=1$ and $\delta=0$
and for $1\leq k\leq n+1$, we have: 
\[
\Pr\left(R_{0}=k\right)=\frac{1}{n+1}=\binom{n}{k-1}\frac{B\left(k,n+2-k\right)}{B\left(1,1\right)}.
\]

Figure \ref{fig:probR0} shows how $\Pr\left(R_{0}=m+1\right)$ and
$\Pr\left(R_{0}=n+1\right)$ vary with $\rho$ and $\delta$. As $\rho$
grows, so that $\mathrm{Var}\left(X_{1}\right)$ increases relative
to $\mathrm{Var}\left(X_{0}\right)$, $\Pr\left(R_{0}=m+1\right)$
rises while $\Pr\left(R_{0}=n+1\right)$ falls. In contrast, an increase
in $\left|\delta\right|$, so that $\mathscr{L}\left(X_{1}\right)$
is increasingly shifted relative to $\mathscr{L}\left(X_{0}\right)$,
sends $\Pr\left(R_{0}=m+1\right)$ towards zero. When $\delta$ itself
increases, so that $\mathbb{E}X_{1}$ grows relative to $\mathbb{E}X_{0}$,
$\Pr\left(R_{0}=n+1\right)$ decreases. These results are consistent
with our expectations.

\section{Discussion and Conclusions \label{sec:Discussion-and-Conclusions}}

We have investigated the rank of a single outlier among a total of
$n+1$ independent Gaussian observations and shown that
\[
R_{0}\stackrel{.}{\sim}1+\mathrm{BetaBinomial}\left(n,a_{\rho,\delta},b_{\rho,\delta}\right).
\]
The derivation of this approximation rests on two key observations.
First, conditional on the outlier\textquoteright s value, $\left(R_{0}\mid X_{0}\right)-1\sim\mathrm{Binomial}\left(n,\;\Phi\left(\nicefrac{\left(X_{0}-\mu\right)}{\sigma}\right)\right)$.
Second, the distribution of $\Phi\left(\nicefrac{\left(X_{0}-\mu\right)}{\sigma}\right)$
can be closely approximated by $\mathrm{Beta}\left(a_{\rho,\delta},b_{\rho,\delta}\right)$
by matching means and variances. Furthermore, conditional on $R_{0}=k$,
the remaining ranks are uniformly distributed over the permutations
of $\left[n+1\right]\setminus\left\{ k\right\} $. These results extend
classical normal-rank theory to the simplest non-i.i.d.\ setting,
yielding formulas that are both computationally efficient and conceptually
transparent.

Our conjugacy-based framework streamlines marginalization and yields
an explicit rank distribution under non-i.i.d.\ sampling. Theorem
\ref{thm:RankAsymp}, together with Hoeffding\textquoteright s inequality
for $\epsilon>0$, then characterizes $R_{0}$'s asymptotic behavior:
\begin{itemize}
\item Mean shift: $R_{0}\longrightarrow n+1$ in probability as $\mu_{0}-\mu\to\infty$,
and $R_{0}\longrightarrow1$ in probability as $\mu-\mu_{0}\to\infty$; 
\item Variance inflation: $\left|R_{0}-\left(\nicefrac{n}{2}+1\right)\right|\longrightarrow\nicefrac{n}{2}$
in probability as $\nicefrac{\sigma_{0}}{\sigma}\to\infty$, and $\lim_{\nicefrac{\sigma}{\sigma_{0}}\to\infty}\Pr\left(\left|R_{0}-\left(\nicefrac{n}{2}+1\right)\right|\geq\epsilon n\right)\leq2\exp\left(-2\epsilon^{2}n\right)$.
\end{itemize}
\par\indent While our beta-binomial approximation yields an efficient
closed-form one-outlier solution, it also suggests several promising
extensions. Generalizing to heterogeneous utilities $\mathbf{X}\sim\mathcal{N}_{n}\left(\boldsymbol{\mu},\:\mathrm{diag}\left(\boldsymbol{\sigma}\right)\right)$
and to correlated utilities $\mathbf{X}\sim\mathcal{N}_{n}\left(\boldsymbol{\mu},\:\boldsymbol{\Sigma}\right)$
would realize Thurstone\textquoteright s original vision. Deriving
non-asymptotic bounds on $\left|\Pr\left(R_{0}=k\right)-p_{\rho,\delta}\left(k\right)\right|$
would equip practitioners with concrete error guarantees (\emph{cf.}\ Theorem
\ref{thm:RONO}). Finally, applying our formulae to rank-based procedures---such
as the Wilcoxon signed-rank test---could systematically assess their
robustness to outliers. We welcome collaborations to advance these
research directions.

We assess the feasibility of deriving exact rank formulas under various
utility distributions. Equations (\ref{eq:ExpRankProb})--(\ref{eq:GammaRankProb})
exploit the exponential distribution\textquoteright s memoryless property
to yield closed-form probabilities for independent $\mathrm{Exp}\left(\lambda_{i}\right)$,
$\mathrm{Gumbel}\left(\mu_{i},\sigma_{i}\right)$, and $\mathrm{Gamma}\left(s,\lambda_{i}\right)$
variates. The Gaussian single-outlier setting offers no such shortcut,
and we therefore rely on approximation techniques and Bayesian conjugacy.
Absent memorylessness, exact closed-form results remain elusive, making
high-fidelity approximations---like those developed here---the only
feasible alternative.

\section*{Acknowledgements}

I am especially grateful to Richard Olshen (May 17, 1942 -- November
8, 2023) for his guidance on the publication process, and to Brad
Efron and Amir Dembo for their insights, which significantly sharpened
this paper\textquoteright s presentation.

\appendix

\section{Proofs of Theorems \ref{thm:theGenMean} and \ref{thm:theGenVar}
\label{sec:Derivations-for-=0000A72.1}}

\genMean*
\begin{proof}
Using (\ref{eq:density}) and independent $X_{0}\sim\mathcal{N}\left(0,1\right)$
and $X_{1}\sim\mathcal{N}\left(\delta,\rho^{2}\right)$ (see footnote
\ref{fn:one-case}), we first note that 
\begin{align}
\mathbb{E}Z_{\rho,\delta} & =\int_{0}^{1}z\frac{\rho\phi\left(\delta+\rho\Phi^{-1}\left(z\right)\right)}{\phi\left(\Phi^{-1}\left(z\right)\right)}dz=\int_{-\infty}^{\infty}\Phi\left(\frac{y-\delta}{\rho}\right)\phi\left(y\right)dy\\
 & =\mathbb{E}\left[\Pr\left(\left.X_{1}\leq X_{0}\right|X_{0}\right)\right]=\Pr\left(X_{1}\leq X_{0}\right).
\end{align}
In what follows let $\mathcal{R}\coloneqq\left\{ \left(x,y\right)^{\mathsf{T}}\in\mathbb{R}^{2}:x>y\right\} $.
Note that
\begin{align}
\mathbb{E}Z_{\rho,\delta} & =\int_{0}^{1}\left(1-\Phi\left(\delta+\rho\Phi^{-1}\left(z\right)\right)\right)dz\\
 & =1-\int_{-\infty}^{\infty}\Phi\left(x\right)\frac{1}{\rho}\phi\left(\frac{x-\delta}{\rho}\right)dx\\
 & =1-\int_{-\infty}^{\infty}\int_{-\infty}^{x}f\left(x\right)\phi\left(y\right)dy\,dx,
\end{align}
where $f\left(x\right)\coloneqq\frac{1}{\rho}\phi\left(\frac{x-\delta}{\rho}\right)$
is the density function for $\mathcal{N}\left(\delta,\rho^{2}\right)$.
This implies that 
\begin{equation}
\mathbb{E}Z_{\rho,\delta}=1-\Pr\left(\mathbf{X}\in\mathcal{R}\right)=\Pr\left(\mathbf{X}\in\mathcal{R}^{\mathsf{c}}\right),\label{eq:normalProb}
\end{equation}
where $\mathcal{R}^{\mathsf{c}}\coloneqq\left\{ \left(x,y\right)^{\mathsf{T}}\in\mathbb{R}^{2}:x\leq y\right\} $
and 
\begin{equation}
\mathbf{X}\sim\mathcal{N}_{2}\left(\left(\begin{array}{c}
\delta\\
0
\end{array}\right),\left(\begin{array}{cc}
\rho^{2} & 0\\
0 & 1
\end{array}\right)\right).
\end{equation}
Transforming $\mathbb{R}^{2}$ makes probability (\ref{eq:normalProb})
easier to calculate. To that end, while
\begin{equation}
\mathbf{R}\coloneqq\left[\begin{array}{rr}
\cos\left(\frac{\pi}{4}\right) & -\sin\left(\frac{\pi}{4}\right)\\
\sin\left(\frac{\pi}{4}\right) & \cos\left(\frac{\pi}{4}\right)
\end{array}\right]=\frac{1}{\sqrt{2}}\left[\begin{array}{rr}
1 & -1\\
1 & 1
\end{array}\right]
\end{equation}
rotates $\mathbb{R}^{2}$ counterclockwise by $\frac{\pi}{4}\textrm{ radians}=45^{\circ}$,
we have
\begin{align}
\mathbf{R}\mathcal{R}^{c} & \coloneqq\left\{ \mathbf{R}\mathbf{x}:\mathbf{x}\in\mathcal{R}^{c}\right\} =\left\{ \left(x,y\right)^{\mathsf{T}}\in\mathbb{R}^{2}:x\leq0\right\} ,\label{eq:rotatedSet}\\
\mathbf{Y} & \coloneqq\mathbf{R}\mathbf{X}\sim\mathcal{N}_{2}\left(\frac{\delta}{\sqrt{2}}\left(\begin{array}{c}
1\\
1
\end{array}\right),\frac{1}{2}\left(\begin{array}{cc}
\rho^{2}+1 & \rho^{2}-1\\
\rho^{2}-1 & \rho^{2}+1
\end{array}\right)\right).\label{eq:BivNormMean}
\end{align}
Equations (\ref{eq:rotatedSet}) and (\ref{eq:BivNormMean}) then
imply that
\begin{align}
\mathbb{E}Z_{\rho,\delta} & =\Pr\left(\mathbf{X}\in\mathcal{R}^{\mathsf{c}}\right)=\Pr\left(\mathbf{R}\mathbf{X}\in\mathbf{R}\mathcal{R}^{\mathsf{c}}\right)\label{eq:repeatNormProb}\\
 & =\Pr\left(Y_{1}\leq0\right)=\Phi\left(\nicefrac{-\delta}{\sqrt{1+\rho^{2}}}\right),\label{eq:probY1}
\end{align}
where (\ref{eq:repeatNormProb}) uses (\ref{eq:normalProb}) and (\ref{eq:probY1})
uses the marginal $Y_{1}\sim\mathcal{N}\left(\frac{\delta}{\sqrt{2}},\frac{1+\rho^{2}}{2}\right)$.
\end{proof}
\genVar*
\begin{proof}
In what follows let $\mathcal{R}\coloneqq\left\{ \left(x,y,z\right)^{\mathsf{T}}\in\mathbb{R}^{3}:\max\left(x,y\right)\leq z\right\} $.
By Theorem \ref{thm:theGenMean} it suffices to show that
\begin{equation}
\mathbb{E}Z_{\rho,\delta}^{2}=\int_{\nicefrac{11\pi}{12}}^{\nicefrac{19\pi}{12}}G_{\rho,\delta}\left(\theta\right)d\theta+\frac{\cos^{-1}\left(\nicefrac{-1}{\left(\rho^{2}+1\right)}\right)}{2\pi\exp\left(\nicefrac{\delta^{2}}{\left(\rho^{2}+2\right)}\right)}.\label{eq:deriveVPD-1}
\end{equation}
To that end, with independent $X_{0}\sim\mathcal{N}\left(0,1\right)$,
$X_{1},X_{2}\sim\mathcal{N}\left(\delta,\rho^{2}\right)$ (see footnote
\ref{fn:one-case}), note first that, 
\begin{align}
\mathbb{E}Z_{\rho,\delta}^{2} & =\rho\int_{0}^{1}y^{2}\frac{\phi\left(\delta+\rho\Phi^{-1}\left(y\right)\right)}{\phi\left(\Phi^{-1}\left(y\right)\right)}dy=\int_{-\infty}^{\infty}\Phi\left(\frac{z-\delta}{\rho}\right)^{2}\phi\left(z\right)dz\\
 & =\mathbb{E}\left[\Pr\left(\left.X_{1}\leq X_{0},X_{2}\leq X_{0}\right|X_{0}\right)\right]=\Pr\left(X_{1}\leq X_{0},X_{2}\leq X_{0}\right)\label{eq:deriveVPD-2}\\
 & =\int_{-\infty}^{\infty}\int_{-\infty}^{z}\int_{-\infty}^{z}f\left(x\right)f\left(y\right)\phi\left(z\right)dx\,dy\,dz,
\end{align}
where $f\left(x\right)\coloneqq\frac{1}{\rho}\phi\left(\frac{x-\delta}{\rho}\right)$
is the density function for $\mathcal{N}\left(\delta,\rho^{2}\right)$.
This implies that
\begin{equation}
\mathbb{E}Z_{\rho,\delta}^{2}=\Pr\left(\mathbf{X}\in\mathcal{R}\right),\label{eq:prob3space}
\end{equation}
where 
\begin{equation}
\mathbf{X}\sim\mathcal{N}_{3}\left(\left[\begin{array}{c}
\delta\\
\delta\\
0
\end{array}\right],\left[\begin{array}{ccc}
\rho^{2} & 0 & 0\\
0 & \rho^{2} & 0\\
0 & 0 & 1
\end{array}\right]\right).
\end{equation}
Transforming $\mathbb{R}^{3}$ makes probability (\ref{eq:prob3space})
easier to calculate. Namely, we rotate the space so that the spine
$\alpha\left(1,1,1\right)^{\mathsf{T}}$, $\alpha\in\mathbb{R}$,
of wedge $\mathcal{R}$ is vertical, thereby shrinking the problem
from three dimensions to two. Letting $\mathbf{s}\coloneqq\left(1,1,1\right)^{\mathsf{T}}$
and $\mathbf{v}\coloneqq\left(0,0,1\right)^{\mathsf{T}}$, we rotate
$\mathbb{R}^{3}$ by 
\begin{equation}
\cos^{-1}\left(\frac{\mathbf{s}^{\mathsf{T}}\mathbf{v}}{\left\Vert \mathbf{s}\right\Vert \left\Vert \mathbf{v}\right\Vert }\right)=\cos^{-1}\left(\frac{1}{\sqrt{3}}\right)\textrm{ radians}
\end{equation}
about unit axis $\mathbf{u}\coloneqq\left(\nicefrac{1}{\sqrt{2}},\nicefrac{-1}{\sqrt{2}},0\right)^{\mathsf{T}}$
using rotation matrix 
\begin{equation}
\mathbf{R}\coloneqq\frac{1}{6}\left[\begin{array}{rrr}
\sqrt{3}+3 & \sqrt{3}-3 & -2\sqrt{3}\\
\sqrt{3}-3 & \sqrt{3}+3 & -2\sqrt{3}\\
2\sqrt{3} & 2\sqrt{3} & 2\sqrt{3}
\end{array}\right]
\end{equation}
(see Equation 9.63 in \citet{C15}). We then have
\begin{equation}
\mathbf{Y}\coloneqq\mathbf{R}\mathbf{X}\sim\mathcal{N}_{3}\left(\frac{1}{\sqrt{3}}\left[\begin{array}{r}
\delta\\
\delta\\
2\delta
\end{array}\right],\frac{1}{3}\left[\begin{array}{rrr}
1+2\rho^{2} & 1-\rho^{2} & -1+\rho^{2}\\
1-\rho^{2} & 1+2\rho^{2} & -1+\rho^{2}\\
-1+\rho^{2} & -1+\rho^{2} & 1+2\rho^{2}
\end{array}\right]\right).
\end{equation}
Furthermore, with 
\begin{equation}
\left[\begin{array}{c}
a\\
b\\
c
\end{array}\right]\coloneqq\mathbf{R}\left[\begin{array}{r}
1\\
0\\
0
\end{array}\right]=\frac{1}{6}\left[\begin{array}{r}
3+\sqrt{3}\\
-3+\sqrt{3}\\
2\sqrt{3}
\end{array}\right]\textrm{ and }\mathbf{R}\left[\begin{array}{r}
0\\
1\\
0
\end{array}\right]=\left[\begin{array}{c}
b\\
a\\
c
\end{array}\right],
\end{equation}
we note that
\begin{align}
\mathbf{R}\mathcal{R} & =\left\{ \left(x,y,z\right)^{\mathsf{T}}\in\mathbb{R}^{3}:y\leq\min\left(\frac{ax}{b},\frac{bx}{a}\right)\right\} \label{eq:xyzLimits}\\
 & =\left\{ \left(r,\theta,z\right)^{\mathsf{T}}\in\left[0,\infty\right)\times\left[0,2\pi\right)\times\mathbb{R}:\frac{11\pi}{12}\leq\theta\leq\frac{19\pi}{12}\right\} ,\label{eq:rthetazLimits}
\end{align}
where the product $\mathbf{R}\mathcal{R}$ is defined in (\ref{eq:rotatedSet})
and (\ref{eq:rthetazLimits}) uses polar coordinates for the $xy$-plane.
Now, (\ref{eq:xyzLimits}) implies we need only consider $\left(Y_{1},Y_{2}\right)^{\mathsf{T}}$
which has marginal distribution 
\begin{equation}
\left[\begin{array}{r}
Y_{1}\\
Y_{2}
\end{array}\right]\sim\mathcal{N}_{2}\left(\frac{1}{\sqrt{3}}\left[\begin{array}{c}
\delta\\
\delta
\end{array}\right],\frac{1}{3}\left[\begin{array}{rr}
1+2\rho^{2} & 1-\rho^{2}\\
1-\rho^{2} & 1+2\rho^{2}
\end{array}\right]\right)\eqqcolon\mathcal{N}_{2}\left(\boldsymbol{\mu},\boldsymbol{\Sigma}\right)
\end{equation}
and density function
\begin{equation}
g\left(\mathbf{y}\right)\coloneqq\frac{\exp\left(-\frac{1}{2}\left(\mathbf{y}-\boldsymbol{\mu}\right)^{\mathsf{T}}\boldsymbol{\Sigma}^{-1}\left(\mathbf{y}-\boldsymbol{\mu}\right)\right)}{2\pi\sqrt{\left|\boldsymbol{\Sigma}\right|}},
\end{equation}
where $\left|\boldsymbol{\Sigma}\right|=\frac{1}{3}\rho^{2}\left(\rho^{2}+2\right)$
is the determinant of $\boldsymbol{\Sigma}$ and $\mathbf{y}\in\mathbb{R}^{2}$.
Picking up from (\ref{eq:prob3space}), we have $\mathbb{E}Z_{\rho,\delta}^{2}=\Pr\left(\mathbf{X}\in\mathcal{R}\right)=\Pr\left(\mathbf{R}\mathbf{X}\in\mathbf{R}\mathcal{R}\right)$
\begin{align}
 & =\Pr\left(Y_{2}\leq\min\left(\frac{aY_{1}}{b},\frac{bY_{1}}{a}\right)\right)\\
 & =\underset{y_{2}\leq\min\left(\frac{ay_{1}}{b},\frac{by_{1}}{a}\right)}{\int\int}g\left(\mathbf{y}\right)dy_{1}dy_{2}\\
 & =\int_{\nicefrac{11\pi}{12}}^{\nicefrac{19\pi}{12}}\int_{0}^{\infty}g\left(r\cos\theta,r\sin\theta\right)r\,dr\,d\theta\\
 & =\frac{\sqrt{3}e^{-\frac{\delta^{2}}{\rho^{2}+2}}}{2\pi\rho\sqrt{\rho^{2}+2}}\int_{\nicefrac{11\pi}{12}}^{\nicefrac{19\pi}{12}}\int_{0}^{\infty}\exp\left(B_{\rho,\delta}\left(\theta\right)r-A_{\rho}\left(\theta\right)r^{2}\right)r\,dr\,d\theta\\
 & =\frac{\sqrt{3}e^{-\frac{\delta^{2}}{\rho^{2}+2}}}{2\pi\rho\sqrt{\rho^{2}+2}}\int_{\nicefrac{11\pi}{12}}^{\nicefrac{19\pi}{12}}\frac{1}{2A_{\rho}\left(\theta\right)}\left[\frac{B_{\rho,\delta}\left(\theta\right)}{\sqrt{2A_{\rho}\left(\theta\right)}}\frac{\Phi\left(\frac{B_{\rho,\delta}\left(\theta\right)}{\sqrt{2A_{\rho}\left(\theta\right)}}\right)}{\phi\left(\frac{B_{\rho,\delta}\left(\theta\right)}{\sqrt{2A_{\rho}\left(\theta\right)}}\right)}+1\right]d\theta,\label{eq:varClose}
\end{align}
with $B_{\rho,\delta}\left(\theta\right)$ and $A_{\rho}\left(\theta\right)>0$
as in (\ref{eq:B}) and (\ref{eq:A}). We finally have $\int_{\nicefrac{11\pi}{12}}^{\nicefrac{19\pi}{12}}\frac{d\theta}{2A_{\rho}\left(\theta\right)}$
\begin{align}
 & =\int_{\nicefrac{11\pi}{12}}^{\nicefrac{19\pi}{12}}\frac{\rho^{2}\left(\rho^{2}+2\right)d\theta}{\rho^{2}\left(\sin\left(2\theta\right)+2\right)+2\cos^{2}\left(\theta+\nicefrac{\pi}{4}\right)}\label{eq:varFinalIntegral}\\
 & =\rho\sqrt{\frac{\rho^{2}+2}{3}}\left[\pi-\tan^{-1}\left(\frac{\left(2+\sqrt{3}\right)\rho^{2}+\sqrt{3}+1}{\rho\sqrt{\rho^{2}+2}}\right)\right.\\
 & \hphantom{aaaaaaaaaaaaaaaaaaaaaaaaa}\left.-\tan^{-1}\left(\frac{\left(2-\sqrt{3}\right)\rho^{2}-\sqrt{3}+1}{\rho\sqrt{\rho^{2}+2}}\right)\right]\\
 & =\rho\sqrt{\frac{\rho^{2}+2}{3}}\left[\pi-\tan^{-1}\left(\rho\sqrt{\rho^{2}+2}\right)\right]=\rho\sqrt{\frac{\rho^{2}+2}{3}}\cos^{-1}\left(-\frac{1}{\rho^{2}+1}\right),\label{eq:varConcl}
\end{align}
where the first part of (\ref{eq:varConcl}) uses the identity $\tan^{-1}u+\tan^{-1}v=\tan^{-1}\frac{u+v}{1-uv}\mod\pi$,
when $uv\ne1$, and the second part uses basic trigonometry. Substituting
(\ref{eq:varConcl}) into (\ref{eq:varClose}) gives (\ref{eq:varZrhodelta})
and (\ref{eq:deriveVPD-1}) to (\ref{eq:deriveVPD-2}) gives (\ref{eq:varProbDiff}),
completing the proof.
\end{proof}

\section{Proofs of Theorems \ref{thm:limitZ} and \ref{thm:limitX} \label{sec:Derivations-for-=0000A72.3}}

\limitZ*
\begin{proof}
(\ref{eq:Zto1}) and (\ref{eq:Zto0}) follow from the triangle inequality,
Chebyshev's inequality, and (\ref{eq:varBounds}). Fixing $0<\epsilon<1$
and $\delta$ small (large) enough in (\ref{eq:proofZto1}) ((\ref{eq:proofZto0})),
we have 
\begin{align}
\Pr\left(\left|Z_{\rho,\delta}-1\right|>\epsilon\right) & \leq\frac{\mathrm{Var}\left(Z_{\rho,\delta}\right)}{\left(1-\epsilon-\Phi\left(\nicefrac{-\delta}{\sqrt{\rho^{2}+1}}\right)\right)^{2}}\stackrel[-\infty]{\delta}{\longrightarrow}0\label{eq:proofZto1}\\
\Pr\left(\left|Z_{\rho,\delta}\right|>\epsilon\right) & \leq\frac{\mathrm{Var}\left(Z_{\rho,\delta}\right)}{\left(\epsilon-\Phi\left(\nicefrac{-\delta}{\sqrt{\rho^{2}+1}}\right)\right)^{2}}\stackrel[\infty]{\delta}{\longrightarrow}0.\label{eq:proofZto0}
\end{align}

We have (\ref{eq:ZtoBern}) if we can show that $\lim_{\epsilon\rightarrow0^{+}}\lim_{\rho\rightarrow0^{+}}\Pr\left(Z_{\rho,\delta}\leq\epsilon\right)=\Phi\left(\delta\right)$
and $\lim_{\epsilon\rightarrow1^{-}}\lim_{\rho\rightarrow0^{+}}\Pr\left(Z_{\rho,\delta}\geq\epsilon\right)=\Phi\left(-\delta\right)$.
To that end note that
\begin{align}
\lim_{\epsilon\rightarrow0^{+}}\lim_{\rho\rightarrow0^{+}}\Pr\left(Z_{\rho,\delta}\leq\epsilon\right) & =\lim_{\epsilon\rightarrow0^{+}}\lim_{\rho\rightarrow0^{+}}\Phi\left(\delta+\rho\Phi^{-1}\left(\epsilon\right)\right)\\
 & =\lim_{\epsilon\rightarrow0^{+}}\Phi\left(\delta\right)=\Phi\left(\delta\right)\label{eq:contFinite0}
\end{align}
and that
\begin{align}
\lim_{\epsilon\rightarrow1^{-}}\lim_{\rho\rightarrow0^{+}}\Pr\left(Z_{\rho,\delta}\geq\epsilon\right) & =\lim_{\epsilon\rightarrow1^{-}}\lim_{\rho\rightarrow0^{+}}\Pr\left(Z_{\rho,\delta}>\epsilon\right)\label{eq:continuity}\\
 & =\lim_{\epsilon\rightarrow1^{-}}\lim_{\rho\rightarrow0^{+}}\left\{ 1-\Phi\left(\delta+\rho\Phi^{-1}\left(\epsilon\right)\right)\right\} \\
 & =\lim_{\epsilon\rightarrow1^{-}}\left\{ 1-\Phi\left(\delta\right)\right\} =\Phi\left(-\delta\right),\label{eq:contFinite1}
\end{align}
where (\ref{eq:continuity}) uses the continuity of $Z_{\rho,\delta}$,
and (\ref{eq:contFinite0}) and (\ref{eq:contFinite1}) use $\left|\Phi^{-1}\left(\epsilon\right)\right|<\infty$
for $\epsilon\in\left(0,1\right)$ and the continuity of $\Phi$. 

We obtain (\ref{eq:stanZtoN1}) by expanding $\Phi^{-1}\left(x\right)$
about $\Phi\left(\nicefrac{-\delta}{\sqrt{\rho^{2}+1}}\right)$, which
gives
\begin{align}
\Phi^{-1}\left(x\right) & =\Phi^{-1}\left(\Phi\left(\nicefrac{-\delta}{\sqrt{\rho^{2}+1}}\right)\right)+\frac{x-\Phi\left(\nicefrac{-\delta}{\sqrt{\rho^{2}+1}}\right)}{\phi\left(\Phi^{-1}\left(x^{*}\right)\right)}\\
 & =-\frac{\delta}{\sqrt{\rho^{2}+1}}+\frac{x-\Phi\left(\nicefrac{-\delta}{\sqrt{\rho^{2}+1}}\right)}{\phi\left(\Phi^{-1}\left(x^{*}\right)\right)},\label{eq:taylor}
\end{align}
where $x^{*}$ is between $x$ and $\Phi\left(\nicefrac{-\delta}{\sqrt{\rho^{2}+1}}\right)$.
Now, fixing $y\in\mathbb{R}$, we have 
\begin{align}
\Pr\left(\rho\left(Z_{\rho,\delta}-\Phi\left(\nicefrac{-\delta}{\sqrt{\rho^{2}+1}}\right)\right)\leq y\right) & =\Pr\left(Z_{\rho,\delta}\leq\Phi\left(\nicefrac{-\delta}{\sqrt{\rho^{2}+1}}\right)+\nicefrac{y}{\rho}\right)\\
 & =\Phi\left(\delta+\rho\Phi^{-1}\left(\Phi\left(\nicefrac{-\delta}{\sqrt{\rho^{2}+1}}\right)+\nicefrac{y}{\rho}\right)\right)\\
 & =\Phi\left(\delta\left(1-\nicefrac{\rho}{\sqrt{\rho^{2}+1}}\right)+\nicefrac{y}{\phi\left(\Phi^{-1}\left(y^{*}\right)\right)}\right)\label{eq:useTaylor}\\
 & \stackrel{\rho\rightarrow\infty}{\longrightarrow}\Phi\left(\sqrt{2\pi}y\right),\label{eq:conv2normal}
\end{align}
where (\ref{eq:useTaylor}) uses (\ref{eq:taylor}) with $y^{*}$
between $\Phi\left(\nicefrac{-\delta}{\sqrt{\rho^{2}+1}}\right)+\nicefrac{y}{\rho}$
and $\Phi\left(\nicefrac{-\delta}{\sqrt{\rho^{2}+1}}\right)$, and
(\ref{eq:conv2normal}) follows because $y^{*}\rightarrow\nicefrac{1}{2}$
as $\rho\rightarrow\infty$.

We obtain (\ref{eq:stanZtoN2}) by noting that 
\begin{align}
\Pr\left(\rho\left(Z_{\rho,\delta}-\Phi\left(\nicefrac{-\delta}{\sqrt{\rho^{2}+1}}\right)\right)\leq y\right) & =\Pr\left(Z_{\rho,\delta}\leq\Phi\left(\nicefrac{-\delta}{\sqrt{\rho^{2}+1}}\right)+\nicefrac{y}{\rho}\right)\\
 & =\Phi\left(\delta+\rho\Phi^{-1}\left(\Phi\left(\nicefrac{-\delta}{\sqrt{\rho^{2}+1}}\right)+\nicefrac{y}{\rho}\right)\right)\\
 & =\Phi\left(\delta\left(1-\nicefrac{\rho}{\sqrt{\rho^{2}+1}}\right)+\nicefrac{y}{\phi\left(\Phi^{-1}\left(y^{*}\right)\right)}\right)\label{eq:useTaylorRatio}\\
 & =\Phi\left(r\left(\rho-\nicefrac{\rho^{2}}{\sqrt{\rho^{2}+1}}\right)+\nicefrac{y}{\phi\left(\Phi^{-1}\left(y^{*}\right)\right)}\right)\\
 & \stackrel{\rho,\left|\delta\right|\rightarrow\infty}{\longrightarrow}\Phi\left(\nicefrac{y}{\phi\left(-r\right)}\right)=\Phi\left(\nicefrac{y}{\phi\left(r\right)}\right),\label{eq:conv2normalRatio}
\end{align}
where (\ref{eq:useTaylorRatio}) uses (\ref{eq:taylor}) with $y^{*}$
between $\Phi\left(\nicefrac{-\delta}{\sqrt{\rho^{2}+1}}\right)+\nicefrac{y}{\rho}$
and $\Phi\left(\nicefrac{-\delta}{\sqrt{\rho^{2}+1}}\right)$ and
(\ref{eq:conv2normalRatio}) follows because $y^{*}$ approaches $\Phi\left(-r\right)$
as $\rho,\left|\delta\right|\rightarrow\infty$ with $\nicefrac{\delta}{\rho}=r$
fixed.

Finally, (\ref{eq:ZtoHalf}) and (\ref{eq:ZtoPhi}) follow from (\ref{eq:stanZtoN1})
and (\ref{eq:stanZtoN2}), the triangle inequality, and Chebyshev's
inequality. For sufficiently large $\rho$ (\ref{eq:proofZtoHalf})
or $\rho,\left|\delta\right|$ (\ref{eq:proofZtoPhi}), we have 
\begin{align}
\Pr\left(\left|Z_{\rho,\delta}-\nicefrac{1}{2}\right|>\epsilon\right) & \leq\frac{\mathrm{Var}\left(Z_{\rho,\delta}\right)}{\left(\epsilon-\left|\nicefrac{1}{2}-\Phi\left(\nicefrac{-\delta}{\sqrt{\rho^{2}+1}}\right)\right|\right)^{2}}\stackrel[\infty]{\rho}{\longrightarrow}0\label{eq:proofZtoHalf}\\
\Pr\left(\left|Z_{\rho,\delta}-\Phi\left(-r\right)\right|>\epsilon\right) & \leq\frac{\mathrm{Var}\left(Z_{\rho,\delta}\right)}{\left(\epsilon-\left|\Phi\left(-r\right)-\Phi\left(\nicefrac{-\delta}{\sqrt{\rho^{2}+1}}\right)\right|\right)^{2}}\stackrel[\infty]{\rho,\left|\delta\right|}{\longrightarrow}0.\label{eq:proofZtoPhi}
\end{align}
\end{proof}
We state and prove two auxiliary propositions that underpin Theorem
\ref{thm:limitX}. 
\begin{prop}
\label{prop:limitab}Mapping functions $a_{\rho,\delta},b_{\rho,\delta}$
converge to the following limits:
\[
\begin{array}{rccccc}
 & \rho\rightarrow0^{+} & \rho\rightarrow\infty & \delta\rightarrow-\infty & \delta\rightarrow\infty & \rho,\left|\delta\right|\rightarrow\infty\\
a_{\rho,\delta}: & 0 & \infty & \infty & 0 & \infty\\
b_{\rho,\delta}: & 0 & \infty & 0 & \infty & \infty
\end{array}
\]
where the limits with $\rho,\left|\delta\right|\rightarrow\infty$
keep $\nicefrac{\delta}{\rho}=r$ fixed.
\end{prop}

\begin{proof}
While results for $\rho\rightarrow0^{+}$, $\rho\rightarrow\infty$,
and $\rho,\left|\delta\right|\rightarrow\infty$ ($\nicefrac{\delta}{\rho}=r$
fixed) follow directly from Theorem \ref{thm:limitZ}, we focus on
$\left|\delta\right|\rightarrow\infty$, starting with technical results. 

Putting $H_{\rho,\delta}\left(\theta\right)\coloneqq G_{\rho,\delta}\left(\theta\right)\left/\Phi\left(\nicefrac{B_{\rho,\delta}\left(\theta\right)}{\sqrt{2A_{\rho}\left(\theta\right)}}\right)\right.$,
we start by showing that 
\begin{align}
\int_{\nicefrac{11\pi}{12}}^{\nicefrac{19\pi}{12}}H_{\rho,\delta}\left(\theta\right)d\theta & =1-2\Phi\left(\nicefrac{\delta}{\sqrt{\rho^{2}+1}}\right),\textrm{ so that}\label{eq:integralH}\\
\int_{\nicefrac{11\pi}{12}}^{\nicefrac{19\pi}{12}}G_{\rho,\delta}\left(\theta\right)d\theta & \leq\Phi\left(\nicefrac{-\sqrt{2}\delta}{\sqrt{\rho^{2}+2}}\right)\left[1-2\Phi\left(\nicefrac{\delta}{\sqrt{\rho^{2}+1}}\right)\right],\label{eq:intUpperBnd}
\end{align}
for $G_{\rho,\delta}$ in (\ref{eq:G}). $\mathrm{Var}\left(Z_{\rho,\delta}\right)=\mathrm{Var}\left(Z_{\rho,-\delta}\right)$
in (\ref{eq:expEquation}) gives (\ref{eq:integralH}). Putting 
\begin{equation}
\tau_{1}\left(\delta\right)\coloneqq\int_{\nicefrac{11\pi}{12}}^{\nicefrac{19\pi}{12}}G_{\rho,\delta}\left(\theta\right)d\theta,\label{eq:tau1}
\end{equation}
\begin{equation}
\tau_{2}\left(\delta\right)\coloneqq\frac{\cos^{-1}\left(\nicefrac{-1}{\left(\rho^{2}+1\right)}\right)}{2\pi\exp\left(\nicefrac{\delta^{2}}{\left(\rho^{2}+2\right)}\right)},\quad\tau_{3}\left(\delta\right)\coloneqq\Phi\left(\nicefrac{-\delta}{\sqrt{\rho^{2}+1}}\right)^{2},\label{eq:tau23}
\end{equation}
we have $\tau_{1}\left(\delta\right)+\tau_{2}\left(\delta\right)-\tau_{3}\left(\delta\right)=\tau_{1}\left(-\delta\right)+\tau_{2}\left(-\delta\right)-\tau_{3}\left(-\delta\right)$
as in (\ref{eq:varZrhodelta}). Now, 
\begin{align}
\tau_{1}\left(-\delta\right) & =\frac{\sqrt{3}e^{-\frac{\delta^{2}}{\rho^{2}+2}}}{2\pi\rho\sqrt{\rho^{2}+2}}\int_{\nicefrac{11\pi}{12}}^{\nicefrac{19\pi}{12}}\frac{B_{\rho,-\delta}\left(\theta\right)}{\left[2A_{\rho}\left(\theta\right)\right]^{\nicefrac{3}{2}}}\frac{\Phi\left(\frac{B_{\rho,-\delta}\left(\theta\right)}{\sqrt{2A_{\rho}\left(\theta\right)}}\right)}{\phi\left(\frac{B_{\rho,-\delta}\left(\theta\right)}{\sqrt{2A_{\rho}\left(\theta\right)}}\right)}d\theta\\
 & =-\frac{\sqrt{3}e^{-\frac{\delta^{2}}{\rho^{2}+2}}}{2\pi\rho\sqrt{\rho^{2}+2}}\int_{\nicefrac{11\pi}{12}}^{\nicefrac{19\pi}{12}}\frac{B_{\rho,\delta}\left(\theta\right)}{\left[2A_{\rho}\left(\theta\right)\right]^{\nicefrac{3}{2}}}\frac{\Phi\left(-\frac{B_{\rho,\delta}\left(\theta\right)}{\sqrt{2A_{\rho}\left(\theta\right)}}\right)}{\phi\left(-\frac{B_{\rho,\delta}\left(\theta\right)}{\sqrt{2A_{\rho}\left(\theta\right)}}\right)}d\theta\\
 & =-\frac{\sqrt{3}e^{-\frac{\delta^{2}}{\rho^{2}+2}}}{2\pi\rho\sqrt{\rho^{2}+2}}\int_{\nicefrac{11\pi}{12}}^{\nicefrac{19\pi}{12}}\frac{B_{\rho,\delta}\left(\theta\right)}{\left[2A_{\rho}\left(\theta\right)\right]^{\nicefrac{3}{2}}}\frac{\left[1-\Phi\left(\frac{B_{\rho,\delta}\left(\theta\right)}{\sqrt{2A_{\rho}\left(\theta\right)}}\right)\right]}{\phi\left(\frac{B_{\rho,\delta}\left(\theta\right)}{\sqrt{2A_{\rho}\left(\theta\right)}}\right)}d\theta\\
 & =\tau_{1}\left(\delta\right)-\int_{\nicefrac{11\pi}{12}}^{\nicefrac{19\pi}{12}}H_{\rho,\delta}\left(\theta\right)d\theta
\end{align}
because $B_{\rho,-\delta}\left(\theta\right)=-B_{\rho,\delta}\left(\theta\right)$.
Note that $\tau_{2}\left(-\delta\right)=\tau_{2}\left(\delta\right)$.
We finally have
\begin{align}
\tau_{3}\left(-\delta\right) & =\Phi\left(\nicefrac{\delta}{\sqrt{\rho^{2}+1}}\right)^{2}=\left[1-\Phi\left(\nicefrac{-\delta}{\sqrt{\rho^{2}+1}}\right)\right]^{2}\\
 & =1-2\Phi\left(\nicefrac{-\delta}{\sqrt{\rho^{2}+1}}\right)+\tau_{3}\left(\delta\right)=-1+2\Phi\left(\nicefrac{\delta}{\sqrt{\rho^{2}+1}}\right)+\tau_{3}\left(\delta\right).
\end{align}
Putting this all together with $\tau_{1}\left(\delta\right)+\tau_{2}\left(\delta\right)-\tau_{3}\left(\delta\right)=\tau_{1}\left(-\delta\right)+\tau_{2}\left(-\delta\right)-\tau_{3}\left(-\delta\right)$
gives (\ref{eq:integralH}).

Turning to (\ref{eq:intUpperBnd}), we first consider the case $\delta>0$.
Note first that $G_{\rho,\delta}\left(\theta\right)<0$ for $\delta>0$
and $\nicefrac{11\pi}{12}<\theta<\nicefrac{19\pi}{12}$. This follows
from the definitions of $B_{\rho,\delta}\left(\theta\right)$ and
$G_{\rho,\delta}\left(\theta\right)$ in (\ref{eq:B}) and (\ref{eq:G}),
namely $\sin\left(\theta+\nicefrac{\pi}{4}\right)<0$ for $\nicefrac{11\pi}{12}<\theta<\nicefrac{19\pi}{12}$,
and implies that
\begin{align}
\int_{\nicefrac{11\pi}{12}}^{\nicefrac{19\pi}{12}}G_{\rho,\delta}\left(\theta\right)d\theta & \leq\int_{\nicefrac{11\pi}{12}}^{\nicefrac{19\pi}{12}}H_{\rho,\delta}\left(\theta\right)d\theta\min_{\frac{11\pi}{12}\leq\theta\leq\frac{19\pi}{12}}\Phi\left(\nicefrac{B_{\rho,\delta}\left(\theta\right)}{\sqrt{2A_{\rho}\left(\theta\right)}}\right)\\
 & =\left[1-2\Phi\left(\nicefrac{\delta}{\sqrt{\rho^{2}+1}}\right)\right]\Phi\left(\nicefrac{B_{\rho,\delta}\left(\frac{5\pi}{4}\right)}{\sqrt{2A_{\rho}\left(\frac{5\pi}{4}\right)}}\right)\label{eq:useIHlem1}\\
 & =\left[1-2\Phi\left(\nicefrac{\delta}{\sqrt{\rho^{2}+1}}\right)\right]\Phi\left(\nicefrac{-\sqrt{2}\delta}{\sqrt{\rho^{2}+2}}\right),
\end{align}
where (\ref{eq:useIHlem1}) uses (\ref{eq:integralH}). We now turn
to the case $\delta<0$. Note that $B_{\rho,\delta}\left(\theta\right)>0$
when $\delta<0$ and $\nicefrac{11\pi}{12}\leq\theta\leq\nicefrac{19\pi}{12}$
implies that $G_{\rho,\delta}\left(\theta\right)>0$ when $\delta<0$
and $\nicefrac{11\pi}{12}\leq\theta\leq\nicefrac{19\pi}{12}$. Following
a path similar to that above, we have
\begin{align}
\int_{\nicefrac{11\pi}{12}}^{\nicefrac{19\pi}{12}}G_{\rho,\delta}\left(\theta\right)d\theta & \leq\int_{\nicefrac{11\pi}{12}}^{\nicefrac{19\pi}{12}}H_{\rho,\delta}\left(\theta\right)d\theta\max_{\frac{11\pi}{12}\leq\theta\leq\frac{19\pi}{12}}\Phi\left(\nicefrac{B_{\rho,\delta}\left(\theta\right)}{\sqrt{2A_{\rho}\left(\theta\right)}}\right)\\
 & =\left[1-2\Phi\left(\nicefrac{\delta}{\sqrt{\rho^{2}+1}}\right)\right]\Phi\left(\nicefrac{B_{\rho,\delta}\left(\frac{5\pi}{4}\right)}{\sqrt{2A_{\rho}\left(\frac{5\pi}{4}\right)}}\right)\label{eq:useIHlem3}\\
 & =\left[1-2\Phi\left(\nicefrac{\delta}{\sqrt{\rho^{2}+1}}\right)\right]\Phi\left(\nicefrac{-\sqrt{2}\delta}{\sqrt{\rho^{2}+2}}\right),
\end{align}
where (\ref{eq:useIHlem3}) uses (\ref{eq:integralH}). We finally
note that $G_{\rho,0}\left(\theta\right)=1-2\Phi\left(\nicefrac{0}{\sqrt{\rho^{2}+1}}\right)=0$,
so that the bound works when $\delta=0$, giving (\ref{eq:intUpperBnd}).

We are now ready to prove that $\lim_{\delta\rightarrow\infty}a_{\rho,\delta}=0$.
From (\ref{eq:alpha3}) we note that 
\begin{equation}
\lim_{\delta\rightarrow\infty}a_{\rho,\delta}=\lim_{\delta\rightarrow\infty}\frac{\Phi\left(\nicefrac{-\delta}{\sqrt{\rho^{2}+1}}\right)^{2}}{\mathrm{Var}\left(Z_{\rho,\delta}\right)}
\end{equation}
because $\Phi\left(\nicefrac{-\delta}{\sqrt{\rho^{2}+1}}\right)\stackrel[\infty]{\delta}{\longrightarrow}0$
and $\Phi\left(\nicefrac{\delta}{\sqrt{\rho^{2}+1}}\right)\stackrel[\infty]{\delta}{\longrightarrow}1$.
Now, if we can show that
\begin{equation}
\lim_{\delta\rightarrow\infty}\left(\frac{\Phi\left(\nicefrac{-\delta}{\sqrt{\rho^{2}+1}}\right)^{2}}{\mathrm{Var}\left(Z_{\rho,\delta}\right)}\right)^{-1}=\lim_{\delta\rightarrow\infty}\frac{\mathrm{Var}\left(Z_{\rho,\delta}\right)}{\Phi\left(\nicefrac{-\delta}{\sqrt{\rho^{2}+1}}\right)^{2}}=\infty,
\end{equation}
we are done. That is to say, by Theorem \ref{thm:theGenVar} (note
especially (\ref{eq:varClose}) of the proof) it is enough to show
that
\begin{equation}
\frac{c_{1,\rho}e^{\frac{-\delta^{2}}{\rho^{2}+2}}}{\Phi\left(\nicefrac{-\delta}{\sqrt{\rho^{2}+1}}\right)^{2}}\int_{\nicefrac{11\pi}{12}}^{\nicefrac{19\pi}{12}}\frac{1}{2A_{\rho}\left(\theta\right)}\left[1+\frac{B_{\rho,\delta}\left(\theta\right)}{\sqrt{2A_{\rho}\left(\theta\right)}}\frac{\Phi\left(\frac{B_{\rho,\delta}\left(\theta\right)}{\sqrt{2A_{\rho}\left(\theta\right)}}\right)}{\phi\left(\frac{B_{\rho,\delta}\left(\theta\right)}{\sqrt{2A_{\rho}\left(\theta\right)}}\right)}\right]d\theta\stackrel[\infty]{\delta}{\rightarrow}\infty,\label{eq:intToInf}
\end{equation}
where $c_{1,\rho}\coloneqq\frac{1}{2\pi\rho}\sqrt{\frac{3}{\rho^{2}+2}}$.
To see (\ref{eq:intToInf}) we note that the dominated convergence
theorem (DCT) allows us to bring the limit (and any terms that depend
on $\delta$) under the integral sign. To see that the DCT applies,
note first that 
\begin{equation}
0<\frac{1}{2A_{\rho}\left(\theta\right)}=\frac{\rho^{2}\left(\rho^{2}+2\right)}{\rho^{2}\left(\sin\left(2\theta\right)+2\right)+2\cos^{2}\left(\theta+\nicefrac{\pi}{4}\right)}\leq\nicefrac{2}{3}\left(\rho^{2}+2\right)
\end{equation}
because $\sin\left(2\theta\right)+2\in\left[\nicefrac{3}{2},3\right]$
and $\cos^{2}\left(\theta+\nicefrac{\pi}{4}\right)\in\left[0,\nicefrac{3}{4}\right]$
when $\theta\in\left[\nicefrac{11\pi}{12},\nicefrac{19\pi}{12}\right]$.
Noting then that $B_{\rho,\delta}\left(\theta\right)<0$ when $\delta>0$
(see the proof of (\ref{eq:intUpperBnd})), we have 
\begin{align}
\left|\frac{c_{1,\rho}}{e^{\frac{\delta^{2}}{\rho^{2}+2}}}\frac{B_{\rho,\delta}\left(\theta\right)}{\sqrt{2A_{\rho}\left(\theta\right)}}\frac{\Phi\left(\frac{B_{\rho,\delta}\left(\theta\right)}{\sqrt{2A_{\rho}\left(\theta\right)}}\right)}{\phi\left(\frac{B_{\rho,\delta}\left(\theta\right)}{\sqrt{2A_{\rho}\left(\theta\right)}}\right)}\right| & =-\frac{c_{1,\rho}}{e^{\frac{\delta^{2}}{\rho^{2}+2}}}\frac{B_{\rho,\delta}\left(\theta\right)}{\sqrt{2A_{\rho}\left(\theta\right)}}\frac{\Phi\left(\frac{B_{\rho,\delta}\left(\theta\right)}{\sqrt{2A_{\rho}\left(\theta\right)}}\right)}{\phi\left(\frac{B_{\rho,\delta}\left(\theta\right)}{\sqrt{2A_{\rho}\left(\theta\right)}}\right)}\\
 & \leq-H_{\rho,\delta}\left(\theta\right)\Phi\left(\nicefrac{-\sqrt{2}\delta}{\sqrt{\rho^{2}+2}}\right),\label{eq:twoTerms}
\end{align}
where the terms in (\ref{eq:twoTerms}) come from the proof of (\ref{eq:intUpperBnd}).
Using (\ref{eq:integralH}) we then have 
\begin{equation}
-\int_{\nicefrac{11\pi}{12}}^{\nicefrac{19\pi}{12}}H_{\rho,\delta}\left(\theta\right)d\theta=2\Phi\left(\nicefrac{\delta}{\sqrt{\rho^{2}+1}}\right)-1\leq1,
\end{equation}
so that the DCT applies to the left-hand side of (\ref{eq:intToInf}).
Ignoring terms that do not depend on $\delta$ (both are positive)
and focussing on the limit of the integrand, we now show that
\begin{equation}
\lim_{\delta\rightarrow\infty}\frac{\phi\left(-c_{2,\rho}\delta\right)}{\Phi\left(-c_{3,\rho}\delta\right)^{2}}\left[1-c_{4,\rho}\left(\theta\right)\delta\frac{\Phi\left(-c_{4,\rho}\left(\theta\right)\delta\right)}{\phi\left(-c_{4,\rho}\left(\theta\right)\delta\right)}\right]=\infty,\label{eq:cleanLim}
\end{equation}
for $\nicefrac{11\pi}{12}\leq\theta\leq\nicefrac{19\pi}{12}$, and
\begin{align}
c_{2,\rho} & \coloneqq\sqrt{\nicefrac{2}{\left(\rho^{2}+2\right)}}>0,\qquad c_{3,\rho}\coloneqq\nicefrac{1}{\sqrt{\rho^{2}+1}}>0,\\
c_{4,\rho}\left(\theta\right) & \coloneqq\frac{\sqrt{6}\rho\left|\sin\left(\theta+\nicefrac{\pi}{4}\right)\right|}{\sqrt{\left(\rho^{2}+2\right)\left\{ \rho^{2}\left[\sin\left(2\theta\right)+2\right]+2\cos^{2}\left(\theta+\nicefrac{\pi}{4}\right)\right\} }}>0.
\end{align}
Our proof of (\ref{eq:cleanLim}) uses the following inequality:
\begin{align}
\Phi\left(x\right) & \leq\min\left(-\phi\left(x\right)\left[\nicefrac{1}{x}-\nicefrac{1}{x^{3}}+\nicefrac{3}{x^{5}}\right],\nicefrac{-\phi\left(x\right)}{x}\right)\textrm{ for }x<0,\label{eq:seqIneq}\\
\textrm{\emph{i.e.}, }1-\Phi\left(x\right) & \leq\min\left(\hphantom{-}\phi\left(x\right)\left[\nicefrac{1}{x}-\nicefrac{1}{x^{3}}+\nicefrac{3}{x^{5}}\right],\nicefrac{\hphantom{-}\phi\left(x\right)}{x}\right)\textrm{ for }x>0,\label{eq:fromSmall}
\end{align}
where \S2.3.4 of \citet{S10} derives (\ref{eq:fromSmall}). Plugging
(\ref{eq:seqIneq}) into (\ref{eq:cleanLim}) we have
\begin{align}
\textrm{\eqref{eq:cleanLim}} & \geq\frac{c_{3,\rho}^{2}}{c_{4,\rho}\left(\theta\right)^{2}}\lim_{\delta\rightarrow\infty}\frac{\phi\left(-c_{2,\rho}\delta\right)}{\phi\left(-c_{3,\rho}\delta\right)^{2}}\left[1-\frac{3}{c_{4,\rho}\left(\theta\right)^{2}\delta^{2}}\right]\\
 & =\frac{\sqrt{2\pi}c_{3,\rho}^{2}}{c_{4,\rho}\left(\theta\right)^{2}}\lim_{\delta\rightarrow\infty}\exp\left(\frac{\delta^{2}}{\left(\rho^{2}+1\right)\left(\rho^{2}+2\right)}\right)\left[1-\frac{3}{c_{4,\rho}\left(\theta\right)^{2}\delta^{2}}\right]=\infty,
\end{align}
so that we have shown (\ref{eq:cleanLim}), for $\nicefrac{11\pi}{12}\leq\theta\leq\nicefrac{19\pi}{12}$,
and so (\ref{eq:intToInf}), which implies that $\lim_{\delta\rightarrow\infty}a_{\rho,\delta}=0$.

We now turn to $\lim_{\delta\rightarrow\infty}b_{\rho,\delta}=\infty$.
From (\ref{eq:beta3}) we note that
\begin{equation}
\lim_{\delta\rightarrow\infty}b_{\rho,\delta}=\lim_{\delta\rightarrow\infty}\frac{\Phi\left(\nicefrac{-\delta}{\sqrt{\rho^{2}+1}}\right)}{\mathrm{Var}\left(Z_{\rho,\delta}\right)}-1,\label{eq:calcBeta3Lim-2}
\end{equation}
where (\ref{eq:calcBeta3Lim-2}) uses $\Phi\left(\nicefrac{\delta}{\sqrt{\rho^{2}+1}}\right)\stackrel[\infty]{\delta}{\longrightarrow}1$.
Now, if we can show that 
\begin{equation}
\lim_{\delta\rightarrow\infty}\left(\frac{\Phi\left(\nicefrac{-\delta}{\sqrt{\rho^{2}+1}}\right)}{\mathrm{Var}\left(Z_{\rho,\delta}\right)}\right)^{-1}=\lim_{\delta\rightarrow\infty}\frac{\mathrm{Var}\left(Z_{\rho,\delta}\right)}{\Phi\left(\nicefrac{-\delta}{\sqrt{\rho^{2}+1}}\right)}=0,\label{eq:calcBeta3Lim-2.5}
\end{equation}
we have the desired result. To that end note that
\begin{equation}
\frac{\mathrm{Var}\left(Z_{\rho,\delta}\right)}{\Phi\left(\nicefrac{-\delta}{\sqrt{\rho^{2}+1}}\right)}\leq\frac{\cos^{-1}\left(\nicefrac{-1}{\left(\rho^{2}+1\right)}\right)}{2\pi\exp\left(\nicefrac{\delta^{2}}{\left(\rho^{2}+2\right)}\right)\Phi\left(\nicefrac{-\delta}{\sqrt{\rho^{2}+1}}\right)}-\frac{\Phi\left(\nicefrac{-\sqrt{2}\delta}{\sqrt{\rho^{2}+2}}\right)}{\Phi\left(\nicefrac{-\delta}{\sqrt{\rho^{2}+1}}\right)},\label{eq:calcBeta3Lim-3}
\end{equation}
where (\ref{eq:calcBeta3Lim-3}) uses Theorem \ref{thm:theGenVar},
(\ref{eq:intUpperBnd}), $\Phi\left(\nicefrac{\delta}{\sqrt{\rho^{2}+1}}\right)\stackrel[\infty]{\delta}{\longrightarrow}1$,
and $\Phi\left(\nicefrac{-\delta}{\sqrt{\rho^{2}+1}}\right)\stackrel[\infty]{\delta}{\longrightarrow}0$.
We then have
\begin{align}
\lim_{\delta\rightarrow\infty}\frac{\Phi\left(\nicefrac{-\delta}{\sqrt{\rho^{2}+1}}\right)}{\exp\left(\nicefrac{-\delta^{2}}{\left(\rho^{2}+2\right)}\right)} & =\lim_{\delta\rightarrow\infty}\frac{\left(\rho^{2}+2\right)\exp\left(\nicefrac{\rho^{2}\delta^{2}}{2\left(\rho^{2}+1\right)\left(\rho^{2}+2\right)}\right)}{2\delta\sqrt{2\pi\left(\rho^{2}+1\right)}}=\infty,\label{eq:calcBeta3Lim-4}\\
\lim_{\delta\rightarrow\infty}\frac{\Phi\left(\nicefrac{-\sqrt{2}\delta}{\sqrt{\rho^{2}+2}}\right)}{\Phi\left(\nicefrac{-\delta}{\sqrt{\rho^{2}+1}}\right)} & =\sqrt{\frac{2\left(\rho^{2}+1\right)}{\rho^{2}+2}}\lim_{\delta\rightarrow\infty}e^{\nicefrac{-\rho^{2}\delta^{2}}{\left[2\left(\rho^{2}+1\right)\left(\rho^{2}+2\right)\right]}}=0,\label{eq:calcBeta3Lim-5}
\end{align}
where (\ref{eq:calcBeta3Lim-4}) and (\ref{eq:calcBeta3Lim-5}) use
L'Hôpital's rule. Using (\ref{eq:calcBeta3Lim-4}) and (\ref{eq:calcBeta3Lim-5})
in (\ref{eq:calcBeta3Lim-3}) gives (\ref{eq:calcBeta3Lim-2.5}).
Substituting (\ref{eq:calcBeta3Lim-2.5}) into (\ref{eq:calcBeta3Lim-2})
gives $\lim_{\delta\rightarrow\infty}b_{\rho,\delta}=\infty$. Showing
that $a_{\rho,\delta}\rightarrow0$ and $b_{\rho,\delta}\rightarrow\infty$,
as $\delta\rightarrow\infty$, completes the proof because $a_{\rho,-\delta}=b_{\rho,\delta}$.
\end{proof}
\begin{prop}
\label{prop:Xto}For $\alpha,\beta>0$ and $X_{\alpha,\beta}\sim\mathrm{Beta}\left(\alpha,\beta\right)$
we have:
\begin{enumerate}
\item \label{enu:propXto0}If $\alpha=o\left(\beta\right)$ and $\beta\rightarrow\infty$,
then $X_{\alpha,\beta}\longrightarrow0$.
\item \label{enu:propXto1}If $\alpha\rightarrow\infty$ and $\beta=o\left(\alpha\right)$,
then $X_{\alpha,\beta}\longrightarrow1$.
\item \label{enu:propXtoBern}If $\alpha,\beta\rightarrow0^{+}$ so that
$\nicefrac{\alpha}{\left(\alpha+\beta\right)}\rightarrow\lambda\in\left(0,1\right)$,
then $X_{\alpha,\beta}\implies\mathrm{Bernoulli}\left(\lambda\right)$.
\item \label{enu:propXtoNorm}If $\alpha,\beta\rightarrow\infty$ so that
$\nicefrac{\alpha}{\left(\alpha+\beta\right)}\rightarrow\lambda\in\left(0,1\right)$,
then $\sqrt{\alpha+\beta}\left(X_{\alpha,\beta}-\nicefrac{\alpha}{\left(\alpha+\beta\right)}\right)\implies\mathcal{N}\left(0,\lambda\left(1-\lambda\right)\right)$.
\end{enumerate}
\end{prop}

\begin{proof}
We prove part \ref{enu:propXto0}. Part \ref{enu:propXto1} follows
in the same manner. First, we note that
\begin{equation}
\hat{\lim}\,\mathrm{Var}\left(X_{\alpha,\beta}\right)=\hat{\lim}\,\frac{\alpha}{\left(\alpha+\beta\right)^{2}\left(1+\nicefrac{\alpha}{\beta}+\nicefrac{1}{\beta}\right)}=0,\label{eq:betaVarTo0}
\end{equation}
where $\hat{\lim}$ is the limit with $\beta\rightarrow\infty$ and
$\alpha=o\left(\beta\right)$. Fix $\epsilon>0$ and $\beta$ large
enough. Then, using the triangle inequality, Chebyshev's inequality,
and (\ref{eq:betaVarTo0}), we have
\begin{equation}
\hat{\lim}\,\Pr\left(\left|X_{\alpha,\beta}\right|>\epsilon\right)\leq\hat{\lim}\,\left.\mathrm{Var}\left(X_{\alpha,\beta}\right)\right/\left(\epsilon-\nicefrac{\alpha}{\left(\alpha+\beta\right)}\right)^{2}=0.
\end{equation}

For part \ref{enu:propXtoBern} let $B\left(x;\alpha,\beta\right)\coloneqq\int_{0}^{x}y^{\alpha-1}\left(1-y\right)^{\beta-1}dy$
be the incomplete beta function for which $B\left(x;\alpha,\beta\right)=\alpha^{-1}x^{\alpha}\left(1+\mathcal{O}\left(x\right)\right)$,
as $x\rightarrow0^{+}$ (\citet{P68}). Note also that $\Gamma\left(x\right)\sim\nicefrac{1}{x}-\gamma$,
for Euler's constant\textbf{ }$\gamma\approx0.577216$, as $x\rightarrow0^{+}$.
Then, for $\epsilon\in\left(0,1\right)$, we have
\begin{align}
\Pr\left(X_{\alpha,\beta}\leq\epsilon\right) & =\frac{B\left(\epsilon;\alpha,\beta\right)}{B\left(\alpha,\beta\right)}=\frac{B\left(\epsilon;\alpha,\beta\right)\Gamma\left(\alpha+\beta\right)}{\Gamma\left(\alpha\right)\Gamma\left(\beta\right)}\\
 & \sim\frac{\beta\epsilon^{\alpha}}{\alpha+\beta}\frac{1-\left(\alpha+\beta\right)\gamma}{\left(1-\alpha\gamma\right)\left(1-\beta\gamma\right)}\rightarrow1-\lambda,
\end{align}
where $\sim$ assumes $\alpha,\beta,\epsilon$ small and $\rightarrow$
sends $\alpha,\beta,\epsilon$ to zero from above. That is, we have
$\lim_{\alpha,\beta,\epsilon\rightarrow0^{+}}\Pr\left(X_{\alpha,\beta}\leq\epsilon\right)=1-\lambda$.
Noting that $B\left(x;\alpha,\beta\right)=B\left(\alpha,\beta\right)-B\left(1-x;\beta,\alpha\right)$,
we next have
\begin{align}
\Pr\left(X_{\alpha,\beta}\geq1-\epsilon\right) & =1-\frac{B\left(1-\epsilon;\alpha,\beta\right)}{B\left(\alpha,\beta\right)}=\frac{B\left(\epsilon;\beta,\alpha\right)\Gamma\left(\alpha+\beta\right)}{\Gamma\left(\alpha\right)\Gamma\left(\beta\right)}\\
 & \sim\frac{\alpha\epsilon^{\beta}}{\alpha+\beta}\frac{1-\left(\alpha+\beta\right)\gamma}{\left(1-\alpha\gamma\right)\left(1-\beta\gamma\right)}\rightarrow\lambda,
\end{align}
where $\sim$ assumes $\alpha,\beta,\epsilon$ small and $\rightarrow$
sends $\alpha,\beta,\epsilon$ to zero from above. That is, we have
$\lim_{\alpha,\beta,\epsilon\rightarrow0^{+}}\Pr\left(X_{\alpha,\beta}\geq1-\epsilon\right)=\lambda$

For part \ref{enu:propXtoNorm} we assume, without loss of generality,
that $\alpha,\beta\in\left\{ 1,2,\ldots\right\} $. With $\xi_{i}\stackrel{\mathrm{iid}}{\sim}\mathrm{Exp}\left(1\right)$,
$1\leq i\leq\alpha+\beta$, let 
\begin{equation}
G\coloneqq\sum_{i=1}^{\alpha}\xi_{i}\sim\mathrm{Gamma}\left(\alpha,1\right)\textrm{ and }G'\coloneqq\sum_{i=\alpha+1}^{\alpha+\beta}\xi_{i}\sim\mathrm{Gamma}\left(\beta,1\right),
\end{equation}
so that $G$ and $G'$ are independent, and $X_{\alpha,\beta}\stackrel{\mathscr{L}}{=}\frac{G}{G+G'}\sim\mathrm{Beta}\left(\alpha,\beta\right)$,
implying that $\sqrt{\alpha+\beta}\left(X_{\alpha,\beta}-\nicefrac{\alpha}{\left(\alpha+\beta\right)}\right)$
\begin{align}
 & \stackrel{\mathscr{L}}{=}\sqrt{\alpha+\beta}\left(\frac{\sum_{i=1}^{\alpha}\xi_{i}-\frac{\alpha}{\alpha+\beta}\sum_{i=1}^{\alpha+\beta}\xi_{i}}{\sum_{i=1}^{\alpha+\beta}\xi_{i}}\right)\\
 & =\frac{\frac{\sqrt{\alpha+\beta}}{\alpha+\beta}\frac{\beta}{\alpha+\beta}\sum_{i=1}^{\alpha}\xi_{i}-\frac{\sqrt{\alpha+\beta}}{\alpha+\beta}\frac{\alpha}{\alpha+\beta}\sum_{i=\alpha+1}^{\alpha+\beta}\xi_{i}}{\frac{1}{\alpha+\beta}\sum_{i=1}^{\alpha+\beta}\xi_{i}}\label{eq:full-expr}\\
 & =\frac{\sqrt{\frac{\alpha}{\alpha+\beta}}\frac{\beta}{\alpha+\beta}\frac{1}{\sqrt{\alpha}}\sum_{i=1}^{\alpha}\left(\xi_{i}-1\right)-\sqrt{\frac{\beta}{\alpha+\beta}}\frac{\alpha}{\alpha+\beta}\frac{1}{\sqrt{\beta}}\sum_{i=\alpha+1}^{\alpha+\beta}\left(\xi_{i}-1\right)}{\frac{1}{\alpha+\beta}\sum_{i=1}^{\alpha+\beta}\xi_{i}}\label{eq:combined}
\end{align}
The result follows from the Strong Law of Large Numbers (SLLN), the
Central Limit Theorem (CLT), the independence of the two terms in
the numerator of (\ref{eq:combined}), and Slutsky's theorem.
\end{proof}
\limitX*
\begin{proof}
(\ref{eq:Xto1}), (\ref{eq:Xto0}), and (\ref{eq:XtoBern}) follow
from Proposition \ref{prop:limitab} and Propositions \ref{prop:Xto}.\ref{enu:propXto1},
\ref{prop:Xto}.\ref{enu:propXto0}, and \ref{prop:Xto}.\ref{enu:propXtoBern}.
(\ref{eq:XtoHalf}) and (\ref{eq:XtoPhi}) use (\ref{eq:stanXtoN1})
and (\ref{eq:stanXtoN2}), the triangle inequality, and Chebyshev's
inequality. (\ref{eq:stanXtoN1}) and (\ref{eq:stanXtoN2}) use Propositions
\ref{prop:limitab} and \ref{prop:Xto}.\ref{enu:propXtoNorm} and
the following. First, (\ref{eq:stanZtoN1}) implies that $\mathrm{Var}\left(Z_{\rho,\delta}\right)\sim\nicefrac{1}{2\pi\rho^{2}}$,
as $\rho\rightarrow\infty$, which then---using (\ref{eq:alpha3})
and (\ref{eq:beta3})---implies that $\sqrt{\nicefrac{2}{\pi}}\sqrt{a_{\rho,\delta}+b_{\rho,\delta}}\sim\rho$,
as $\rho\rightarrow\infty$. Then,
\begin{align}
 & \lim_{\rho\rightarrow\infty}\rho\left(X_{a_{\rho,\delta},b_{\rho,\delta}}-\Phi\left(\nicefrac{-\delta}{\sqrt{\rho^{2}+1}}\right)\right)\\
 & \qquad\qquad=\sqrt{\nicefrac{2}{\pi}}\lim_{\rho\rightarrow\infty}\sqrt{a_{\rho,\delta}+b_{\rho,\delta}}\left(X_{a_{\rho,\delta},b_{\rho,\delta}}-\nicefrac{a_{\rho,\delta}}{\left(a_{\rho,\delta}+b_{\rho,\delta}\right)}\right)\\
 & \qquad\qquad=\sqrt{\nicefrac{2}{\pi}}\,\mathcal{N}\left(0,\nicefrac{1}{4}\right)=\mathcal{N}\left(0,\nicefrac{1}{2\pi}\right)
\end{align}
because $\lim_{\rho\rightarrow\infty}\nicefrac{a_{\rho,\delta}}{\left(a_{\rho,\delta}+b_{\rho,\delta}\right)}=\lim_{\rho\rightarrow\infty}\Phi\left(\nicefrac{-\delta}{\sqrt{\rho^{2}+1}}\right)=\nicefrac{1}{2}$,
which gives (\ref{eq:stanXtoN1}). 

In a similar way, (\ref{eq:stanZtoN2}) implies that $\mathrm{Var}\left(Z_{\rho,\delta}\right)\sim\left(\nicefrac{\phi\left(r\right)}{\rho}\right)^{2}$,
as $\rho,\left|\delta\right|\rightarrow\infty$ while keeping $\nicefrac{\delta}{\rho}=r$
fixed, which then---using (\ref{eq:alpha3}) and (\ref{eq:beta3})---implies
that
\begin{equation}
\hat{\lim}\left.\sqrt{a_{\rho,\delta}+b_{\rho,\delta}}\right/\rho=\left.\sqrt{\Phi\left(-r\right)\Phi\left(r\right)}\right/\phi\left(r\right),
\end{equation}
where $\hat{\lim}$ is the limit that sends $\rho,\left|\delta\right|\rightarrow\infty$
while keeping $\nicefrac{\delta}{\rho}=r$ fixed. Then, 
\begin{align}
 & \hat{\lim}\,\rho\left(X_{a_{\rho,\delta},b_{\rho,\delta}}-\Phi\left(\nicefrac{-\delta}{\sqrt{\rho^{2}+1}}\right)\right)\\
 & \qquad=\frac{\phi\left(r\right)}{\sqrt{\Phi\left(-r\right)\Phi\left(r\right)}}\hat{\lim}\,\sqrt{a_{\rho,\delta}+b_{\rho,\delta}}\left(X_{a_{\rho,\delta},b_{\rho,\delta}}-\nicefrac{a_{\rho,\delta}}{\left(a_{\rho,\delta}+b_{\rho,\delta}\right)}\right)\\
 & \qquad=\frac{\phi\left(r\right)}{\sqrt{\Phi\left(-r\right)\Phi\left(r\right)}}\,\mathcal{N}\left(0,\Phi\left(-r\right)\Phi\left(r\right)\right)=\mathcal{N}\left(0,\phi\left(r\right)^{2}\right)
\end{align}
because $\hat{\lim}\,\nicefrac{a_{\rho,\delta}}{\left(a_{\rho,\delta}+b_{\rho,\delta}\right)}=\hat{\lim}\,\Phi\left(\nicefrac{-\delta}{\sqrt{\rho^{2}+1}}\right)=\Phi\left(-r\right)$,
which gives (\ref{eq:stanXtoN2}).
\end{proof}

\section{Proof of Theorem \ref{thm:RONO} \label{sec:Derivations-for-=0000A73.1}}

We begin by stating Lemma \ref{lem:diffSmall} and Propositions \ref{prop:densRho0}--\ref{prop:densRhoAbsDeltaInf},
which underpin the proof of Theorem \ref{thm:RONO}.
\begin{lem}
\label{lem:diffSmall}Fix $\alpha,\beta>0$. As $x\rightarrow\infty$,
$\left(1+\nicefrac{\alpha}{x^{4}}+\mathcal{O}\left(\nicefrac{1}{x^{6}}\right)\right)^{-\beta x^{2}}-1=\mathcal{O}\left(\nicefrac{1}{x^{2}}\right)$
and $\left(1+\nicefrac{\alpha}{x^{2}}+\mathcal{O}\left(\nicefrac{1}{x^{4}}\right)\right)^{-\beta x^{2}}-\mathrm{e}^{-\alpha\beta}=\mathcal{O}\left(\nicefrac{1}{x^{2}}\right)$.
\end{lem}

\begin{proof}
For the first one and $x$ large enough we have
\begin{align}
\log\left\{ \left(1+\nicefrac{\alpha}{x^{4}}+\mathcal{O}\left(\nicefrac{1}{x^{6}}\right)\right)^{-\beta x^{2}}\right\}  & =-\beta x^{2}\log\left(1+\nicefrac{\alpha}{x^{4}}+\mathcal{O}\left(\nicefrac{1}{x^{6}}\right)\right)\\
 & =-\beta x^{2}\sum_{k=1}^{\infty}\frac{\left(-1\right)^{k+1}\left(\nicefrac{\alpha}{x^{4}}+\mathcal{O}\left(\nicefrac{1}{x^{6}}\right)\right)^{k}}{k}\\
 & =\nicefrac{-\alpha\beta}{x^{2}}+\mathcal{O}\left(\nicefrac{1}{x^{4}}\right),\\
\left(1+\nicefrac{\alpha}{x^{4}}+\mathcal{O}\left(\nicefrac{1}{x^{6}}\right)\right)^{-\beta x^{2}}-1 & =\exp\left(\nicefrac{-\alpha\beta}{x^{2}}+\mathcal{O}\left(\nicefrac{1}{x^{4}}\right)\right)-1\\
 & =\sum_{k=1}^{\infty}\frac{\left(\nicefrac{-\alpha\beta}{x^{2}}+\mathcal{O}\left(\nicefrac{1}{x^{4}}\right)\right)^{k}}{k!}=\mathcal{O}\left(\nicefrac{1}{x^{2}}\right),
\end{align}
where the result holds as $x\rightarrow\infty$. Now, for the second
one and $x$ large enough, we have
\begin{align}
\log\left\{ \left(1+\nicefrac{\alpha}{x^{2}}+\mathcal{O}\left(\nicefrac{1}{x^{4}}\right)\right)^{-\beta x^{2}}\right\}  & =-\beta x^{2}\log\left(1+\nicefrac{\alpha}{x^{2}}+\mathcal{O}\left(\nicefrac{1}{x^{4}}\right)\right)\\
 & =-\beta x^{2}\sum_{k=1}^{\infty}\frac{\left(-1\right)^{k+1}\left(\nicefrac{\alpha}{x^{2}}+\mathcal{O}\left(\nicefrac{1}{x^{4}}\right)\right)^{k}}{k}\\
 & =-\alpha\beta+\mathcal{O}\left(\nicefrac{1}{x^{2}}\right),\\
\left(1+\nicefrac{\alpha}{x^{2}}+\mathcal{O}\left(\nicefrac{1}{x^{4}}\right)\right)^{-\beta x^{2}}-\mathrm{e}^{-\alpha\beta} & =\mathrm{e}^{-\alpha\beta}\left(\exp\left(\mathcal{O}\left(\nicefrac{1}{x^{2}}\right)\right)-1\right)\\
 & =\mathrm{e}^{-\alpha\beta}\sum_{k=1}^{\infty}\frac{\mathcal{O}\left(\nicefrac{1}{x^{2k}}\right)}{k!}=\mathcal{O}\left(\nicefrac{1}{x^{2}}\right),
\end{align}
where the result holds as $x\rightarrow\infty$. This completes the
proof.
\end{proof}
\begin{prop}
\label{prop:densRho0}For $z\in\left(0,1\right)$ we have $\lim_{\rho\rightarrow0^{+}}\left|f_{\rho,\delta}\left(z\right)-g_{a_{\rho,\delta},b_{\rho,\delta}}\left(z\right)\right|=0$.
\end{prop}

\begin{proof}
We see from (\ref{eq:densityExplicit}) that 
\begin{equation}
\lim_{\rho\rightarrow0^{+}}f_{\rho,\delta}\left(z\right)=\lim_{\rho\rightarrow0^{+}}\rho\exp\left[\nicefrac{1}{2}\left(\Phi^{-1}\left(z\right)+\delta\right)\left(\Phi^{-1}\left(z\right)-\delta\right)\right]=0\label{eq:fto0rhoSmall}
\end{equation}
because $\left(\Phi^{-1}\left(z\right)+\delta\right)\left(\Phi^{-1}\left(z\right)-\delta\right)$
is finite. For the beta distribution, $\rho\rightarrow0^{+}$ implies
that $a_{\rho,\delta},b_{\rho,\delta}\rightarrow0^{+}$ (Proposition
\ref{prop:limitab}). While $\lim_{\rho\rightarrow0^{+}}\nicefrac{a_{\rho,\delta}}{\left(a_{\rho,\delta}+b_{\rho,\delta}\right)}=\lim_{\rho\rightarrow0^{+}}\Phi\left(\nicefrac{-\delta}{\sqrt{\rho^{2}+1}}\right)=\Phi\left(-\delta\right)\in\left(0,1\right)$,
we have $\lim_{\rho\rightarrow0^{+}}\nicefrac{a_{\rho,\delta}b_{\rho,\delta}}{\left(a_{\rho,\delta}+b_{\rho,\delta}\right)}=0$.
Using $\Gamma\left(x\right)\sim\nicefrac{1}{x}-\gamma$, for Euler's
constant\textbf{ }$\gamma\approx0.577216$, as $x\rightarrow0^{+}$,
gives
\begin{equation}
\lim_{\rho\rightarrow0^{+}}g_{a_{\rho,\delta},b_{\rho,\delta}}\left(z\right)=\lim_{\rho\rightarrow0^{+}}\frac{a_{\rho,\delta}b_{\rho,\delta}}{a_{\rho,\delta}+b_{\rho,\delta}}\frac{1-\left(a_{\rho,\delta}+b_{\rho,\delta}\right)\gamma}{\left(1-a_{\rho,\delta}\gamma\right)\left(1-b_{\rho,\delta}\gamma\right)}\frac{1}{z\left(1-z\right)}=0\label{eq:bto0rhoSmall}
\end{equation}
because $\nicefrac{1}{z\left(1-z\right)}$ is finite. Combining (\ref{eq:fto0rhoSmall})
and (\ref{eq:bto0rhoSmall}) gives the result.
\end{proof}
\begin{prop}
\label{prop:densRhoInf}For $z\in\left(0,1\right)$ we have $\lim_{\rho\rightarrow\infty}\left|f_{\rho,\delta}\left(z\right)-g_{a_{\rho,\delta},b_{\rho,\delta}}\left(z\right)\right|=0$.
\end{prop}

\begin{proof}
Putting $\bar{z}\coloneqq1-z$, $\mu\coloneqq\nicefrac{a_{\rho,\delta}}{\left(a_{\rho,\delta}+b_{\rho,\delta}\right)}=\Phi\left(\nicefrac{-\delta}{\sqrt{\rho^{2}+1}}\right)\eqqcolon1-\bar{\mu}$,
$a\coloneqq a_{\rho,\delta}$, $b\coloneqq b_{\rho,\delta}$, and
$v\coloneqq\mathrm{Var}\left(Z_{\rho,\delta}\right)$, we note that
$\log g_{a,b}\left(z\right)$
\begin{align}
 & =\log\Gamma\left(a+b\right)-\log\Gamma\left(a\right)-\log\Gamma\left(b\right)+\left(a-1\right)\log z+\left(b-1\right)\log\bar{z}\\
 & \sim a\log\nicefrac{z}{\mu}+b\log\nicefrac{\bar{z}}{\bar{\mu}}+\nicefrac{1}{2}\log\nicefrac{\mu b}{2\pi}-\log z\bar{z}\label{eq:useStirling}\\
 & =\left(\nicefrac{\mu\bar{\mu}}{v}-1\right)\left\{ \mu\log\nicefrac{z}{\mu}+\bar{\mu}\log\nicefrac{\bar{z}}{\bar{\mu}}\right\} +\nicefrac{1}{2}\log\left(\nicefrac{\mu\bar{\mu}}{2\pi}\left(\nicefrac{\mu\bar{\mu}}{v}-1\right)\right)-\log z\bar{z},\label{eq:useDefa3b3}
\end{align}
where (\ref{eq:useStirling}) uses Stirling's approximation because
$a_{\rho,\delta},b_{\rho,\delta}\rightarrow\infty$ by Proposition
\ref{prop:limitab}, and (\ref{eq:useDefa3b3}) uses (\ref{eq:alpha3})
and (\ref{eq:beta3}). 

If $z\in\left.\left(0,1\right)\right\backslash \left\{ \nicefrac{1}{2}\right\} $,
then $\lim_{\rho\rightarrow\infty}f_{\rho,\delta}\left(z\right)=0$
as the $\rho^{2}$ term dominates (\ref{eq:densityExplicit})'s exponent
and $-\Phi^{-1}\left(z\right)^{2}<0$. Now, for the beta distribution
we have 
\begin{equation}
\lim_{\rho\rightarrow\infty}\log g_{a_{\rho,\delta},b_{\rho,\delta}}\left(z\right)=\lim_{\rho\rightarrow\infty}\left\{ \nicefrac{\pi\rho^{2}}{4}\log4z\bar{z}+\log\nicefrac{\rho}{4}-\log z\bar{z}\right\} =-\infty,\label{eq:betaDensRhoBig}
\end{equation}
where the first equality in (\ref{eq:betaDensRhoBig}) uses (\ref{eq:useDefa3b3})
and (\ref{eq:stanZtoN1}), and the second equality uses the asymptotic
dominance of $\rho^{2}$ over $\log\rho$ and $4z\bar{z}\in\left(0,1\right)$,
or $\log4z\bar{z}<0$, when $z\in\left.\left(0,1\right)\right\backslash \left\{ \nicefrac{1}{2}\right\} $.
This proves the result when $z\in\left.\left(0,1\right)\right\backslash \left\{ \nicefrac{1}{2}\right\} $.

For $z=\nicefrac{1}{2}$, note that $f_{\rho,\delta}\left(\nicefrac{1}{2}\right)=\rho\exp\left(\nicefrac{-\delta^{2}}{2}\right)$
by (\ref{eq:densityExplicit}). Further, we have 
\begin{align}
\log g_{a_{\rho,\delta},b_{\rho,\delta}}\left(\nicefrac{1}{2}\right) & \sim\nicefrac{-\pi\rho^{2}}{2}\left[\mu\log\left(2\mu\right)+\bar{\mu}\log\left(2\bar{\mu}\right)\right]+\log\rho,\label{eq:logbdensRhoBig}
\end{align}
where we use (\ref{eq:useDefa3b3}), (\ref{eq:stanZtoN1}), and $\lim_{\rho\rightarrow\infty}\mu=\lim_{\rho\rightarrow\infty}\bar{\mu}=\nicefrac{1}{2}$.
This implies that
\begin{align}
g_{a_{\rho,\delta},b_{\rho,\delta}}\left(\nicefrac{1}{2}\right) & \sim\rho\left[\left(2\mu\right)^{\mu}\left(2\bar{\mu}\right)^{\bar{\mu}}\right]^{\nicefrac{-\pi\rho^{2}}{2}}.\label{eq:bdensRhoBig}
\end{align}
Then, noting that $\left[2\Phi\left(-x\right)\right]^{\Phi\left(-x\right)}\left[2\Phi\left(x\right)\right]^{\Phi\left(x\right)}=1+\nicefrac{x^{2}}{\pi}+\mathcal{O}\left(x^{4}\right)$
by Taylor series expansion, we have
\begin{align}
g_{a_{\rho,\delta},b_{\rho,\delta}}\left(\nicefrac{1}{2}\right) & \sim\rho\left[1+\frac{\nicefrac{\delta^{2}}{\pi}}{\rho^{2}+1}+\mathcal{O}\left(\nicefrac{1}{\rho^{4}}\right)\right]^{\nicefrac{-\pi\rho^{2}}{2}}.
\end{align}
Applying the second statement in Lemma \ref{lem:diffSmall}, we have
$f_{\rho,\delta}\left(\nicefrac{1}{2}\right)-g_{a_{\rho,\delta},b_{\rho,\delta}}\left(\nicefrac{1}{2}\right)\sim$
\begin{equation}
\rho\left\{ \exp\left(\nicefrac{-\delta^{2}}{2}\right)-\left[1+\frac{\nicefrac{\delta^{2}}{\pi}}{\rho^{2}+1}+\mathcal{O}\left(\nicefrac{1}{\rho^{4}}\right)\right]^{\nicefrac{-\pi\rho^{2}}{2}}\right\} =\mathcal{O}\left(\nicefrac{1}{\rho}\right).
\end{equation}
This gives the result when $z=\nicefrac{1}{2}$ and completes the
proof.
\end{proof}
\begin{prop}
\label{prop:densAbsDeltaInf}For $z\in\left(0,1\right)$ we have $\lim_{\left|\delta\right|\rightarrow\infty}\left|f_{\rho,\delta}\left(z\right)-g_{a_{\rho,\delta},b_{\rho,\delta}}\left(z\right)\right|=0$.
\end{prop}

\begin{proof}
Note that $\lim_{\left|\delta\right|\rightarrow\infty}f_{\rho,\delta}\left(z\right)=0$
because, in this setting, the $\nicefrac{-\delta^{2}}{2}<0$ term
dominates (\ref{eq:densityExplicit})'s exponent. For the beta setting
we focus on $\delta\rightarrow\infty$, so that $a_{\rho,\delta}\rightarrow0^{+}$
and $b_{\rho,\delta}\rightarrow\infty$ (see Proposition \ref{prop:limitab}).
Noting that $B\left(a_{\rho,\delta},b_{\rho,\delta}\right)\sim\nicefrac{\Gamma\left(a_{\rho,\delta}\right)}{b_{\rho,\delta}^{a_{\rho,\delta}}}$
in this setting gives $\lim_{\delta\rightarrow\infty}\log g_{a_{\rho,\delta},b_{\rho,\delta}}\left(z\right)$
\begin{align}
 & =\lim_{\delta\rightarrow\infty}\left\{ a_{\rho,\delta}\log b_{\rho,\delta}-\log\Gamma\left(a_{\rho,\delta}\right)+\left(a_{\rho,\delta}-1\right)\log z+\left(b_{\rho,\delta}-1\right)\log\left(1-z\right)\right\} \nonumber \\
 & =\lim_{\delta\rightarrow\infty}\left\{ a_{\rho,\delta}\log b_{\rho,\delta}+\log a_{\rho,\delta}+\left(a_{\rho,\delta}-1\right)\log z+\left(b_{\rho,\delta}-1\right)\log\left(1-z\right)\right\} \label{eq:GammaSmall}\\
 & =\lim_{\delta\rightarrow\infty}\left\{ a_{\rho,\delta}\log b_{\rho,\delta}+\log a_{\rho,\delta}+\left(b_{\rho,\delta}-1\right)\log\left(1-z\right)\right\} -\log z\label{eq:intermed}\\
 & =-\infty,\label{eq:betaDeltaLim}
\end{align}
where (\ref{eq:GammaSmall}) uses $\Gamma\left(x\right)\sim\nicefrac{1}{x}$,
as $x\rightarrow0^{+}$, and (\ref{eq:betaDeltaLim}) follows because
the last two bracketed terms in (\ref{eq:intermed}) dominate the
first one ($z\in\left(0,1\right)$ gives $\log\left(1-z\right)<0$).
This implies that $\lim_{\delta\rightarrow\infty}g_{a_{\rho,\delta},b_{\rho,\delta}}\left(z\right)=0$.
An argument similar to that above shows that $\lim_{\delta\rightarrow-\infty}g_{a_{\rho,\delta},b_{\rho,\delta}}\left(z\right)=0$,
completing the proof.
\end{proof}
\begin{prop}
\label{prop:densRhoAbsDeltaInf}For $z\in\left(0,1\right)$ we have
$\lim_{\rho,\left|\delta\right|\rightarrow\infty}\left|f_{\rho,\delta}\left(z\right)-g_{a_{\rho,\delta},b_{\rho,\delta}}\left(z\right)\right|=0$,
where the limit keeps $\nicefrac{\delta}{\rho}=r$ fixed.
\end{prop}

\begin{proof}
If $z\in\left.\left(0,1\right)\right\backslash \left\{ \Phi\left(-r\right)\right\} $,
then $\hat{\lim}\,f_{\rho,\delta}\left(z\right)=0$, where $\hat{\lim}$
sends $\rho,\left|\delta\right|$ to infinity while keeping $\nicefrac{\delta}{\rho}=r$
fixed. To see this, note that
\begin{equation}
f_{\rho,\delta}\left(z\right)=\rho\exp\left\{ \nicefrac{-1}{2}\left[\rho^{2}\left(\Phi^{-1}\left(z\right)+r\right)^{2}-\Phi^{-1}\left(z\right)^{2}\right]\right\} \label{eq:fWhenRFixed}
\end{equation}
when $\delta=r\rho$. The $\rho^{2}$ term dominates the exponent
and its coefficient is negative. We now show that $\hat{\lim}\,g_{a_{\rho,\delta},b_{\rho,\delta}}\left(z\right)=0$
when $z\in\left.\left(0,1\right)\right\backslash \left\{ \Phi\left(-r\right)\right\} $.
In this setting note that $\hat{\lim}\,\log g_{a_{\rho,\delta},b_{\rho,\delta}}\left(z\right)$
\begin{equation}
=\hat{\lim}\left\{ p\bar{p}\left(\nicefrac{\rho}{\phi\left(r\right)}\right)^{2}\left(p\log\nicefrac{z}{p}+\bar{p}\log\nicefrac{\bar{z}}{\bar{p}}\right)+\log\left(\nicefrac{\rho p\bar{p}}{\sqrt{2\pi}\phi\left(r\right)}\right)\right\} -\log z\bar{z},\label{eq:logBetaRhoDelta}
\end{equation}
where we use (\ref{eq:useDefa3b3}), (\ref{eq:stanZtoN2}), $\bar{z}\coloneqq1-z$,
and $p\coloneqq\Phi\left(-r\right)\eqqcolon1-\bar{p}$. Then, noting
that $z\ne p\coloneqq\Phi\left(-r\right)$, we have $p\log\nicefrac{z}{p}+\bar{p}\log\nicefrac{\bar{z}}{\bar{p}}<p\left(1-\nicefrac{z}{p}\right)+\bar{p}\left(1-\nicefrac{\bar{z}}{\bar{p}}\right)=0$,
so that the coefficient of the asymptotically-dominant $\rho^{2}$
term in (\ref{eq:logBetaRhoDelta}) is negative, giving $\hat{\lim}\,\log g_{a_{\rho,\delta},b_{\rho,\delta}}\left(z\right)=-\infty$.
We are done when $z\in\left.\left(0,1\right)\right\backslash \left\{ \Phi\left(-r\right)\right\} $.

If $z=\Phi\left(-r\right)$, we have $f_{\rho,\delta}\left(\Phi\left(-r\right)\right)=\rho\exp\left(\nicefrac{r^{2}}{2}\right)$
by (\ref{eq:fWhenRFixed}). In what follows we put $\mu\coloneqq\Phi\left(\nicefrac{-\delta}{\sqrt{\rho^{2}+1}}\right)\eqqcolon1-\bar{\mu}$.
Now, for the beta distribution, (\ref{eq:useDefa3b3}), (\ref{eq:stanZtoN2}),
and $\hat{\lim}\,\mu=\Phi\left(-r\right)=1-\hat{\lim}\,\bar{\mu}$
imply that
\begin{equation}
g_{a_{\rho,\delta},b_{\rho,\delta}}\left(\Phi\left(-r\right)\right)\hat{\sim}\,\rho\exp\left(\nicefrac{r^{2}}{2}\right)\left\{ \left[\frac{\mu}{\Phi\left(-r\right)}\right]^{\mu}\left[\frac{\bar{\mu}}{\Phi\left(r\right)}\right]^{\bar{\mu}}\right\} ^{-\frac{\Phi\left(-r\right)\Phi\left(r\right)\rho^{2}}{\phi\left(r\right)^{2}}},
\end{equation}
where $f_{1}\left(\rho,\delta\right)\hat{\sim}f_{2}\left(\rho,\delta\right)$
indicates that $\hat{\lim}\,\nicefrac{f_{1}\left(\rho,\delta\right)}{f_{2}\left(\rho,\delta\right)}=1$.
Noting that
\begin{equation}
\left[\frac{\Phi\left(-x\right)}{\Phi\left(-r\right)}\right]^{\Phi\left(-x\right)}\left[\frac{\Phi\left(x\right)}{\Phi\left(r\right)}\right]^{\Phi\left(x\right)}=1+\frac{\phi\left(r\right)^{2}\Delta^{2}}{2\Phi\left(-r\right)\Phi\left(r\right)}+\mathcal{O}\left(\Delta^{3}\right)
\end{equation}
for $\Delta\coloneqq r-x$ (by Taylor series expansion) and setting
$x\coloneqq\nicefrac{\delta}{\sqrt{\rho^{2}+1}}$, we obtain $g_{a_{\rho,\delta},b_{\rho,\delta}}\left(\Phi\left(-r\right)\right)$
\begin{equation}
\hat{\sim}\,\rho\exp\left(\nicefrac{r^{2}}{2}\right)\left[1+\frac{\phi\left(r\right)^{2}\Delta^{2}}{2\Phi\left(-r\right)\Phi\left(r\right)}+\mathcal{O}\left(\Delta^{3}\right)\right]^{-\frac{\Phi\left(-r\right)\Phi\left(r\right)\rho^{2}}{\phi\left(r\right)^{2}}}.
\end{equation}
Lemma \ref{lem:diffSmall} completes the proof if $\Delta=\nicefrac{C}{\rho^{2}}+o\left(\nicefrac{1}{\rho^{2}}\right)$
as $\rho,\left|\delta\right|\rightarrow\infty$ while $\nicefrac{\delta}{\rho}=r$
remains fixed. To see that $\Delta=\nicefrac{C}{\rho^{2}}+o\left(\nicefrac{1}{\rho^{2}}\right)$
in this setting, note that
\begin{align}
\Delta & =r-\frac{\delta}{\sqrt{\rho^{2}+1}}=r\left(1-\frac{\rho}{\sqrt{\rho^{2}+1}}\right)=r\left(\frac{\sqrt{1+\nicefrac{1}{\rho^{2}}}-1}{\sqrt{1+\nicefrac{1}{\rho^{2}}}}\right)\label{eq:useDefR}\\
 & \hat{\sim}\,r\left(\sqrt{1+\nicefrac{1}{\rho^{2}}}-1\right)=r\left[\nicefrac{1}{\left(2\rho^{2}\right)}+\mathcal{O}\left(\nicefrac{1}{\rho^{4}}\right)\right]=\nicefrac{C}{\rho^{2}}+o\left(\nicefrac{1}{\rho^{2}}\right),\label{eq:useTaySqrt}
\end{align}
where (\ref{eq:useDefR}) uses $\delta=r\rho$, and (\ref{eq:useTaySqrt})
uses $\sqrt{y}=1+\nicefrac{1}{2}\left(y-1\right)+\mathcal{O}\left(\left(y-1\right)^{2}\right)$,
a Taylor series expansion. Using Lemma \ref{lem:diffSmall}, we have
$f_{\rho,\delta}\left(\nicefrac{1}{2}\right)-g_{a_{\rho,\delta},b_{\rho,\delta}}\left(\nicefrac{1}{2}\right)\hat{\sim}$
\begin{equation}
\rho\mathrm{e}^{\nicefrac{r^{2}}{2}}\left\{ 1-\left[1+\frac{r\phi\left(r\right)^{2}}{4\rho^{4}\Phi\left(-r\right)\Phi\left(r\right)}+\mathcal{O}\left(\nicefrac{1}{\rho^{6}}\right)\right]^{-\frac{\Phi\left(-r\right)\Phi\left(r\right)\rho^{2}}{\phi\left(r\right)^{2}}}\right\} =\mathcal{O}\left(\nicefrac{1}{\rho}\right).
\end{equation}
This gives the result when $z=\Phi\left(-r\right)$ and  completes
the proof.
\end{proof}
\limitRONO*
\begin{proof}
For $B_{z}\sim\mathrm{Binomial}\left(n,z\right)$, (\ref{eq:integralPR0})
and (\ref{eq:integralApprox}) give $\left|\Pr\left(R_{0}=k\right)-p_{\rho,\delta}\left(k\right)\right|$
\begin{align}
 & =\left|\int_{0}^{1}\Pr\left(B_{z}=k-1\right)\left\{ f_{\rho,\delta}\left(z\right)-g_{a_{\rho,\delta},b_{\rho,\delta}}\left(z\right)\right\} dz\right|\\
 & \leq\int_{0}^{1}\left|f_{\rho,\delta}\left(z\right)-g_{a_{\rho,\delta},b_{\rho,\delta}}\left(z\right)\right|dz\label{eq:applyDCT}\\
 & \leq\int_{0}^{1}f_{\rho,\delta}\left(z\right)dz+\int_{0}^{1}g_{a_{\rho,\delta},b_{\rho,\delta}}\left(z\right)dz=2,\label{eq:useTriIneq}
\end{align}
where (\ref{eq:applyDCT}) and (\ref{eq:useTriIneq}) use $\Pr\left(B_{z}=k-1\right)\leq1$
and the triangle inequality. The DCT applies to (\ref{eq:applyDCT})
by (\ref{eq:useTriIneq}). Propositions \ref{prop:densRho0}--\ref{prop:densRhoAbsDeltaInf}
complete the proof.
\end{proof}

\section{Proof of Proposition \ref{prop:otherRanks} \label{sec:Derivations-for-=0000A73.2}}

\otherRanks*
\begin{proof}
Starting with (\ref{eq:withR0}), we have $\Pr\left(R_{0}=j_{0},R_{i_{1}}=j_{1},\ldots,R_{i_{m}}=j_{m}\right)$
\begin{align}
 & =\Pr\left(\left.R_{i_{1}}=j_{1},\ldots,R_{i_{m}}=j_{m}\right|R_{0}=j_{0}\right)\Pr\left(R_{0}=j_{0}\right)\\
 & =\frac{\Pr\left(R_{0}=j_{0}\right)}{n\left(n-1\right)\cdots\left(n-m+1\right)},\label{eq:probWithR0}
\end{align}
where (\ref{eq:probWithR0}) follows because, conditional on $\left\{ R_{0}=j_{0}\right\} $,
the $\left(R_{i_{1}},R_{i_{2}},\ldots,R_{i_{m}}\right)$ are uniformly
distributed on $\left\{ \mathbf{j}\in\left[n+1\right]^{m}:j_{i}\textrm{ distinct and }j_{i}\ne j_{0},1\leq i\leq m\right\} $,
for $\left[k\right]\coloneqq\left\{ 1,2,\ldots,k\right\} $. This
implies (\ref{eq:noR0}): $\Pr\left(R_{i_{1}}=j_{1},R_{i_{2}}=j_{2},\ldots,R_{i_{m}}=j_{m}\right)$
\begin{align}
 & =\sum_{j_{0}\notin\left\{ j_{1},j_{2},\ldots,j_{m}\right\} }\Pr\left(R_{0}=j_{0},R_{i_{1}}=j_{1},\ldots,R_{i_{m}}=j_{m}\right)\\
 & =\sum_{j_{0}\notin\left\{ j_{1},j_{2},\ldots,j_{m}\right\} }\frac{\Pr\left(R_{0}=j_{0}\right)}{n\left(n-1\right)\cdots\left(n-m+1\right)}=\frac{1-\sum_{k=1}^{m}\Pr\left(R_{0}=j_{k}\right)}{n\left(n-1\right)\cdots\left(n-m+1\right)},\label{eq:probNoR0}
\end{align}
where (\ref{eq:probNoR0}) uses (\ref{eq:probWithR0}). Turning to
(\ref{eq:meanR1}) we have
\begin{align}
\mathbb{E}R_{1} & =\sum_{k=1}^{n+1}\frac{k}{n}\left[1-\Pr\left(R_{0}=k\right)\right]=\frac{1}{n}\sum_{k=1}^{n+1}k-\frac{1}{n}\sum_{k=1}^{n+1}k\Pr\left(R_{0}=k\right)\label{eq:usePR1}\\
 & =\frac{\tilde{n}_{2}}{2n}-\frac{\mathbb{E}R_{0}}{n}=\frac{n+3}{2}-\Phi\left(\nicefrac{-\delta}{\sqrt{\rho^{2}+1}}\right)=\frac{n+1}{2}+\Phi\left(\nicefrac{\delta}{\sqrt{\rho^{2}+1}}\right),\label{eq:useMeanR0}
\end{align}
where (\ref{eq:usePR1}) uses (\ref{eq:probNoR0}) with $m=1$ and
(\ref{eq:useMeanR0}) uses $\mathbb{E}R_{0}=1+n\Phi\left(\nicefrac{-\delta}{\sqrt{\rho^{2}+1}}\right)$
from (\ref{eq:meanR0}) and $\tilde{n}_{2}\coloneqq\left(n+1\right)\left(n+2\right)$.
For (\ref{eq:varR1}) we then have
\begin{align}
\mathrm{Var}\left(R_{1}\right) & =\sum_{k=1}^{n+1}\frac{k^{2}}{n}\left[1-\Pr\left(R_{0}=k\right)\right]-\left(\mathbb{E}R_{1}\right)^{2}\label{eq:usePR1-1}\\
 & =\frac{1}{n}\sum_{k=1}^{n+1}k^{2}-\frac{1}{n}\sum_{k=1}^{n+1}k^{2}\Pr\left(R_{0}=k\right)-\left(\mathbb{E}R_{1}\right)^{2}\\
 & =\frac{\left(2n+3\right)\tilde{n}_{2}}{6n}-\frac{\mathrm{Var}\left(R_{0}\right)+\left(\mathbb{E}R_{0}\right)^{2}}{n}-\left(\mathbb{E}R_{1}\right)^{2}\\
 & =\frac{n^{2}-1}{12}+n\Phi\left(\nicefrac{-\delta}{\sqrt{\rho^{2}+1}}\right)\Phi\left(\nicefrac{\delta}{\sqrt{\rho^{2}+1}}\right)-\left(n-1\right)\mathrm{Var}\left(Z_{\rho,\delta}\right),\label{eq:usePropRes}
\end{align}
where (\ref{eq:usePR1-1}) uses (\ref{eq:probNoR0}) with $m=1$ and
(\ref{eq:usePropRes}) uses (\ref{eq:meanR0}), (\ref{eq:varR0}),
and (\ref{eq:useMeanR0}). For (\ref{eq:covR0R1}) we have $\mathrm{Cov}\left(R_{0},R_{1}\right)=\mathbb{E}\left[R_{0}R_{1}\right]-\mathbb{E}R_{0}\mathbb{E}R_{1}$,
where $\mathbb{E}\left[R_{0}R_{1}\right]$
\begin{align}
 & =\sum_{i=1}^{n+1}\sum_{j\ne i}\frac{ij\Pr\left(R_{0}=i\right)}{n}=\frac{1}{n}\sum_{i=1}^{n+1}i\Pr\left(R_{0}=i\right)\sum_{j\ne i}j\label{eq:usePR1-2}\\
 & =\frac{1}{n}\sum_{i=1}^{n+1}i\Pr\left(R_{0}=i\right)\left(\frac{\tilde{n}_{2}}{2}-i\right)=\frac{\tilde{n}_{2}\mathbb{E}R_{0}}{2n}-\frac{\mathrm{Var}\left(R_{0}\right)+\left(\mathbb{E}R_{0}\right)^{2}}{n}
\end{align}
and (\ref{eq:usePR1-2}) uses (\ref{eq:probWithR0}) with $m=1$.
Using (\ref{eq:meanR0}), (\ref{eq:varR0}), and (\ref{eq:useMeanR0})
we then have 
\begin{equation}
\mathrm{Cov}\left(R_{0},R_{1}\right)=-\Phi\left(\nicefrac{-\delta}{\sqrt{\rho^{2}+1}}\right)\Phi\left(\nicefrac{\delta}{\sqrt{\rho^{2}+1}}\right)-\left(n-1\right)\mathrm{Var}\left(Z_{\rho,\delta}\right).
\end{equation}
Finally, for (\ref{eq:covR1R2}) we have $\mathrm{Cov}\left(R_{1},R_{2}\right)=\mathbb{E}\left[R_{1}R_{2}\right]-\left(\mathbb{E}R_{1}\right)^{2}$,
where $\mathbb{E}\left[R_{1}R_{2}\right]$
\begin{align}
 & =\frac{1}{n\left(n-1\right)}\sum_{i=1}^{n+1}\sum_{j\ne i}ij\left[1-\Pr\left(R_{0}=i\right)-\Pr\left(R_{0}=j\right)\right]\label{eq:usePR1-3}\\
 & =\frac{1}{n\left(n-1\right)}\sum_{i=1}^{n+1}i\left\{ \left[1-\Pr\left(R_{0}=i\right)\right]\sum_{j\ne i}j-\sum_{j\ne i}j\Pr\left(R_{0}=j\right)\right\} \\
 & =\frac{1}{n\left(n-1\right)}\sum_{i=1}^{n+1}i\left\{ \left[1-\Pr\left(R_{0}=i\right)\right]\left(\frac{\tilde{n}_{2}}{2}-i\right)-\mathbb{E}R_{0}+i\Pr\left(R_{0}=i\right)\right\} \\
 & =\frac{1}{n\left(n-1\right)}\left\{ \frac{\tilde{n}_{2}^{2}}{4}-\frac{\left(2n+3\right)\tilde{n}_{2}}{6}-\tilde{n}_{2}\mathbb{E}R_{0}+2\mathrm{Var}\left(R_{0}\right)+2\left(\mathbb{E}R_{0}\right)^{2}\right\} 
\end{align}
and (\ref{eq:usePR1-3}) uses (\ref{eq:probNoR0}) with $m=2$. Using
(\ref{eq:meanR0}), (\ref{eq:varR0}), and (\ref{eq:useMeanR0}) we
then have 
\begin{equation}
\mathrm{Cov}\left(R_{1},R_{2}\right)=-\frac{n+1}{12}-\Phi\left(\nicefrac{-\delta}{\sqrt{\rho^{2}+1}}\right)\Phi\left(\nicefrac{\delta}{\sqrt{\rho^{2}+1}}\right)+2\mathrm{Var}\left(Z_{\rho,\delta}\right),
\end{equation}
which gives (\ref{eq:covR1R2}) and completes the proof.
\end{proof}

\section{Proof of Theorem \ref{thm:RankAsymp} \label{sec:Derivations-for-=0000A73.3}}

We state and prove Proposition \ref{prop:betaBinomLim} and Lemma
\ref{lem:cardCrossProd}, which together underpin the proof of Theorem
\ref{thm:RankAsymp}.
\begin{prop}
\label{prop:betaBinomLim}For $Y_{n}\sim\mathrm{BetaBinomial}\left(n,\alpha,\beta\right)$
in (\ref{eq:betaBinom}) and $\lambda\coloneqq\lim\frac{\alpha}{\alpha+\beta}$:
\begin{enumerate}
\item \label{enu:opt1}When $\alpha,\beta\rightarrow0^{+}$, $Y_{n}\implies n\mathrm{Bernoulli}\left(\lambda\right)$
and $\iota\coloneqq\mathrm{Cor}\left(\xi_{1},\xi_{2}\right)\rightarrow1$.
\item \label{enu:opt2}When $\alpha,\beta\rightarrow\infty$, $Y_{n}\implies\mathrm{Binomial}\left(n,\lambda\right)$
and $\xi_{1},\xi_{2},\ldots,\xi_{n}\implies\textrm{i.i.d.}$
\end{enumerate}
\end{prop}

\begin{proof}
We first consider the case $\alpha,\beta\rightarrow0^{+}$ and then
the case $\alpha,\beta\rightarrow\infty$. 
\begin{enumerate}
\item Note first that $\Gamma\left(x\right)\sim\nicefrac{1}{x}-\gamma$,
for Euler's constant\textbf{ }$\gamma\approx0.577216$, as $x\rightarrow0^{+}$.
We then have the following three cases:
\begin{align}
\Pr\left(Y_{n}=0\right) & =\binom{n}{0}\frac{\Gamma\left(n+\beta\right)}{\Gamma\left(n+\alpha+\beta\right)}\frac{\Gamma\left(\alpha+\beta\right)}{\Gamma\left(\beta\right)}\label{eq:firstBBinBern}\\
 & \stackrel[0^{+}]{\alpha,\beta}{\longrightarrow}\frac{\beta}{\alpha+\beta}\frac{1-\gamma\left(\alpha+\beta\right)}{1-\gamma\beta}\stackrel[0^{+}]{\alpha,\beta}{\longrightarrow}1-\lambda,\label{eq:useContGamma}
\end{align}
where (\ref{eq:useContGamma}) uses the continuity of $\Gamma$. Next,
for $1\leq k\leq n-1$, 
\begin{align}
\Pr\left(Y_{n}=k\right) & =\binom{n}{k}\frac{\Gamma\left(k+\alpha\right)\Gamma\left(n-k+\beta\right)}{\Gamma\left(n+\alpha+\beta\right)}\frac{\Gamma\left(\alpha+\beta\right)}{\Gamma\left(\alpha\right)\Gamma\left(\beta\right)}\\
 & \stackrel[0^{+}]{\alpha,\beta}{\longrightarrow}\frac{n}{k\left(n-k\right)}\frac{\alpha\beta}{\alpha+\beta}\frac{1-\gamma\left(\alpha+\beta\right)}{\left(1-\gamma\alpha\right)\left(1-\gamma\beta\right)}\stackrel[0^{+}]{\alpha,\beta}{\longrightarrow}0,\label{eq:useContGamma-1}
\end{align}
where (\ref{eq:useContGamma-1}) uses the continuity of $\Gamma$
and $\lim_{\alpha,\beta\rightarrow0^{+}}\nicefrac{\alpha\beta}{\left(\alpha+\beta\right)}=0$.
Finally, 
\begin{align}
\Pr\left(Y_{n}=n\right) & =\binom{n}{n}\frac{\Gamma\left(n+\alpha\right)}{\Gamma\left(n+\alpha+\beta\right)}\frac{\Gamma\left(\alpha+\beta\right)}{\Gamma\left(\alpha\right)}\\
 & \stackrel[0^{+}]{\alpha,\beta}{\longrightarrow}\frac{\alpha}{\alpha+\beta}\frac{1-\gamma\left(\alpha+\beta\right)}{1-\gamma\alpha}\stackrel[0^{+}]{\alpha,\beta}{\longrightarrow}\lambda,\label{eq:useContGamma-2}
\end{align}
where (\ref{eq:useContGamma-2}) uses the continuity of $\Gamma$.
Results (\ref{eq:firstBBinBern}) through (\ref{eq:useContGamma-2})
give the convergence in distribution. Obviously, $\nicefrac{1}{\left(\alpha+\beta+1\right)}\rightarrow1$
as $\alpha,\beta\rightarrow0^{+}$.
\item Fixing $0\leq k\leq n$, we have $\Pr\left(Y_{n}=k\right)$
\begin{align}
 & =\binom{n}{k}\frac{B\left(k+\alpha,n-k+\beta\right)}{B\left(\alpha,\beta\right)}\\
 & \sim\binom{n}{k}\frac{\left(k+\alpha\right)^{k+\alpha-\nicefrac{1}{2}}\left(n-k+\beta\right)^{n-k+\beta-\nicefrac{1}{2}}}{\left(n+\alpha+\beta\right)^{n+\alpha+\beta-\nicefrac{1}{2}}}\frac{\left(\alpha+\beta\right)^{\alpha+\beta-\nicefrac{1}{2}}}{\alpha^{\alpha-\nicefrac{1}{2}}\beta^{\beta-\nicefrac{1}{2}}}\label{eq:useStir}\\
 & =C\left(k,n,\alpha,\beta\right)\binom{n}{k}\left(\frac{\frac{\alpha}{\alpha+\beta}+\frac{k}{\alpha+\beta}}{1+\frac{n}{\alpha+\beta}}\right)^{k}\left(\frac{\frac{\beta}{\alpha+\beta}+\frac{n-k}{\alpha+\beta}}{1+\frac{n}{\alpha+\beta}}\right)^{n-k}\label{eq:useConst}\\
 & \stackrel{\alpha,\beta\rightarrow\infty}{\longrightarrow}\binom{n}{k}\lambda^{k}\left(1-\lambda\right)^{n-k},\label{eq:binRes}
\end{align}
where (\ref{eq:useStir}) uses Stirling's approximation and (\ref{eq:useConst})
and (\ref{eq:binRes}) use
\begin{equation}
C\left(k,n,\alpha,\beta\right)\coloneqq\left(\frac{\frac{\left(1+\frac{k}{\alpha}\right)^{\alpha}}{\sqrt{1+\frac{k}{\alpha}}}\frac{\left(1+\frac{n-k}{\beta}\right)^{\beta}}{\sqrt{1+\frac{n-k}{\beta}}}}{\frac{\left(1+\frac{n}{\alpha+\beta}\right)^{\alpha+\beta}}{\sqrt{1+\frac{n}{\alpha+\beta}}}}\right)\stackrel{\alpha,\beta\rightarrow\infty}{\longrightarrow}\frac{e^{k}e^{n-k}}{e^{n}}=1,
\end{equation}
which gives convergence in distribution. We finally turn to the asymptotic
independence of the $\xi_{i}$. For $2\leq k\leq n$, fix $\mathbf{b}\in\left\{ 0,1\right\} ^{k}$
and $s\coloneqq\sum_{j=1}^{k}b_{j}$. For $1\leq i_{1}<i_{2}<\cdots<i_{k}\leq n$,
let $\boldsymbol{\xi}\coloneqq\left(\xi_{i_{1}},\xi_{i_{2}},\ldots,\xi_{i_{k}}\right)$.
Then, 
\begin{align}
\Pr\left(\boldsymbol{\xi}=\mathbf{b}\right) & =\mathbb{E}\left[\Pr\left(\boldsymbol{\xi}=\mathbf{b}\left|X_{\alpha,\beta}\right.\right)\right]=\mathbb{E}\left[X_{\alpha,\beta}^{s}\left(1-X_{\alpha,\beta}\right)^{k-s}\right]\\
 & =\mathbb{E}\left[X_{\alpha,\beta}^{s}\sum_{j=0}^{k-s}\left(-1\right)^{k-s-j}X_{\alpha,\beta}^{k-s-j}\right]\label{eq:useBinomThm-1}\\
 & =\sum_{j=0}^{k-s}\left(-1\right)^{k-s-j}\mathbb{E}X_{\alpha,\beta}^{k-j}\\
 & =\sum_{j=0}^{k-s}\left(-1\right)^{k-s-j}\prod_{r=0}^{k-j-1}\frac{\alpha+r}{\alpha+\beta+r}\label{eq:useBetaMom}\\
 & \stackrel{\alpha,\beta\rightarrow\infty}{\longrightarrow}\sum_{j=0}^{k-s}\left(-1\right)^{k-s-j}\lambda^{k-j}=\lambda^{s}\left(1-\lambda\right)^{k-s},\label{eq:useBinomThm-2}
\end{align}
where (\ref{eq:useBinomThm-1}) and (\ref{eq:useBinomThm-2}) use
the binomial theorem and (\ref{eq:useBetaMom}) uses the well-known
expression for the $\left(k-j\right)$th raw moment of $\mathrm{Beta}\left(\alpha,\beta\right)$.
This gives the asymptotic independence of the $\xi_{i}$, completing
the proof.
\end{enumerate}
\end{proof}
\begin{lem}
\label{lem:cardCrossProd}For $\mathbf{k}\in\left\{ 1,2,\ldots\right\} ^{m}$
let 
\begin{equation}
\mathcal{S}\coloneqq\left\{ \mathbf{l}\in\bigtimes_{j=1}^{m}\left[k_{j}\right]:l_{1},l_{2},\ldots,l_{m}\textrm{ all distinct}\right\} ,
\end{equation}
where $\left[k\right]\coloneqq\left\{ 1,2,\ldots,k\right\} $ and
$\bigtimes_{j=1}^{m}\left[k_{j}\right]\coloneqq\left[k_{1}\right]\times\left[k_{2}\right]\times\cdots\times\left[k_{m}\right]$,
then 
\begin{equation}
\left|\mathcal{S}\right|=\prod_{j=1}^{m}\left(k_{\left(j\right)}-j+1\right),\label{eq:cardCrossProd}
\end{equation}
where $\left|\mathcal{S}\right|$ gives the number of vectors in \textup{$\mathcal{S}$
and} $k_{\left(1\right)}\leq k_{\left(2\right)}\leq\cdots\leq k_{\left(m\right)}$
gives the elements of $\mathbf{k}$ in a non-decreasing order.
\end{lem}

\begin{proof}
We argue constructively. Imagine a tree with levels $0,1,\ldots,m$.
Level 0 gives the root, level 1 its children, \emph{etc.} Level $1\leq j\leq m$
determines the value of $l_{\left(j\right)}$ for $\mathbf{l}\in\bigtimes_{j=1}^{m}\left[k_{j}\right]$,
where $\left(j\right)$ gives the original index of $k_{\left(j\right)}$.
The root has $k_{\left(1\right)}$ children. Each child of the root,
avoiding its parent's value, has $k_{\left(2\right)}-1$ children.
Each grandchild of the root, avoiding its parent's and grandparent's
values, has $k_{\left(3\right)}-2$ children, \emph{etc.} With $\mathcal{L}$
the set of vectors represented by the leaves, $\left|\mathcal{L}\right|$
appears on the right-hand side of (\ref{eq:cardCrossProd}). Each
leaf, with its unique path back to the root, gives a unique vector
in $\mathcal{S}$, so that $\mathcal{L}\subset\mathcal{S}$. To see
that $\mathcal{S}\subset\mathcal{L}$, note that 
\begin{equation}
\left[k_{\left(1\right)}\right]\subset\left[k_{\left(2\right)}\right]\subset\cdots\subset\left[k_{\left(m\right)}\right].
\end{equation}
That is, any vector in $\mathcal{S}$ first selects $l_{\left(1\right)}$,
then $l_{\left(2\right)}$, \ldots , then $l_{\left(m\right)}$, as
in the tree that constructs $\mathcal{L}$. This completes the proof.
\end{proof}
\RankAsymp*
\begin{proof}
We start with (\ref{eq:deltaNegInfRvec}). Fixing $1\leq i\leq n$
and $\epsilon\in\left(0,1\right]$ we note that
\begin{align}
\lim_{\delta\rightarrow-\infty}\Pr\left(\left|\mathbf{1}_{\left\{ X_{i}\leq X_{0}\right\} }-1\right|\geq\epsilon\right) & =\lim_{\delta\rightarrow-\infty}\Pr\left(\mathbf{1}_{\left\{ X_{i}\leq X_{0}\right\} }=0\right)\\
 & =\lim_{\delta\rightarrow-\infty}\Pr\left(X_{i}>X_{0}\right)\\
 & =\lim_{\delta\rightarrow-\infty}\Phi\left(\nicefrac{\delta}{\sqrt{\rho^{2}+1}}\right)=0,\label{eq:useEThm-1}
\end{align}
where (\ref{eq:useEThm-1}) uses Theorem \ref{thm:theGenMean}. Slutsky's
theorem then implies that $R_{0}\longrightarrow n+1$ in setting (\ref{eq:deltaNegInfRvec}).
Plugging $\lim_{\delta\rightarrow-\infty}\Pr\left(R_{0}=j_{0}\right)=\mathbf{1}_{\left\{ j_{0}=n+1\right\} }$
into (\ref{eq:noR0}) then gives
\begin{equation}
\lim_{\delta\rightarrow-\infty}\Pr\left(R_{i_{1}}=j_{1},\ldots,R_{i_{m}}=j_{m}\right)=\left\{ \begin{array}{rl}
0 & \textrm{ if }j_{k}=n+1\\
\frac{1}{\prod_{k=1}^{m}\left(n-k+1\right)} & \textrm{ otherwise,}
\end{array}\right.
\end{equation}
which gives (\ref{eq:deltaNegInfRvec}). A similar argument yields
(\ref{eq:deltaInfRvec}).

Slutsky's theorem, Proposition \ref{prop:limitab}, Proposition \ref{prop:betaBinomLim}.\ref{enu:opt1},
and Theorem \ref{thm:RONO} imply that $R_{0}\implies1+n\mathrm{Bernoulli}\left(\Phi\left(-\delta\right)\right)$
in setting (\ref{eq:rhoZeroRvec}). Fix $\mathbf{j}\in\mathcal{S}_{n+1,m+1}$
and let $\mathbf{j}_{-0}\coloneqq\left(j_{1},j_{2},\ldots,j_{m}\right)$,
$x_{0}\coloneqq\nicefrac{\left(j_{0}-1\right)}{n}$, and $\mathbf{x}_{0}\coloneqq x_{0}\mathbf{1}_{m}$.
Plugging the above result into (\ref{eq:withR0}) yields three cases:
\begin{enumerate}
\item If $j_{0}=1$, then $x_{0}=0$ and $\Pr\left(R_{0}=1,R_{i_{1}}=j_{1},\ldots,R_{i_{m}}=j_{m}\right)$
\begin{align}
 & =\frac{\Pr\left(R_{0}=1\right)}{\prod_{k=1}^{m}\left(n-k+1\right)}\stackrel[0^{+}]{\rho}{\longrightarrow}\frac{\Phi\left(\delta\right)}{\prod_{k=1}^{m}\left(n-k+1\right)}\\
 & =\Pr\left(\xi=x_{0}\right)\Pr\left(\mathbf{U}_{n,m}=\mathbf{j}_{-0}-\mathbf{1}_{m}+\mathbf{x}_{0}\right),
\end{align}
where we note that $\mathbf{j}_{-0}-\mathbf{1}_{m}+\mathbf{x}_{0}\in\mathcal{S}_{n,m}$
because $j_{0}=1$ and $\mathbf{x}_{0}=\mathbf{0}_{m}$.
\item If $2\leq j_{0}\leq n$, then $x_{0}\in\left(0,1\right)$ and $\Pr\left(R_{0}=j_{0},R_{i_{1}}=j_{1},\ldots,R_{i_{m}}=j_{m}\right)$
\begin{align}
 & =\frac{\Pr\left(R_{0}=j_{0}\right)}{\prod_{k=1}^{m}\left(n-k+1\right)}\stackrel[0^{+}]{\rho}{\longrightarrow}\frac{\Pr\left(1+n\xi=j_{0}\right)}{\prod_{k=1}^{m}\left(n-k+1\right)}\\
 & =0\times0=\Pr\left(\xi=x_{0}\right)\Pr\left(\mathbf{U}_{n,m}=\mathbf{j}_{-0}-\mathbf{1}_{m}+\mathbf{x}_{0}\right),\label{eq:lastStep}
\end{align}
where (\ref{eq:lastStep}) uses $x_{0}\in\left(0,1\right)$ and $\mathbf{j}_{-0}-\mathbf{1}_{m}+\mathbf{x}_{0}\notin\mathcal{S}_{n,m}$.
\item If $j_{0}=n+1$, then $x_{0}=1$ and $\Pr\left(R_{0}=n+1,R_{i_{1}}=j_{1},\ldots,R_{i_{m}}=j_{m}\right)$
\begin{align}
 & =\frac{\Pr\left(R_{0}=n+1\right)}{\prod_{k=1}^{m}\left(n-k+1\right)}\stackrel[0^{+}]{\rho}{\longrightarrow}\frac{\Phi\left(-\delta\right)}{\prod_{k=1}^{m}\left(n-k+1\right)}\\
 & =\Pr\left(\xi=x_{0}\right)\Pr\left(\mathbf{U}_{n,m}=\mathbf{j}_{-0}-\mathbf{1}_{m}+\mathbf{x}_{0}\right),
\end{align}
where we note that $\mathbf{j}_{-0}-\mathbf{1}_{m}+\mathbf{x}_{0}\in\mathcal{S}_{n,m}$
because $j_{0}=n+1$ and $\mathbf{x}_{0}=\mathbf{1}_{m}$.
\end{enumerate}
To review, the above exhaustive cases give $\Pr\left(R_{0}=j_{0},\ldots,R_{i_{m}}=j_{m}\right)$
\begin{equation}
\stackrel[0^{+}]{\rho}{\longrightarrow}\Pr\left(\xi=x_{0}\right)\Pr\left(\mathbf{U}_{n,m}=\mathbf{j}_{-0}-\mathbf{1}_{m}+\mathbf{x}_{0}\right),
\end{equation}
where $\xi\sim\mathrm{Bernoulli}\left(\Phi\left(-\delta\right)\right)$
and $\mathbf{U}_{n,m}\sim\mathrm{Uniform}\left(\mathcal{S}_{n,m}\right)$,
showing that $\xi$ and $\mathbf{U}_{n,m}$ are independent. Further,
$\left(R_{i_{1}},R_{i_{2}},\ldots,R_{i_{m}}\right)\implies\mathbf{1}_{m}-\boldsymbol{\xi}_{m}+\mathbf{U}_{n,m}$
as $\rho\rightarrow0^{+}$, where $\boldsymbol{\xi}_{m}\coloneqq\left(\xi,\xi,\ldots,\xi\right)\in\left\{ \mathbf{0}_{m},\mathbf{1}_{m}\right\} $,
giving (\ref{eq:rhoZeroRvec}).

Slutsky's theorem, Propositions \ref{prop:limitab} and \ref{prop:betaBinomLim}.\ref{enu:opt2},
and Theorem \ref{thm:RONO} imply that $R_{0}\implies1+\mathrm{Binomial}\left(n,\nicefrac{1}{2}\right)$
and $R_{0}\implies1+\mathrm{Binomial}\left(n,\Phi\left(-r\right)\right)$
in settings (\ref{eq:rhoInfRvec}) and (\ref{eq:rhoAbsDeltaInfRvec}).
Proposition \ref{prop:otherRanks} then gives (\ref{eq:rhoInfRvec})
and (\ref{eq:rhoAbsDeltaInfRvec}).

We first show that $\frac{1}{n}\left(R_{0}-1\right)=\frac{1}{n}\sum_{i=1}^{n}\mathbf{1}_{\left\{ X_{i}\leq X_{0}\right\} }\implies Z_{\rho,\delta}$
in setting (\ref{eq:nInfRvec}). In that the support of $Z_{\rho,\delta}$
is bounded, showing that 
\begin{equation}
\lim_{n\rightarrow\infty}\mathbb{E}\left[\left(\frac{1}{n}\sum_{i=1}^{n}\mathbf{1}_{\left\{ X_{i}\leq X_{0}\right\} }\right)^{k}\right]=\mathbb{E}\left[Z_{\rho,\delta}^{k}\right],\label{eq:R0toZGoal}
\end{equation}
for $k\geq1$, gives the result (\citet{B08} \S30). To that end, we
assume that $X_{0}\sim\mathcal{N}\left(0,1\right)$ and that the $\left\{ X_{i}\right\} _{i=1}^{k}$
are i.i.d.\ $\mathcal{N}\left(\delta,\rho^{2}\right)$ (see footnote
\ref{fn:one-case}). Then,
\begin{align}
\mathbb{E}\left[Z_{\rho,\delta}^{k}\right] & =\int_{0}^{1}z^{k}\frac{\rho\phi\left(\delta+\rho\Phi^{-1}\left(z\right)\right)}{\phi\left(\Phi^{-1}\left(z\right)\right)}dz=\int_{-\infty}^{\infty}\Phi\left(\frac{y-\delta}{\rho}\right)^{k}\phi\left(y\right)dy\\
 & =\mathbb{E}\left[\Pr\left(X_{1}\leq X_{0},X_{2}\leq X_{0},\ldots,X_{k}\leq X_{0}\left|X_{0}\right.\right)\right]\\
 & =\Pr\left(X_{1}\leq X_{0},X_{2}\leq X_{0},\ldots,X_{k}\leq X_{0}\right).\label{eq:EZk}
\end{align}
We further have that, as $n$ becomes large, $\mathbb{E}\left[\left(\frac{1}{n}\sum_{i=1}^{n}\mathbf{1}_{\left\{ X_{i}\leq X_{0}\right\} }\right)^{k}\right]$
\begin{align}
 & =\frac{1}{n^{k}}\sum_{j=1}^{k}\Pr\left(X_{1}\leq X_{0},X_{2}\leq X_{0},\ldots,X_{j}\leq X_{0}\right)\prod_{i=1}^{j}\left(n-i+1\right)\\
 & =\Pr\left(X_{1}\leq X_{0},X_{2}\leq X_{0},\ldots,X_{k}\leq X_{0}\right)\prod_{i=1}^{k}\left(1-\frac{i-1}{n}\right)+\mathcal{O}\left(\frac{1}{n}\right)\\
 & \stackrel[\infty]{n}{\longrightarrow}\Pr\left(X_{1}\leq X_{0},X_{2}\leq X_{0},\ldots,X_{k}\leq X_{0}\right).\label{eq:limAve}
\end{align}
That is, (\ref{eq:EZk}) and (\ref{eq:limAve}) give (\ref{eq:R0toZGoal}),
and so $\frac{1}{n}\left(R_{0}-1\right)\implies Z_{\rho,\delta}$
in setting (\ref{eq:nInfRvec}).

We turn to the asymptotic distribution of $\mathbf{R}_{\mathbf{i}}\coloneqq\left(R_{i_{1}},R_{i_{2}},\ldots,R_{i_{m}}\right)$.
First, fix $\mathbf{x}\in\left(0,1\right)^{m}$ and, for $1\leq j\leq m$,
let $k_{n,j}\coloneqq\left\lfloor x_{j}n+1\right\rfloor $, so that
$\nicefrac{k_{n,j}}{n}\rightarrow x_{j}$ as $n\rightarrow\infty$.
Then, for $n$ large, we have $\Pr\left(n^{-1}\left(\mathbf{R}_{\mathbf{i}}-\mathbf{1}_{m}\right)\leq\mathbf{x}\right)=\Pr\left(\mathbf{R}_{\mathbf{i}}\leq\mathbf{k}_{n}\right)$
\begin{align}
 & =\sum_{\mathbf{l}\in\mathcal{T}_{n,m}}\Pr\left(\mathbf{R}_{\mathbf{i}}=\mathbf{l}\right)=\sum_{\mathbf{l}\in\mathcal{T}_{n,m}}\frac{1-\sum_{j=1}^{m}\Pr\left(R_{0}=l_{j}\right)}{\prod_{j=1}^{m}\left(n-j+1\right)}\label{eq:useNoR0}\\
 & =\frac{\left|\mathcal{T}_{n,m}\right|}{\prod_{j=1}^{m}\left(n-j+1\right)}-\frac{\sum_{j=1}^{m}\sum_{\mathbf{l}\in\mathcal{T}_{n,m}}\Pr\left(R_{0}=l_{j}\right)}{\prod_{j=1}^{m}\left(n-j+1\right)}\\
 & =\left(\prod_{j=1}^{m}\frac{k_{n,j}+\mathcal{O}\left(1\right)}{n-j+1}\right)\left(1-\sum_{j=1}^{m}\frac{\sum_{l_{j}=1}^{k_{n,j}}\Pr\left(R_{0}=l_{j}\right)}{k_{n,j}+\mathcal{O}\left(1\right)}\right)\\
 & =\prod_{j=1}^{m}\frac{k_{n,j}+\mathcal{O}\left(1\right)}{n-j+1}+\mathcal{O}\left(\frac{1}{n}\right)\stackrel[\infty]{n}{\longrightarrow}\prod_{j=1}^{m}x_{j},\label{eq:limToUnif}
\end{align}
where the notation $\mathbf{a}\leq\mathbf{b}$ indicates that $a_{j}\leq b_{j}$
$\forall j$, and (\ref{eq:useNoR0}) uses (\ref{eq:noR0}) and 
\begin{equation}
\mathcal{T}_{n,m}\coloneqq\left\{ \mathbf{l}\in\bigtimes_{j=1}^{m}\left[k_{n,j}\right]:l_{1},l_{2},\ldots,l_{m}\textrm{ all distinct}\right\} ,\label{eq:defSnm}
\end{equation}
so that, for $k_{n,\left(1\right)}\leq k_{n,\left(2\right)}\leq\cdots\leq k_{n,\left(m\right)}$,
$\left|\mathcal{T}_{n,m}\right|=\prod_{j=1}^{m}\left(k_{n,\left(j\right)}-j+1\right)$
(see Lemma \ref{lem:cardCrossProd}). Result (\ref{eq:limToUnif})
gives $\frac{1}{n}\left(\mathbf{R}_{\mathbf{i}}-\mathbf{1}_{m}\right)\implies\boldsymbol{\Upsilon}_{m}$
in setting (\ref{eq:nInfRvec}). 

We finally turn to the asymptotic independence of $R_{0}$ and $\mathbf{R}_{\mathbf{i}}$.
As above, we fix $\mathbf{x}\in\left(0,1\right)^{m+1}$ and, for $0\leq j\leq m$,
let $k_{n,j}\coloneqq\left\lfloor x_{j}n+1\right\rfloor $, so that
$\nicefrac{k_{n,j}}{n}\rightarrow x_{j}$ as $n\rightarrow\infty$.
Then, for $n$ large, we have $\Pr\left\{ n^{-1}\left[\mathbf{R}_{0,\mathbf{i}}-\mathbf{1}_{m+1}\right]\leq\mathbf{x}\right\} $
\begin{align}
 & =\Pr\left\{ \mathbf{R}_{0,\mathbf{i}}\leq\mathbf{k}_{n}\right\} =\sum_{\mathbf{l}\in\mathcal{T}_{n,m+1}}\Pr\left\{ \mathbf{R}_{0,\mathbf{i}}=\mathbf{l}\right\} \label{eq:useDefSnm}\\
 & =\frac{1}{\prod_{j=1}^{m}\left(n-j+1\right)}\sum_{\mathbf{l}\in\mathcal{T}_{n,m+1}}\Pr\left(R_{0}=l_{0}\right)\label{eq:useWithR0}\\
 & =\frac{\prod_{j=0}^{m}\left(k_{n,j}+\mathcal{O}\left(1\right)\right)}{\left(k_{n,0}+\mathcal{O}\left(1\right)\right)\prod_{j=1}^{m}\left(n-j+1\right)}\sum_{l_{0}=1}^{k_{n,0}}\Pr\left(R_{0}=l_{0}\right)\label{eq:useSizeS}\\
 & =\left(\prod_{j=1}^{m}\frac{k_{n,j}+\mathcal{O}\left(1\right)}{n-j+1}\right)\Pr\left(\frac{R_{0}-1}{n}\leq x_{0}\right)\\
 & \stackrel[\infty]{n}{\longrightarrow}\Pr\left(\boldsymbol{\Upsilon}_{m}\leq\mathbf{x}_{-0}\right)\Pr\left(Z_{\rho,\delta}\leq x_{0}\right),\label{eq:useConvDef}
\end{align}
where (\ref{eq:useDefSnm}) and (\ref{eq:useSizeS}) use (\ref{eq:defSnm})
with $0\leq j\leq m$, (\ref{eq:useWithR0}) uses (\ref{eq:withR0}),
and (\ref{eq:useConvDef}) uses $\mathbf{x}_{-0}\coloneqq\left(x_{1},x_{2},\ldots,x_{m}\right)$
and $\frac{1}{n}\left(R_{0}-1\right)\implies Z_{\rho,\delta}$, as
$n\rightarrow\infty$. The factored form of (\ref{eq:useConvDef})
gives asymptotic independence, and so (\ref{eq:nInfRvec}).
\end{proof}

\end{document}